\documentclass[1p, 10pt]{elsarticle}

\usepackage{amsmath, amsthm, amssymb, amsfonts}

\newcommand{\su}{^}

\usepackage{subfigure}
\usepackage{graphicx,epsfig,color}
\usepackage{scrpage2}
\usepackage{exscale}
\usepackage{pstricks}

\theoremstyle{plain}

\newtheorem{theorem}{Theorem}[section]

\newtheorem{lemma}[theorem]{Lemma}

\newtheorem{corollary}[theorem]{Corollary}

\newtheorem{proposition}[theorem]{Proposition}

\theoremstyle{definition}

\newtheorem{definition}[theorem]{Definition}

\newtheorem{assumption}[theorem]{Assumption}

\theoremstyle{remark}
\newtheorem{remark}[theorem]{Remark}

\newtheorem{example}[theorem]{Example}

\newtheorem*{proof1}{Proof of Theorem \ref{oneone}}
\newtheorem*{proof2}{Proof of Theorem \ref{twotwo}}
\newtheorem*{proof3}{Proof of Corollary \ref{thirdthird}}
\newtheorem*{proof4}{Proof of Theorem \ref{fourfour}}
\newtheorem*{ak}{Acknowledgements}

\DeclareSymbolFont{AMSb}{U}{msb}{m}{n}
\DeclareMathSymbol{\N}{\mathalpha}{AMSb}{"4E}
\DeclareMathSymbol{\R}{\mathalpha}{AMSb}{"52}
\DeclareMathSymbol{\Z}{\mathalpha}{AMSb}{"5A}
\DeclareMathSymbol{\D}{\mathalpha}{AMSb}{"44}
\DeclareMathSymbol{\s}{\mathalpha}{AMSb}{"53}

\newcommand{\sF}{\scriptscriptstyle{F}}
\newcommand{\sC}{\scriptscriptstyle{C}}
\newcommand{\sB}{\scriptscriptstyle{B}}
\newcommand{\sX}{\scriptscriptstyle{X}}
\newcommand{\sM}{\scriptscriptstyle{M}}
\newcommand{\sN}{\scriptscriptstyle{N}}
\newcommand{\sK}{\scriptscriptstyle{K}}
\newcommand{\sY}{\scriptscriptstyle{Y}}
\DeclareMathOperator{\vol}{vol}
\DeclareMathOperator{\Ch}{Ch}
\DeclareMathOperator{\lip}{Lip}

\newcommand{\sk}{\sin_{\scriptscriptstyle{K}}}
\newcommand{\ck}{\cos_{\scriptscriptstyle{K}}}

\DeclareMathOperator{\spec}{spec}
\DeclareMathOperator{\Con}{Con}
\DeclareMathOperator{\supp}{supp}

\DeclareMathOperator{\de}{d}
\DeclareMathOperator{\m}{m}
\DeclareMathOperator{\ric}{ric}
\DeclareMathOperator{\diam}{diam}

\newcommand{\Ee}{\mathcal{E}^{\sF}}
\newcommand{\ChF}{\Ch^{\sF}}
\newcommand{\ChW}{\Ch^{\scriptscriptstyle{I_K\times_{N\!,\sk} F}}}
\newcommand{\ChC}{\Ch^{\scriptscriptstyle{\Con_{N\!,K}(F)}}}
\newcommand{\ChX}{\Ch^{\sX}}
\begin{document}
\title{\textsc{Cones over metric measure spaces and the maximal diameter theorem}}

\address[ck]{Institute of Applied Mathematics, Endenicher Allee 60, D-53115 Bonn}
\author[ck]{Christian Ketterer}
%\corref{cor1}\fnref{fn1}
\ead{ketterer@iam.uni-bonn.de}
\begin{abstract}
The main result of this article states that the $(K,N)$-cone over some metric measure space satisfies the reduced
Riemannian curvature-dimension condition $RCD^*(KN,N+1)$ if and only if the underlying space satisfies $RCD^*(N-1,N)$. 
The proof uses a characterization of reduced Riemannian curvature-dimension bounds by Bochner's inequality that was established for general
metric measure spaces by Erbar, Kuwada and Sturm in \cite{erbarkuwadasturm} (independently, the same result has been announced by Ambrosio, Mondino and Savar\'e).
As a corollary of our result and the Gigli-Cheeger-Gromoll splitting theorem \cite{giglisplittingshort} we obtain a maximal diameter theorem in the context of metric measure spaces that satisfy the condition $RCD^*$.\\
\end{abstract}
\maketitle

\tableofcontents
\section{Introduction}

The euclidean cone over a metric space is an important construction
in the theory of Riemannian manifolds and metric spaces. Cones appear
naturally in numerous situations, e.g. as model spaces or
as tangent cones. A special feature is the simplicity of their
definition.
The distance function of the euclidean cone over a metric space $(F,\de_{\sF})$ is induced by the following pseudo-metric:
\begin{align*}
\de_{\Con}((r,x),(s,y))=\sqrt{r^2+s^2-2rs\cos(\de_{\sF}(x,y)\wedge\pi)}
\end{align*}
for $(r,x)$ and $(s,y)$ in $[0,\infty)\times F$.
One of the fundamental results in the context of
Alexandrov spaces with curvature bounded from below is that
the euclidean cone has a curvature bound equal to 0 if and only if
the underlying space has a curvature bound equal to 1. This is a
generalization of the observation that the
cone over a circle of radius $\leq 1$ is flat away from the origin. 
If
one weakens the notion of curvature and considers the Ricci
curvature of a Riemannian manifold,
one could observe the following. A
simple computation for the curvature tensors shows that away from the
origin, where the cone is a smooth Riemannian manifold in the classical sense, the Ricci tensor is pointwise bounded from below by 0 if
and only if the Ricci tensor of the underlying manifold $F$ is pointwise bounded from below by $\dim F-1$. 
We want to investigate if there are generalizations of this result in a non-smooth context.

Lott/Villani \cite{lottvillani} and Sturm \cite{stugeo1,stugeo2} introduced a synthetic notion
of lower Ricci curvature bounds for metric measure spaces and initiated a program to study the analytic and geometric properties of these spaces. The idea is
to define curvature in terms of convexity of entropy functionals on
the $L^2$-Wasserstein space of absolutely continuous probability measures. Their approach yields the so-called
\textit{curvature-dimension condition} $CD(\kappa,N)$ for metric measure spaces. Meanwhile
 a huge number of results
concerning functional inequalities, eigenvalue estimates, regularity
and stability properties have been established for $CD$-spaces. However, this class is still
rather big since it includes also Finsler geometries. The curvature-dimension condition does not imply a linear heat semi group of the corresponding Cheeger energy. For this reason, 
a refined version of the curvature-dimension condition was introduced by Ambrosio, Gigli and Savar\'e in \cite{agsriemannian}. 
They propose so-called
\textit{Riemannian Ricci curvature bounds} that make the linearity of the heat flow part of the definition.
Surprisingly, this condition is again equivalent to another unified notion, the so-called \textit{evolution
variational inequality} that holds for gradient flow curves in the $L^2$-Wasserstein space.  
Later, this idea was extended to finite dimensions by Erbar, Kuwada and Sturm.

The main purpose of this article is to generalize the results
concerning cones over Riemannian manifolds to arbitrary
metric measure spaces that satisfy a \textit{reduced Riemannian curvature-dimension
condition}. We do not restrict ourselves to euclidean cones but
consider also spherical and elliptic cones that can be defined by the
unified notion of $K$-cones for $K\in \mathbb{R}$ (see Definition \ref{kncone}). 
We will prove our result in the setting of metric measures spaces. Therefore, we also introduce suitable measures. 
For the euclidean cone a measure is defined by $d\m^{\sN}(r,x)=r^{\sN}dr\otimes d\m_{\sF}(x),$
where $\m_{\sF}$ is the reference measure of the underlying metric measure space.
It mimics the Riemannian volume of a cone.
The additional parameter $N\geq 0$ corresponds to a dimension bound for $F$. In this way, we obtain the notion of $(K,N)$-cone. 
Our main results are
\begin{theorem}\label{oneone}
Let $(F,\de_{\sF},m_{\sF})$ be a metric measure space that satisfies
$RCD^*(N-1,N)$ for $N\geq 1$ and $\diam F\leq \pi$. Let $K\geq 0$. Then the $(K,N)$-cone
$\Con_{\sK}^{\sN}(F)$
satisfies $RCD^*(KN,N+1)$.
\end{theorem}
\begin{theorem}\label{twotwo}
Let $(F,\de_{\sF},\m_{\sF})$ be a metric measure space. Suppose the $(K,N)$-cone $\Con_{\sK}^{\sN}(F)$
satisfies $RCD^*(KN,N+1)$ for $K\in \mathbb{R}$ and $N\geq 0$. Then
\begin{itemize}
\item[(1)] if $N\geq 1$, $F$ satisfies $RCD^*(N-1,N)$ and $\diam F\leq \pi$,
\item[(2)] if $N\in [0,1)$, $F$ is a point, or, $N=0$ and $F$ consists of exactly two points with distance $\pi$.
\end{itemize}
\end{theorem}
\begin{corollary}
Let $(F,\de_{\sF},\m_{\sF})$ be a metric measure space. $K\geq 0$ and $N\geq 1$. Then $\Con_{\sK}^{\sN}(F)$ satisfies $RCD^*(KN,N+1)$ if and only if $F$ satisfies $RCD^*(N-1,N)$ and $\diam F\leq \pi$.
\end{corollary}
\paragraph{The maximal diameter theorem} 
A corollary of our main theorems and the recently established Gigli-Cheeger-Gromoll splitting theorem in the context of $RCD(0,N)$-spaces \cite{giglisplittingshort} is the following maximal diameter theorem for $RCD^*$-spaces.
\begin{theorem}\label{fourfour}
Let $(F,\de_{\sF},\m_{\sF})$ be a metric measure space that satisfies $RCD^*(N,N+1)$ for $N\geq 0$. If $N=0$, we assume that $\diam F\leq \pi$. Let $x,y$ be points in $F$ such that $\de_{\sF}(x,y)=\pi$. 
Then, there exists a metric measure space $(F',\de_{\sF'},\m_{\sF'})$ such that $(F,\de_{\sF},\m_{\sF})$ is isomorphic to $[0,\pi]\times_{\sin}^{\sN} F'$ and 
\begin{itemize}
\item[(1)] if $N\geq 1$, $(F',\de_{\sF'},\m_{\sF'})$ satisfies $RCD^*(N-1,N)$ and $\diam F'\leq \pi$,
\item[(2)] if $N\in [0,1)$, $F'$ is a point, or $N=0$ and $F$ consists of exactly two points with distance $\pi$.
\end{itemize}
\end{theorem}
We remind that a metric measure space that satisfies $RCD^*(N-1,N)$ 
always has bounded diameter by $\pi$ (see Theorem \ref{bonnet} and \cite{stugeo2,ohtmea,erbarkuwadasturm}).
In the context of Riemannian manifolds Theorem \ref{fourfour} was proven by Cheng. It states that an $n$-dimensional Riemannian manifold that has
Ricci curvature bounded from below by $n-1$ and attains its
maximal diameter, is the standard sphere $\mathbb{S}^n$. We remark that a spherical suspension that is a smooth Riemannian manifold without boundary is a sphere. 
A result of Anderson \cite{anderson} shows that 
%small perturbations of diameter change the homeomorphism type of the manifold.  Namely, 
for any even dimension $n\geq 4$ and any $\epsilon>0$ one can find a
Riemannian manifold $M_{\epsilon}$ that satisfies a Ricci bound of $n-1$ and contains points $x,y\in M_{\epsilon}$ with $\de_{M_{\epsilon}}(x,y)=\pi-\epsilon=\diam M_{\epsilon}$ 
but which is not homeomorph to a sphere. 
In \cite{almostrigidity} Cheeger and Colding prove that any $n$-dimensional Riemannian manifold with Ricci curvature bounded from below by $n-1$ and almost maximal 
diameter is close in the Gromov-Hausdorff distance to a spherical 
suspension over some geodesic metric space.
Especially, Cheeger-Colding obtain the following result for Ricci limit spaces.
\begin{theorem}[Cheeger-Colding]
If a metric space $(X,\de_{\sX})$ is the Gromov-Hausdorff limit of a sequence of $n$-dimensional Riemannian manifolds $M_i$ with $\ric_{M_i}\geq n-1$ and there are points $x,y\in X$ such that 
$\de_{\sX}(x,y)=\pi$, then there exists a length space $(Y,\de_{\sY})$ with $\diam Y\leq \pi$ such that $[0,\pi]\times_{\sin}Y=X$.
\end{theorem}
%And in the setting of metric measure spaces the result of Bacher and Sturm \cite{bastco} shows that one can construct spherical suspensions which have maximal diameter 
%and satisfy the condition
%$CD(N-1,N)$ but which are not homeomorphic to a sphere. 
In particular, our result is sharp, improves the theorem of Cheeger and Colding by giving
additional information on the underlying space $Y$, and provides an alternative proof for Cheng's theorem.
As corollary of Theorem \ref{fourfour} we also obtain:
\begin{corollary}\label{thirdthird}
Let $(F,\de_{\sF},\m_{\sF})$ be a metric measure space that satisfies $RCD^*(N-1,N)$ for $N\geq 0$. 
Assume there are points $x_i,y_i\in F$ for $i=1,\dots, n$ with $n>N$ such that $\de_{\sF}(x_i,y_i)=\pi$ for any $i$ and $\de_{\sF}(x_i,x_j)=\frac{\pi}{2}$ for $i\neq j$.
Then $N=n-1\in\mathbb{N}$ and $(F,\de_{\sF},\m_{\sF})=\mathbb{S}^{N}$.
\end{corollary}
\paragraph{Outline of the proof} 
Let us briefly sketch the main ideas for the proof of Theorem \ref{oneone}.
A first step in direction of the result was done by Bacher and
Sturm in \cite{bastco}. They prove Theorem \ref{oneone} when the underlying
metric measure space is a smooth Riemannian manifold. Later, the author
adapted their ideas in \cite{ketterer} to extend the result to
warped products that is a generalization of the concept of metric cone. In some
sense the proof in both cases follows the Lagrangian interpretation
of curvature-dimension bounds that comes from the theory of optimal
transport. One establishes the convexity of the entropy functional
along Wasserstein geodesics from bounds
for the Ricci tensor. The main problem is to deal
with the set of singularity points (in the case of a cone this is
just the origin) where the underlying space differs from
an ordinary euclidean product. It turns out that the
curvature-dimension bound for $F$ guarantees that the optimal transport of
absolutely continuous measures does not not see these singularities and consequently, they do not affect the convexity of the entropy. 
Now, one is tempted to prove the theorem for general metric measure spaces with the same strategy by
deducing the convexity of the entropy directly from the convexity of the entropy of 
the underlying space $F$. The statement that singularities can be
neglected holds as well in the general framework. However, as simple
the definition of the cone metric might be, the relation of optimal
transport in the cone and optimal transport in the underlying space is rather complicated as can be seen from easy examples. Hence, we need to follow another strategy.

By the work of Ambrosio, Gigli and Savar\'e \cite{agsriemannian} we know that spaces that satisfy a Riemannian curvature-dimension condition are directly linked to the theory of Dirichlet forms. 
Especially, it is the right setting to prove a version of Bochner's inequality. In the context of a Riemannian manifolds $M$ this is
\begin{align}\label{mimic}
\Gamma_2(u)=\textstyle{\frac{1}{2}}\Delta^{\sM} |\nabla u|^2-\langle\nabla u, \nabla\Delta^{\sM} u\rangle
%=\ric_{\sM}(\nabla u)+\left\|\nabla^2u\right\|_{HS}^2
\geq \kappa|\nabla u|^2+\textstyle{\frac{1}{N}}(\Delta^{\sM} u)^2
\end{align} 
for $u\in C^{\infty}(M)$ if the Ricci tensor of $M$ is bounded from below by $\kappa\in \mathbb{R}$ and the dimension is bounded from above by $N\in[1,\infty)$. 
(\ref{mimic}) also characterizes a Ricci curvature and dimension bound of $M$. 
Recently, (\ref{mimic}) was also established for general metric measure spaces in a series of publications by several authors. 
In a first step Gigli, Kuwada and Ohta proved the Bochner inequality with $N=\infty$ on finite dimensional Alexandrov spaces with curvature bounded from below \cite{giglikuwadaohta}.
Then, Ambrosio, Gigli and Savar\'e proved  
the result for arbitrary $RCD(\kappa,\infty)$-spaces \cite{agsbakryemery}, and finally 
Erbar, Kuwada and Sturm \cite{erbarkuwadasturm} (independently announced by Ambrosio, Mondino and Savar\'e \cite{amsbochner}) solved the problem for $RCD^*(\kappa,N)$-spaces with finite $N$. 
Even more, they prove the full equivalence between the 
reduced Riemannian curvature-dimension condition ($RCD^*(\kappa,N)$) and (\ref{mimic}) for any parameters $\kappa\in\mathbb{R}$ and $N\in [1,\infty)$.
% in the context of metric measure spaces as it is true for Riemannian manifolds.

Bochner's inequality captures the Eulerian picture of curvature-dimension bounds. This viewpoint was already
used by Bakry, Emery and Ledoux
 in the setting of diffusion semi groups. They established a so-called Bakry-Emery curvature-dimension condition that exactly mimics inequality (\ref{mimic}). 
They were able to deduce many results from Riemannian
geometry like Sobolev inequalities, Li-Yau gradient estimates or Bonnet-Myers Theorem in a purely
abstract setting only relying on the existence of a nice algebra of functions.
In fact, we will vaguely follow their strategy using the recent result on the characterization of
Riemannian curvature-dimension bounds to prove our theorem. 

Why is this a reasonable strategy?
We remind the reader of the situation of smooth Riemannian manifolds. 
Let us consider general warped products. $(B^{d},g_{\sB})$ and $(F^{n},g_{\sF})$ are 
Riemannian manifolds with Ricci curvature bounded from below by $(d-1)K$ and $(n-1)K_{\sF}$ respectively and $f:B\rightarrow (0,\infty)$ 
is a smooth function such that 
\begin{align}\label{bionade}
\nabla^2f(v,v)\leq -K|v|_{\sB}^2\hspace{5pt}\mbox{ for }v\in TB\hspace{5pt}\mbox{ and }\hspace{5pt}|\nabla f|^2_{\sB}\leq K_{\sF}-Kf^2\hspace{5pt}\mbox{on }B.
\end{align}
We can define the Riemannian warped product $B\times_f F=(B\times F,g_{B\times_fF})$ where 
$g_{B\times_f F}=\pi^*_{\sB}g_{\sB}+(\pi_{\sB}^*f)^2\pi^*_{\sF}g_{\sF}$
is a Riemannian metric on $B\times F$. For example, in the case of the euclidean cone we would choose $B=(0,\infty)$ and $f(r)=r$ and $K_{\sF}=1$ and $K=0$.
The Ricci tensor can be calculated explicitly at any point $(p,x)\in B\times F$ and is given by
\begin{align}\label{equation}
\ric_{B\times_f F}(\tilde{X}_{(p,x)}+\tilde{V}_{(p,x)})=&\ric_{\sB}(X_p)-n\textstyle{\frac{\nabla^2f(X_p)}{{f(p)}}}\nonumber\\
&+\ric_{\sF}(V_x)-\left(\textstyle{\frac{\Delta^{\sB} f(p)}{f(p)}}+(N-1)\textstyle{\frac{|\nabla f_p|^2}{f^2(p)}}\right)|\tilde{V}_{(p,x)}|^2
\end{align}
where $\tilde{X}$ and $\tilde{V}$ are horizontal and vertical lifts on $B\times_f F$ of vector fields $X$ and $V$ on $B$ and $F$ respectively. 
For the precise definitions we refer to \cite{on}. 
From this formula one can easily deduce our main theorem by applying the curvature properties of $B$ and $F$ and the assumption (\ref{bionade}) for $f$. 
Since $f>0$, we can forget the issue of singularities for a moment. 
We can see that calculations can be done pointwise and almost only take place on $B$. 
Now, we interpret Bochner's inequality (\ref{mimic}) as the functional analog of the corresponding bound for the Ricci tensor. 
Then, we establish a 
similar formula as (\ref{equation}) by just using the smooth structure of $B$. In section 3.2 we obtain the following pointwise inequality
\begin{align}
 \Gamma^{\sB\times^{\sN}_{f}\sF}_2(u)(p,x)&\geq \Gamma_2^{\sB}(u^x)(p)+\textstyle{\frac{1}{f^2(p)}}\Delta^{\sF}u^p\textstyle{\frac{\langle\nabla f,\nabla u^x\rangle_p}{f(p)}}&\nonumber\\
&+\textstyle{\frac{1}{f^4(p)}}\Gamma_2^{\sF}(u^p)(x)-\left(\textstyle{\frac{\Delta^{\sB} f(p)}{f(p)}}+(N-1)\textstyle{\frac{|\nabla f_p|^2}{f^2(p)}}\right)\textstyle{\frac{1}{f^2(p)}}|\nabla u^p_x|^2\hspace{6pt}\m_{\sC}\mbox{-a.e.}
\end{align}
for any $u\in C^{\infty}_0(\hat{B})\otimes \mathcal{A}^{\sF}$ where $\mathcal{A}^{\sF}$ is a algebra of functions on $F$ in the sense of Bakry-Emery calculus where 
Bochner's inequality holds with parameter $(N-1)K_{\sF}$ and $N$.
Now, if we use again the curvature and concavity conditions for $B$, $F$ and $f$ we obtain a sharp Bochner inequality for $u\in C^{\infty}_0(\hat{B})\otimes \mathcal{A}^{\sF}$.
This result holds in general for any warped product (Theorem \ref{maintheorem}).
At this point, we have proven a sharp $\Gamma_2$-estimate for any $N$-warped product provided $B$, $F$ and $f$ satisfy the assumptions of Theorem \ref{maintheorem}.
But we do not if it yields the full Bakry-Emery curvature-dimension since
we do not know if $C^{\infty}_0(\hat{B})\otimes \mathcal{A}^{\sF}$ is a dense subset in $D_2(L^{\sC})$. Indeed, in general, it is false that $\mathcal{E}^{\sC}$ satisfies a curvature-dimension condition even when a $\Gamma_2$-estimate holds on a big class of functions. For example, consider $F=S^1_{\scriptscriptstyle{2\pi}}$ that is a 
$1$-dimensional sphere of diameter $2\pi$ and and consider the corresponding $1$-cone $[0,\infty)\times_r^{\sN}\mathcal{E}^{S^1_{\scriptscriptstyle{2\pi}}}$\ .  
In \cite{bastco}
Bacher and Sturm prove that this cone does not satisfy any curvature-dimension condition. 
This can be seen from the behavior of optimal transport
since the cone in the case of a big circle is a kind of covering and if mass is transported from on sheet 
to another, the ''cheapest'' way to do it is to transport all mass through the origin which destroys any convexity of the entropy. 
Bacher and Sturm observed that this situation can be avoided if and only if the diameter of the underlying space is smaller than $\pi$. 
But on the other side, from our result one can see that the $\Gamma_2$ estimate holds without any restriction. 

Another observation is related to this problem. It is known (see
\cite[Appendix to Section X.I, Example 4]{reedsimon}) that the
Laplace operator that acts on smooth functions with compact support
in $\mathbb{R}^{N+1}\backslash \left\{0\right\}$ is essentially
self-adjoint if and only if $N\geq 3$ where $N\in\mathbb{N}$. But
this situation exactly corresponds to the case of an euclidean cone
over $\mathbb{S}^N$ with admissible algebra
$\mathcal{A}^{\sF}=C^{\infty}(\mathbb{S}^N)$. So in this case in
general the operator $L^{\sC}$ restricted to
$C^{\infty}_0((0,\infty))\otimes C^{\infty}(\mathbb{S}^N)$ will
provide more than one self-adjoint extension and the Friedrich's
extension does not need to coincide with the closure of
$C^{\infty}_0((0,\infty))\otimes C^{\infty}(\mathbb{S}^N)$ with respect to the graph norm. 
So, we cannot hope that $C^{\infty}_0(\hat{B})\otimes \mathcal{A}^{\sF}$ will be dense in the domain of $L^{\sC}$ in general. 
But we will see that in the Eulerian picture that is described by the $\Gamma_2$-estimate, the crucial quantity is not the diameter
but the first positive eigenvalue of $L^{\sF}$. For metric measure spaces that satisfy $RCD^*(N-1,N)$ there is a spectral gap $\lambda_1\geq N$. This 
fact together with results from the theory of $1$-dimensional essentially self-adjoint operators allows to prove the density of an admissible class of function in the domain of
$L^C$ in the case of $(K,N)$-cones. Additionally, we obtain a complete picture about how the spectral gap of $L^{\sF}$ enters the proof, and this should be seen in comparison to 
the Lagrangian viewpoint of Bacher and Sturm.

Hence, we can establish a Bakry-Emery condition for cones.
Then, we can use again the equivalence with the $RCD^*$ condition to prove Theorem \ref{oneone}. The final technical problem at this point is to prove that the intrinsic distance of cones in the sense of Dirichlet forms (see section 3.1) 
is the corresponding cone metric over the space we started with. 

\paragraph{Plan of the paper}
The paper is roughly divided into two parts. In section 2 and 3 we just consider strongly local and regular Dirichlet forms. Section 2 is an introduction into all necessary notions 
and definitions concerning this subject.
We briefly define symmetric Dirichlet forms in section 2.1 and introduce relevant properties. In section 2.2 we introduce the Bakry-Emery curvature dimension condition in the classical sense 
and in the sense that was proposed by Ambrosio, Gigli and Savar\'e in \cite{agsbakryemery}.
In section 2.3 we give important smooth examples of Dirichlet forms that satisfy the Bakry-Emery curvature dimension condition and that will be use later.
Section 3 is concerned with the proof of the sharp Bakry-Emery condition for cones over Dirichlet forms. In section 3.1, we first introduce skew products that is a well-known construction
recipe for Dirichlet forms which was introduced by Fukushima and Oshima. We modify this notion slightly and obtain so-called warped products and cones. 
In section 3.2 we prove sharp $\Gamma_2$-estimates for warped products on a suitable class of function. 
In section 3.3 we prove that in the special case of cones this class is dense in the domain of the corresponding self-adjoint operator. Finally, we obtain the full Bakry-Emery curvature-dimension 
condition for cones in 
section 3.5.

Section 4 and 5 constitute the second part of the article that is concerned with metric measure spaces that satisfy a 
Riemannian curvature-dimension condition in the sense of Lott-Villani-Sturm-Ambrosio-Gigli-Savar\'e. In section 4 we give a rough overview on classical results and recent developments in the field. 
We introduce the curvature-dimension condition as proposed by Sturm in \cite{stugeo2} in section 4.1. In section 4.2 we present important results on Poincar\'e and Sobolev inequalities on metric
measure spaces that are mainly due to Hajlasz and
Koskela \cite{koskela} and the first order calculus for metric measure spaces that was developed by Ambrosio, Gigli and Savar\'e. This allows us to give the definition 
of Riemannian curvature bounds in section 4.3.
Finally, in section 5 we prove our main theorems.  In section 5.1 we repeat the notion of metric warped product and cones and we show in section 5.2 that it is consistent with the notion that was developed for 
Dirichlet forms, at least, if we assume curvature-dimension bounds. In section 5.3 we prove the cone theorem and 
in section 5.4 we prove
the maximal diameter theorem.
\begin{remark}
This article differs slightly from a previous version with the same name including more detailed proofs.
\end{remark}

\section{Preliminaries on Dirichlet forms}
\subsection{Dirichlet forms and their $\Gamma$-operator}\label{nonsymmetriccase}
We consider a locally compact and separable Hausdorff space $(X, \mathcal{O}_{\sX})$ and a positive, $\sigma$-finite Radon measure $\m_{\sX}$ on $X$ such that $\supp \left[m_{\sX}\right]=X$. 
%Then we call the measure space $(X,\mathcal{O}_{\sX},\m_{\sX})$ admissible.
We denote by $L^p(X,\m_{\sX})=:L^p(\m_{\sX})$ for $p\in [0,\infty]$ the Lebesgue spaces with respect $\m_{\sX}$.
Let $(\mathcal{E}^{\sX},D(\mathcal{E}^{\sX}))$ be a symmetric Dirichlet form on $L^2(m_{\sX})$ where $D(\mathcal{E}^{\sX})$ is a dense subset of $L^2(\m_{\sX})$. 
A symmetric Dirichlet form is a $L^2(X,\m_{\sX})$-lower semi-continuous, quadratic form that satisfies the Markov property.
Dirichlet forms are closed, 
i.e. the domain $D(\mathcal{E}^{\sX})$ is a Hilbert space with respect to the energy norm that comes from the inner product
$$
(u,u)_{D(\mathcal{E})}=(u,u)_{L^2(\m_{\sX})}+\mathcal{E}^{\sX}(u,u).
$$
There is a self-adjoint, negative-definite operator $\left(L^{\sX}, D_2(L^{\sX})\right)$ on $L^2(\m_{\sX})$. 
Its domain is 
$$D_2(L^{\sX})=\left\{u\in D(\mathcal{E}^{\sX})\!:\exists v\in L^2(X,m_{\sX}):-(v,w)_{L^2(\m_{\sX})}=\mathcal{E}^{\sX}(u,w)\ \forall w\in D(\mathcal{E}^{\sX})\right\}.$$ 
We set $v=:L^{\sX}u$.
$D_2(L^{\sX})$ is dense in $L^2(X,\m_{\sX})$ and equipped with the topology given by the graph norm.
$L^{\sX}$ induces a strongly continuous Markov semi-group $(P^{\sX}_t)_{t\geq 0}$ on $L^2(X,m_{\sX})$. 
The relation between form, operator and semi-group is standard (see \cite{fukushima}).
\smallskip
\\
A Dirichlet form is called regular if $\mathcal{E}^{\sX}$ possesses a core.
A core of $\mathcal{E}^{\sX}$ is by definition a subset $\mathcal{C}^{\sX}$ of $D(\mathcal{E}^{\sX})\cap C_0(X)$ such that $\mathcal{C}^{\sX}$ is dense in 
$D(\mathcal{E}^{\sX})$ with respect to the energy norm and dense in $C_0(X)$ with respect to uniform convergence
where ${C}_0(X)$ is the set of continuous functions with compact support in $X$. 
We say that a symmetric form is strongly local if 
$\mathcal{E}^{\sX}(u,v)=0$ whenever $u,v\in D(\mathcal{E}^{\sX})$ and $(u+a)v=0$ $\m_{\sX}$-almost surely in $X$ for some $a\in\mathbb{R}$.
\begin{definition}[$\Gamma$-operator for Dirichlet forms]\label{definitionGamma}
Set $D^{\infty}(\mathcal{E}^{\sX})=D(\mathcal{E}^{\sX})\cap L^{\infty}(X,m_{\sX})$. Then $D^{\infty}(\mathcal{E}^{\sX})$ is an algebra (see \cite{hirsch}) and for 
$u,\phi\in D^{\infty}(\mathcal{E}^{\sX})$ the following operator is well-defined
\begin{equation*}
\Gamma^{\sX}(u;\phi):=\mathcal{E}^{\sX}(u,u\phi)-\frac{1}{2}\mathcal{E}^{\sX}(u^2,\phi).
\end{equation*}
It can be extended by continuity to any $u\in D(\mathcal{E}^{\sX})$. We call $\mathbb{G}$ the set of functions $u\in D(\mathcal{E}^{\sX})$ such that the linear form
$\phi\mapsto \Gamma^{\sX}(u;\phi)$ can be represented by an absolutely continuous measure w.r.t $\m_{\sX}$ with density $\Gamma^{\sX}(u)\in L^1_+(X,\m_{\sX})$. If $\mathcal{E}^{\sX}$ is symmetric, we get 
the following representation
\begin{equation}\label{symmetric}
\mathcal{E}^{\sX}(u,u)=\int_{\sX}\Gamma^{\sX}(u)d\m\hspace{5pt}\mbox{ for any }u\in\mathbb{G}.
\end{equation}
By polarization we can extend the $\Gamma$-operator as trilinear form as follows
\begin{align*}
\Gamma^{\sX}(u,v;\phi)=\frac{1}{2}\left(\Gamma^{\sX}(u;\phi)+\Gamma^{\sX}(v;\phi)-\Gamma^{\sX}(u-v;\phi)\right)\hspace{2pt}\mbox{for }u,v\in D(\mathcal{E}^{\sX}),\hspace{2pt}\phi\in D^{\infty}(\mathcal{E}^{\sX})
\end{align*}
If $\mathbb{G}=D(\mathcal{E}^{\sX})$, we say $\mathcal{E}^{\sX}$ admits a ``carr\'e du champ'' or $\Gamma$-operator. 
Fundamental properties of $\Gamma^{\sX}: D(\mathcal{E}^{\sX})\times D(\mathcal{E}^{\sX})\rightarrow L^1(X,\m_{\sX})$ are positivity, 
symmetry, bilinearity and continuity (see \cite[Proposition 4.1.3]{hirsch}).
\paragraph{Leibniz rule}
The strong locality of $\mathcal{E}^{\sX}$
implies the strong locality of $\Gamma^{\sX}$:
$
1_U\cdot \Gamma^{\sX}(u,v)=0
$
for all $u,v\in \mathbb{G}$ and for all open sets $U$ on which $u$ is constant(see \cite[Appendix]{sturmdirichlet1}) 
and the Leibniz rule: For all $u,v,w\in \mathbb{G}$ such that $v,w\in L^{\infty}(\m_{\sX})$ it holds $v\cdot w\in \mathbb{G}$ and
\begin{align}\label{leibnizkeks}
\Gamma^{\sX}(u,v\cdot w)=\Gamma^{\sX}(u,v)\cdot w+v\cdot \Gamma^{\sX}(u,w)
\end{align}
(see \cite[Appendix]{sturmdirichlet1}). One can prove the following
\begin{lemma} Assume $\mathbb{G}=D(\mathcal{E}^{\sX})$. (\ref{leibnizkeks}) also holds
for $u,v,w\in D(\mathcal{E}^{\sX})$ with $v,\Gamma^{\sX}(v)\in L^{\infty}(\m_{\sX})$.
\end{lemma}
%\begin{proof}
%Consider $w_n=(w\wedge -n)\vee n\in D(\mathcal{E}^{\sX})\cap L^{\infty}(\m_{\sX})$ for $n\in \mathbb{N}$. $w_n$ is bounded and by Theorem 1.4.2 (iii) in \cite{fukushima} 
%$(w_n)_n$ converges to $w$ in $D(\mathcal{E}^{\sX})$. 
%We can apply (\ref{leibnizkeks}) for $v,w_n-w_m\in D(\mathcal{E}^{\sX})\cap L^{\infty}$:
%\begin{align*}
%\mathcal{E}^{\sX}(v(w_n-w_m))&=\int_{\sX}\!\!\big(\Gamma^{\sX}(w_n-w_m)v^2\\
%&\hspace{40pt}+2v(w_n-w_m)\Gamma^{\sX}(v,w_n-w_m)+(w_n-w_m)^2\Gamma^{\sX}(v)\big)d\m_{\sX}
%\end{align*}
%Since $v,\Gamma^{\sX}(v)\in L^{\infty}(\m_{\sX})$, the right hand side converges to $0$ if $n,m\rightarrow\infty$, hence, $v\cdot w_n$ is a Cauchy sequence in $D(\mathcal{E}^{\sX})$ that converges to $g=v\cdot w$. Thus, $v\cdot w\in D(\mathcal{E}^{\sX})$ 
%and similar one can show the formula (\ref{leibnizkeks}).
%\end{proof}
\paragraph{Chain rule}
We say $\mathcal{E}^{\sX}$ is of diffusion type if $L^{\sX}$ satisfies the following chain rule. Let $\eta$ be in $C^2(\mathbb{R})$ with $\eta(0)=0$. 
If $u\in D_2(L^{\sX})$ with $\Gamma(u)\in L^2(X,m_{\sX})$ and $\eta(u)\in D(L^{\sX})$, then
\begin{eqnarray}\label{diffusion}
L^{\sX}\eta(u)=\eta'(u)L^{\sX}u+\eta''(u)\Gamma^{\sX}(u).
\end{eqnarray}
This is the case when $\mathbb{G}=D(\mathcal{E}^{\sX})$ (see \cite[Corollary 6.1.4]{hirsch}).
\\
\\
If $\mathcal{E}^{\sX}$ is strongly local and admits a ``carr\'e du champ'' operator, we can define $D_{loc}(\mathcal{E}^{\sX})$ as follows. $u\in D_{loc}(\mathcal{E}^{\sX})$ if $u\in L^2_{loc}(\m_{\sX})$ and for any compact set $K$ there exists 
$v\in D(\mathcal{E}^{\sX})$ such that $v=u$ $\m_{\sX}$-a.e. on $K$. Hence, for any $u\in D_{loc}(\mathcal{E}^{\sX})$ there exists $\Gamma^{\sX}(u)\in L^1_{loc}(\m_{\sX})$.
The intrinsic distance of $\mathcal{E}^{\sX}$ is defined by 
\begin{align*}
\de_{\mathcal{E}^{\sX}}(x,y)=\sup\left\{u(x)-(y): u\in D_{loc}(\mathcal{E}^{\sX})\cap C(X),\hspace{4pt}\Gamma^{\sX}(u)\leq 1 \hspace{4pt}\m\mbox{-a.e.}\right\}.
\end{align*}
The intrinsic distance is not metric in general but a pseudo-metric since there can be points $x\neq y$ with $\de_{\mathcal{E}^{\sX}}(x,y)=0$.
For the rest of this article we always assume that $\mathcal{E}^{\sX}$ is a strongly local and regular Dirichlet form with $\mathbb{G}=D(\mathcal{E}^{\sX})$. 
Then we will call $\mathcal{E}^{\sX}$ also admissible.
\end{definition}
\begin{definition}\label{doub} Let $\mathcal{E}^{\sX}$ be an admissible Dirichlet form.
\begin{itemize}
 \item[(i)] We say the Dirichlet form $\mathcal{E}^{\sX}$ is strongly regular if the topology of $\de_{\mathcal{E}^{\sX}}$ coincides with the original one.
 \item[(ii)]
We say that $(X,\de_{\mathcal{E}^{\sX}},\m_{\sX})$ satisfies the doubling property if there is constant $2^n=C_0$ for some $n\geq 0$ such that for all $x\in X$ and $0<r<\diam (X,\de_{\sX})$
\begin{align*}
\m_{\sX}(B_{2r}(x))\leq C_0\m_{\sX}(B_r(x)).
\end{align*}
 \item[(iii)]
We say that $\mathcal{E}^{\sX}$ supports a weak local $(q,p)$-Poincar\'e inequality with $1\leq p \leq q<\infty$ if there exist constants $C>0$ and $\lambda\geq 1$ such that for all $u\in D(\mathcal{E}^{\sX})$, any point
$x\in X$ and $r>0$
\begin{align}\label{ichwerdnarrisch}
\bigg(\int_{B_r(x)}|u-u_{B_r(x)}|^q d\m_{\sX}\bigg)^{\frac{1}{q}}\leq Cr\bigg(\int_{B_{\lambda r}(x)}\Gamma^{\sX}(u)^{\frac{p}{2}}d\m_{\sX}\bigg)^{\frac{1}{p}}.
\end{align}
If $\lambda=1$, we say $\mathcal{E}^{\sX}$ supports a strong $(1,p)$-Poincar\'e inequality.
Some authors also use the term Poincar\'e-Sobolev inequality for the case $q>1$ and $(1,p)$-Poincar\'e inequalities are just called $p$-Poincar\'e inequality (see for example \cite{koskela}).

\end{itemize}
\end{definition}
\begin{remark}\label{remarkpoincare}
\begin{itemize}
\item[(i)] It is known that under the doubling property for $(X,\de_{\mathcal{E}^{\sX}},\m_{\sX})$ weak local Poincar\'e inequalities imply strong ones.
\item[(ii)] By H\"older's inequality, a weak local $(1,p)$-Poincar\'e inequality implies a weak local $(1,p')$-Poincar\'e inequality for $p'\geq p$.
\end{itemize}\end{remark}
\begin{remark}\label{feller}
A doubling property and Poincar\'e inequalities are regularity assumptions for Dirichlet forms that were used by Sturm in \cite{sturmdirichlet1,sturmdirichlet2,sturmdirichlet3} to derive
the following results:\smallskip
\\
Let $\mathcal{E}^{\sX}$ a strongly local and strongly regular Dirichlet form and let $\de_{\mathcal{E}^{\sX}}$ be its intrinsic distance. 
Assume that closed balls $\bar{B}_r(x)$ are compact for any $r>0$ and $x\in X$.
Assume a doubling property holds and $\mathcal{E}^{\sX}$ supports a weak local $(2,2)$-Poincar\'e inequality. Then
\begin{itemize}
 \item[(1)] $P_t^{\sX}$ admits an $\alpha$-continuous kernel and is a Feller semi-group.
 \item[(2)] $P_t^{\sX}$ is $L^2\rightarrow L^{\infty}$-ultracontractive:
$
 \left\|P_t^{\sX}\right\|_{L^2\rightarrow L^{\infty}}\leq 1.
$
 \item[(3)] If $\m(X)<\infty$, harmonic functions are constant.
\end{itemize}
$L^2\rightarrow L^{\infty}$-ultracontractivity actually comes from an upper bound for the heat kernel (see \cite[Chapter 14.1]{grigoryanheat} and \cite[Theorem 4.1]{sturmdirichlet3}). 
A Feller semi-group is a semi-group that maps bounded continuous functions to bounded continuous functions.
\end{remark}
\subsection{The Bakry-Emery curvature-dimension condition}
In this section we introduce the curvature-dimension condition for Dirichlet forms in the sense of Bakry, Emery and Ledoux. 
The specific feature of this approach is the existence of an algebra $\mathcal{A}^{\sX}$ of bounded measurable functions on $X$ 
that is dense in $D_2(L^{\sX})$ and in all $L^p$-spaces, stable by $L^{\sX}$
%, stable by $P_t$ for any $t \geq 0$ 
and stable by composition with $C^{\infty}$-functions of several variables that vanish at $0$. We call such an algebra admissible. 
In the context of unbounded operators it is also a core for $(L^{\sX},D_2(L^{\sX}))$. A core for an unbounded operator is a subset of its domain that is dense with respect to the graph norm.
The algebra allows to introduce notions of curvature and dimension 
on a purely algebraic level and provides a calculus that simplifies proofs significantly. 

A consequence of the existence of an admissible algebra is that 
the ``carr\'e du champ''-operator for elements in $\mathcal{A}^{\sX}$ is obtained by
the following rule
\begin{align*}
\Gamma^{\sX}(u)=\frac{1}{2}L^{\sX}(u^2)-uL^{\sX}u\hspace{0.5cm}\mbox{for all } u\in\mathcal{A}^{\sX}.
\end{align*}
Provided $D(\mathcal{E}^{\sX})=\mathbb{G}$, this rule is consistent with Definition \ref{definitionGamma} (see \cite{hirsch}, section I.4). 
Replacing $L^{\sX}$ by $\Gamma^{\sX}$ in the definition of the carr\'e du champ we can define the so-called iterated carr\'e du champ or $\Gamma_2$-operator
\begin{align*}
\Gamma_2^{\sX}(u,v)=\frac{1}{2}L^{\sX}\Gamma^{\sX}(u,v)-\Gamma^{\sX}(u,L^{\sX}v)\hspace{0.5cm}\mbox{for all } u,v\in\mathcal{A}^{\sX}.
\end{align*}
We write $\Gamma^{\sX}(u)$ for $\Gamma^{\sX}(u,u)$ and similarly for $\Gamma^{\sX}_2$.
\begin{definition}[Classical Bakry-Emery curvature-dimension condition]\label{bakrycd1}
Assume there is an admissible algebra $\mathcal{A}^{\sX}$ for $\mathcal{E}^{\sX}$. Then
$\mathcal{E}^{\sX}$ satisfies the ``classical'' Bakry-Emery curvature dimension condition $BE(\kappa,N)$ of curvature $\kappa\in\mathbb{R}$ and dimension $1\leq N <\infty$ if 
\begin{align}\label{Gamma2estimate}
\Gamma^{\sX}_2(u)\geq \kappa \Gamma^{\sX}(u)+\frac{1}{N}\big(L^{\sX}u\big)^2\hspace{4mm}\mbox{ for all }u\in\mathcal{A}^{\sX}.
\end{align}
The inequality is understood to hold $\m_{\sX}$-almost everywhere in $X$.
Similar, the condition $BE(\kappa,\infty)$ holds if
$
\Gamma^{\sX}_2(u)\geq \kappa \Gamma^{\sX}(u)\hspace{1mm}\m_{\sX}\mbox{-a.e.}\mbox{ for all }u\in\mathcal{A}^{\sX}
$
and $BE(\kappa,N)$ implies $BE(\kappa,\infty)$.
\end{definition}
In 
many situations an algebra $\mathcal{A}^{\sX}$ is not available. To overcome this problem, in \cite{agsbakryemery} the Definition \ref{bakrycd1} was reformulated in an ``intrinsic'' way 
that also makes 
sense without the admissible algebra. For the rest of this section we will briefly present this approach and investigate the relation to the previous one. 
A more detailed description can be found in \cite{agsbakryemery}.
We still consider a regular and strongly local Dirichlet form $\mathcal{E}^{\sX}$ 
on some admissible space $X$ like in Section \ref{nonsymmetriccase}.
The $\Gamma_2$-operator can be defined in a weak sense by 
\begin{align*}
2\Gamma^{\sX}_2(u,v;\phi)=\Gamma^{\sX}(u,v;L^{\sX}\phi)-2\Gamma^{\sX}(u,L^{\sX}v;\phi)\hspace{5pt}\mbox{ for $u,v\in D(\Gamma_2)\mbox{ and }\phi\in D^{b,2}_+(L^{\sX})$}
\end{align*}
where
$
D(\Gamma^{\sX}_2):=\left\{u\in D_2(L^{\sX}): L^{\sX}u\in D(\mathcal{E}^{\sX})\right\}
$
and the set of test functions is denoted by 
\begin{align*}
D^{b,2}_+(L^{\sX})&:=\big\{\phi\in D_2(L^{\sX}): \phi, L^{\sX}\phi\in L^{\infty}(X,\m), \phi>0\big\}.
\end{align*}
$\Gamma^{\sX}_2$ is not symmetric in $u$ and $v$. We set $\Gamma^{\sX}_2(u,u;\phi)=\Gamma^{\sX}_2(u;\phi)$.
%Now we are able to state the curvature dimension condition in a weak sense.
\begin{definition}[Bakry-Emery curvature-dimension condition]\label{bakrycd2}
Let $\kappa\in\mathbb{R}$ and $N\geq 1$. We say that $\mathcal{E}^{\sX}$ satisfies the intrinsic Bakry-Emery curvature-dimension condition (or just Bakry-Emery condition) $BE(\kappa,N)$ if 
for every $u\in D(\Gamma^{\sX}_2)\mbox{ and }\phi\in D^{b,2}_+(L^{\sX})$, we have
\begin{align}\label{crucialestimate}
\Gamma^{\sX}_2(u;\phi)\geq \kappa \Gamma^{\sX}(u;\phi)+\frac{1}{N}\int_{\sX}(L^{\sX}u)^2\phi d\m.
\end{align}
In this case we have that $\mathbb{G}=D(\mathcal{E}^{\sX})$ (see \cite[Corollary 2.3]{agsbakryemery}). Hence, $\mathcal{E}^{\sX}$ is of diffusion-type.
As before we can also define $BE(\kappa,\infty)$ and the implications $BE(\kappa,N)\Rightarrow BE(\kappa,N')\Rightarrow BE(\kappa,\infty)$ for $N'\geq N$ hold as well.
\end{definition}
\begin{theorem}[Bakry-Ledoux gradient estimate]\label{bakryledouxestimate}
Let $\mathcal{E}^{\sX}$ be an admissible Dirichlet form. The estimate (\ref{crucialestimate}) for $\kappa \in\mathbb{R}$, $N\geq 1$ and any $(u,\phi)\in D(\Gamma^{\sX}_2)$ with $\phi\geq 0$ is equivalent to
the following gradient estimate. For any $u\in\mathbb{G}$ and $t>0$, $P_t^{\sX}u$ belongs to $\mathbb{G}$ and we have
\begin{align}
\Gamma^{\sX}(P_t^{\sX}u)+\frac{1-e^{-2\kappa t}}{N\kappa }\left(L^{\sX}P_t^{\sX}u\right)^2\leq e^{-2\kappa t}P_t^{\sX}\Gamma^{\sX}(u)\hspace{4pt}\m\mbox{-a.e. in }X.
\end{align}
\end{theorem}
\proof $\rightarrow$ The proof of the theorem in this form can be found in \cite{agsbakryemery} (see also \cite{bakryledouxliyau, erbarkuwadasturm}).
\begin{remark} If there is an admissible algebra that is stable with respect to $P_t^{\sX}$, the definitions \ref{bakrycd1} and \ref{bakrycd2} are consistent. On the one hand, Definition \ref{bakrycd2} and the existence of an admissible algebra $\mathcal{A}^{\sX}$
imply that for any test function
$\phi\in D^{b,2}_+(L^{\sX})$ and any $u\in\mathcal{A}^{\sX}$
\begin{align*}
\int_{\sX}\Gamma^{\sX}_2(u)\phi d\m_{\sX}\geq \kappa \Gamma(u;\phi)+\frac{1}{N}\int_{\sX}(L^{\sX}u)^2\phi d\m.
\end{align*}
Then we can replace $\phi\in D^{b,2}_+(L^{\sX})$ by any bounded and measurable function $\phi\geq 0$ by using the mollifying property of $P_t^{\sX}$, exactly like in \cite{agsbakryemery} and \cite{erbarkuwadasturm}.
This implies the classical Bakry-Emery condition in the sense of Definition \ref{bakrycd1} for $u\in\mathcal{A}^{\sX}$. 
On the other hand, if we assume the Bakry-Emery condition in the sense of Definition \ref{bakrycd1} for some admissible algebra $\mathcal{A}^{\sX}$ that is also stable under $P_t^{\sX}$, 
we can apply the following lemma.
\end{remark}
\begin{lemma}\label{importantlemma}
Assume there is a subset $\Xi\subset D_2(L^{\sX})$ that is dense with respect to the graph norm and stable under the Markovian semi-group $P_t^{\sX}$, and assume we have 
\begin{align*}
\Gamma^{\sX}_2(u;\phi)\geq \kappa \Gamma^{\sX}(u;\phi)+\frac{1}{N}\int_{\sX}(L^{\sX}u)^2\phi d\m\hspace{4pt}\mbox{ if $u\in D(\Gamma^{\sX}_2)\cap\Xi$ and $\phi\in D^{b,2}_+(L^{\sX})$. }
\end{align*}
Then $\mathcal{E}^{\sX}$ satisfies $BE(\kappa,N)$.
\end{lemma}
\proof The proof is straightforward and uses the mollification property of the semigroup and mainly Lemma 1.3.3 from \cite{fukushima} but we omit details here.
\subsection{Some examples of Dirichlet forms}\label{symmetriccase}
In this section we consider some examples in more detail. They will also play an important role later in the article.
Let $B$ be a smooth, $d$-dimensional manifold with or without boundary and let $g$ be a Riemannian metric on $B$. $\vol_{g}=\m_{\sB}$ is a smooth Radon measure 
and $(B,\m_{\sB})$ is an admissible space. 
We set $\hat{B}=B\backslash \partial B$, we assume that $(\hat{B},\de_{\sB})$ is geodesically convex,
%We remind the reader on the following inclusions
%$
%\begin{align*}
%C^{\infty}_0(\hat{B})\subset C^{\infty}_0(B)\subset C^{\infty}(\hat{B})
%\end{align*}
%$
%where the last inclusion has to be understood in the following sense:
%$
%u\in C^{\infty}(B) \hspace{4pt}\Rightarrow\hspace{4pt} u|_{\hat{B}}\in C^{\infty}(\hat{B}). 
%$
and we consider the standard Dirichlet form with Dirichlet boundary conditions. 
%More precisely, this is the form closure of
%\begin{align*}
%\mathcal{E}^{\sB}(u)=\int_{B}|\nabla u|_g^2d\vol_{\sB} \mbox{ where }u\in C_0^{\infty}(\hat{B}).
%\end{align*}
Its domain is
$ D(\mathcal{E}^{\sB})=W^{1,2}_0(\hat{B},d\vol_{\sB}).$
The associated self-adjoint operator is the Dirichlet Laplace operator $\Delta^{\sB}$ with domain $D^2(L^{\sB})=W^{2,2}_0(\hat{B},\vol_{\sB}).$
In this context, we have $\Gamma^{\sB}(u)=|\nabla u|_g^2$.
We assume that $(B,g)$ has Ricci curvature bounded from below by $(d-1)K$.
\paragraph{Weighted Riemannian manifolds} 
Consider a smooth
$f:B\rightarrow \mathbb{R}_{\geq 0}$ in $D^2(L^{\sB})$ such that $f|_{\partial B}=0$, $f>0$ in $\hat{B}$ and $f$ is $\mathcal{F}K$-concave in the following sense:
$
\nabla^2f(v)\geq K |v|_g^2 \hspace{5pt}\mbox{ for any }v\in TB,
$
where $\nabla^2f$ denotes the Hessian of $f$ with respect to $g$.
Consider $({B},g ,f^{\sN}d\vol_{\sB})$. We can define another symmetric form $\mathcal{E}^{\sB,f^{\sN}}$ on $L^2(B,f^{\scriptscriptstyle{N}}d\vol_{\sB})$ by
\begin{equation}\label{weighteddirichletform}
\mathcal{E}^{B,f^{\sN}}(u)=\int_{B}|\nabla u|_g^2f^{\sN}d\vol_{\sB}\hspace{5pt}\mbox{ for }u\in C^{\infty}_0(\hat{B}).
\end{equation}
Then $\mathcal{E}^{\sB,f^{\sN}}$ is closable on $C^{\infty}_0(\hat{B})$ (see \cite[Theorem 6.3.1]{fukushima}, \cite{maroeckner}), and it
becomes a strongly local and regular Dirichlet form.
$
\Gamma^{\sB,f^{\sN}}(u)=\Gamma^{\sB}(u)=|\nabla u|_g^2\hspace{4pt}\mbox{ for any }u\in C^{\infty}_0(\hat{B}).
$

\begin{remark}
If we replace in (\ref{weighteddirichletform}) $C^{\infty}_0(\hat{B})$ by $C_0^{\infty}(B)$, we obtain the Dirichlet form with Neumann boundary conditions. 
When the boundary of $B$ is empty, then the form coincides with the form with Dirichlet boundary conditions. In general, this is not the case. But if
the boundary $\partial B$ is a polar set in the sense of Grigor'yan and Masamune (see \cite{grigoryan}), it is also true. 
In the weighted situation that we are considering the boundary is a polar set in this sense. 
In particular, it follows that $C^{\infty}_0(B)\subset D(\mathcal{E}^{\sB,\sin_{\sK}^{\sN}})$.
\end{remark}
\begin{proposition}
For $u\in C^{\infty}(\hat{B})\cap D^2(L^{\sB,f^{\sN}})$ there is an explicit formula for the generator of $\mathcal{E}^{\sB,f^{\sN}}$ given by
\begin{equation}\label{op}
\big(L^{B,f^{\sN}}u\big)(p)=\big(\Delta^{\sB}u\big)(p) +\frac{N}{f(p)}\left\langle\nabla f,\nabla u\right\rangle|_p
\hspace{4pt}\mbox{ for any $p\in \hat{B}$.}
\end{equation}
\end{proposition}
\begin{proposition}\label{keineahnung}
Let $(B,g,f^{\sN}\vol_{\sB})$ be as above.
Then for any $u\in C^{\infty}(\hat{B})$ the following $\Gamma_2$-estimate holds pointwise everywhere in $\hat{B}$:
\begin{align}
 \Gamma_2^{\sB,f^{\sN}}(u)
 %:=\textstyle{\frac{1}{2}}L^{\sB,f^{\sN}}|\nabla u|^2_g-\big\langle\nabla u,\nabla L^{\sB,f^{\sN}}u\big\rangle_g
 \geq(d+N-1)K|\nabla u|^2_g+\textstyle{\frac{1}{d+N}}\big(L^{\sB,f^{\sN}}u\big)^2.\label{est2}
\end{align}
\end{proposition}
\begin{proof}
Since we have (\ref{mimic}) for $\Gamma^{\sB}_2=\frac{1}{2}\Delta^{\sB}|\nabla u|_g^2-\langle \nabla u,\nabla \Delta^{\sB}u\rangle$ and since $f$ is $\mathcal{F}K$-concave, we get pointwise for any $u\in C^{\infty}(\hat{B})$
\begin{align}
\Gamma_2^{\sB,f^{\sN}}(u)%&=\hspace{1pt}{\frac{1}{2}}L^{\sB,f^{\sN}}|\nabla u|^2_g-\big\langle\nabla u,\nabla L^{\sB,f^{\sN}}u\big\rangle_g\nonumber\\
%&=\hspace{1pt}{\frac{1}{2}}\Delta^{\sB}|\nabla u|_g^2+\hspace{1pt}{\frac{1}{2}\frac{N}{f}}\big\langle\nabla f,\nabla |\nabla u|_g^2\big\rangle-\big\langle \nabla \Delta^{\sB}u,\nabla u\big\rangle-N\big\langle\nabla\big(\hspace{1pt}{\frac{1}{f}}\big\langle \nabla f,\nabla u\big\rangle\big),\nabla u\big\rangle\nonumber\\
%&=\Gamma_2^{\sB}(u)+\frac{1}{2}\frac{N}{f}\Gamma^{\sB}(f,\Gamma^{\sB}(u))-N\Gamma^{\sB}(\frac{1}{f}\Gamma^{\sB}(f,u),u)\nonumber\\
%&=\Gamma_2^{\sB}(u)-\hspace{1pt}{\frac{N}{f}}\nabla^2f(u)
%-\big\langle \nabla \Delta^{\sB}u,\nabla u\big\rangle
%-N\big\langle\nabla\big(\hspace{1pt}{\frac{1}{f}}\big\langle \nabla f,\nabla u\big\rangle\big),\nabla u\big\rangle\nonumber\\
&=\ric_{\sB}(\nabla u)+\|\nabla^2u \|^2_{HS}-\hspace{1pt}{\frac{N}{f}}\nabla^2f(\nabla u)+\hspace{1pt}{\frac{N}{f^2}}\big\langle\nabla f,\nabla u\big\rangle\big\langle \nabla u,\nabla f\big\rangle\nonumber\\
&\geq(d-1)K|\nabla u|_g^2+\hspace{1pt}{\frac{1}{d}}\left(\Delta^{\sB}u\right)^2+NK|\nabla u|_g^2+\hspace{1pt}{\frac{1}{N}}\big(\hspace{1pt}{\frac{N}{f}}\big\langle \nabla f,\nabla u\big\rangle\big)^2\label{estimatos}\\
&\geq(d+N-1)K\Gamma^{\sB}(u)+\hspace{1pt}{\frac{1}{d+N}}\big(L^{\sB,f^{\sN}}u\big)^2.\nonumber
\end{align}
%We used
%$
%\textstyle{\frac{1}{d}a^2+\frac{1}{N}b^2\geq \frac{1}{N+d}\left(a+b\right)^2}
%$ for all $d,N\geq 1$ and for all $a,b\in\mathbb{R}$. 
Details can be found in chapter 14 of \cite{viltot}.
\end{proof}

\begin{example}[$1$-dimensional model spaces]\label{onedimensionalexample}
Let $B$ be of the form
$I_K=[0,{{\pi}/{K}}]${ for }$K>0${ and }$[0,\infty)${ for }$K\leq 0$. 
The corresponding operator $L^{I_K}$ is $d^2/dx^2$ and its domain is $H_0^{1,2}(I_K,dx)$. It satisfies $BE(0,1)$.
Consider $f:I_K\rightarrow\mathbb{R}_{\geq 0}$ in $D_2(L^{I_K})$ 
that is given by
$$
 f(t)=\sin_K(t):=\begin{cases}
       \textstyle{\frac{1}{\sqrt{K}}}\sin(\sqrt{K}t)&\mbox{ for }K>0\\
       t&\mbox{ for }K=0\\
       \textstyle{\frac{1}{\sqrt{|K|}}}\sinh(\sqrt{|K|}t)&\mbox{ for }K<0.\\
      \end{cases}
$$
We can define $\mathcal{E}^{I_K,f^{\sN}}$ as before. 
\end{example}
%\begin{proposition}\label{rosen}
%Let $K>0$ and $N\geq 1$. $\mathcal{E}^{I_{\sK},\sin_{\sK}^{\sN}}$ satisfies $BE(NK,N+1)$.
%\end{proposition}
%\begin{proof}
%Consider the densely defined operator  \begin{align*}
%L=L^{I_{\sK},\sin_{\sK}^{\sN}}|_{C^{\infty}(I_{\sK})}=\frac{d^2}{dr^2}-N\frac{\cos}{\sin}\frac{d}{dr}
%\end{align*}
%on $L^2(\sin_{\sK}^{\sN}tdt)$. By change of variable $\sin r=s$, $L$ can be described as operator
%\begin{align*}
%\tilde{L}u(s)=(1-s^2)\frac{d^2}{ds^2}u(s)-(N+1)s\frac{d}{ds}u(s)
%\end{align*}
%on $L^2((1-t^2)^{(n/2)-1}dt)$ (e.g., \cite{ledouxgeometry}) A complete, orthonormal class of eigenfunctions for $\tilde{L}$ is given by a family of Jacobi Polynomials $(Q_n)_{n\in\mathbb{N}}$ (see \cite{mazet}). Then $(Q_n\circ \sin r)_{n\in\mathbb{N}}$
%is also a complete family of eigenfunction for $L^{I_{\sK},\sin_{\sK}^{\sN}}$. This set generates an algebra of smooth functions on $I_{\sK}$ that is dense in $D^2(L^{I_{\sK},\sin_{\sK}^{\sN}})$,
%stable with respect to $L^{I_{\sK},\sin_{\sK}^{\sN}}$ and stable with respect to the semi group. Then, because of the previous proposition, the classical Bakry-Emery condition $BE(NK,N+1)$ is satisfied, and 
%Lemma \ref{importantlemma} implies also the weak formulation.\end{proof}
%
\begin{proposition}\label{rosen}
Let $K>0$ and $N\geq 1$. $\mathcal{E}^{I_{\sK},\sin_{\sK}^{\sN}}$ satisfies $BE(NK,N+1)$.
\end{proposition}
\begin{proof}
$\mathcal{E}^{I_{\sK},\sin_{\sK}^{\sN}}$ is the Cheeger energy of $(I_{\sK},\sin_{\sK}^{\sN}rdr)$ (see Section \ref{firstordercalculus}) 
and $(I_{\sK},\sin_{\sK}^{\sN}rdr)$ satisfies the condition $CD(KN,N+1)$ by the equivalence of Theorem \ref{theorembochner} that will be presented later. 
Hence, the result follows (see \cite{stugeo2}).
\end{proof}

\section{Skew products and the Bakry-Emery curvature-dimension condition}\label{proofmain}

\subsection{Skew and $N$-skew products between Dirichlet forms}
In this section we define skew and $N$-skew products for Dirichlet forms. The notion of skew product is well-known and 
has been introduced by Fukushima and Oshima in \cite{fukushimaoshima}. 
An $N$-skew product is a slight modification of that definition where we also change the topology of the underlying space. 

Let $B$ be a $d$-dimensional Riemannian manifold with or without boundary such that $B$ is geodesically convex and 
$\ric_{\sB}\geq (d-1)K$ and let $\mathcal{E}^{\sB}$ be its standard Dirichlet form with Dirichlet boundary conditions.
Let $f\in D_2(L^{\sB})$ be smooth and $\mathcal{F}K$-concave.
Let $\mathcal{E}^{\sF}$ be a regular and strongly local Dirichlet form on $L^2(F,\m_{\sF})$ where $F$ is an admissible space.
Consider $(B\times F,\mathcal{O}_{\sB}\otimes\mathcal{O}_{\sF})$ with $\m_{\scriptscriptstyle{C}}=f^{\sN}d\vol_{\sB}\otimes d\m_{\sF}$
and the tensor product 
$C^{\infty}_0(B)\otimes D(\mathcal{E}^{\sF})$. $\mathcal{O}_{\sB}\otimes\mathcal{O}_{\sF}$ is the product topology.
\begin{remark}
The elements of $C^{\infty}_0(B)\otimes D(\mathcal{E}^{\sF})$ are functions of the form
$
\sum_{i=1}^ku_1^iu_2^i
$
for some finite $k\in \mathbb{N}$ and $u_1^i\in C^{\infty}_0(B)$ and $u_2^i\in D(\mathcal{E}^{\sF})$. We will follow this convention in the rest of the article. In the literature the tensor product
between infinite dimensional Hilbert spaces $\mathcal{H}_i$ for $i=1,2$ means that one also takes the closure with respect to the induced inner product. 
Later, this construction will also appear and we use the following notation $\overline{\mathcal{H}_1\otimes \mathcal{H}_2}$.
\end{remark}
\begin{definition}[Skew product]\label{skewskew}
Consider the closure of the following densely defined symmetric form on $L^2(B\times F,f^{\sN}d\vol_{\sB}\otimes d\m_{\sF})$:
\begin{align}\label{beside}
\mathcal{E}^{\sC}(u)=\int_{\sF}\mathcal{E}^{\sB}(u^x)d\m_{\sF}(x)+\int_{B}\mathcal{E}^{\sF}(u^p)f^{N-2}(p)d\vol_{\sB}(p)<\infty
\end{align}
 for $u\in C^{\infty}_0(\hat{B})\otimes D(\mathcal{E}^{\sF})$ where $u^x=u(\cdot,x)$ and $u^p=u(p,\cdot)$ are the horizontal respectively vertical sections of $u$. 
$({B}\times F,\mathcal{O}_{{\sB}}\otimes\mathcal{O}_{\sF},\mathcal{E}^{\sC})$ is called skew product between ${B}$, $f$ and $\mathcal{E}^{\sC}$.
\end{definition}
\begin{remark}\label{somethingeasy}
$D(\mathcal{E}^{\sB,f^{\sN}})\otimes D(\mathcal{E}^{\sF})\subset D(\mathcal{E}^{\sC})$ and 
\begin{align}
\mathcal{E}^{\sC}(u)=\int_{\sF}\mathcal{E}^{\sB}(u^x)d\m_{\sF}(x)+\int_{B}\mathcal{E}^{\sF}(u^p)f^{N-2}(p)d\vol_{\sB}(p),
\end{align}
holds for any $u\in D(\mathcal{E}^{\sB,f^{\sN}})\otimes D(\mathcal{E}^{\sF})$.
\\
\\
The next proposition is a Fubini-type result and was proven by Okura in \cite{okura}.\end{remark}
\begin{proposition}\label{okurafubinitype}
Let $\mathcal{E}^{\sC}$ be a skew product like in Definition \ref{definitionwarpedproduct}. Consider $u\in D(\mathcal{E}^{\sC})$. Then $u^x\in D(\mathcal{E}^{\sB,f^{\sN}})$ 
%$u^x\in D(\mathcal{E}^{B,f^{\sN},f^{-2}})$ 
for $\m_{\sF}$-almost every $x\in F$ and $u^p\in D(\mathcal{E}^{\sF})$ for
$\vol_{\sB}$-almost every $p\in B$ and we have
\begin{align}\label{holds}
\mathcal{E}^{\sC}(u)=\int_{\sF}\mathcal{E}^{\sB}(u^x)d\m_{\sF}(x)+\int_{B}\mathcal{E}^{\sF}(u^p)f^{N-2}(p)d\vol_{\sB}(p),
\end{align}
Especially $\mathcal{E}^{\sC}$ admits a $\Gamma$-operator if and only if $\mathcal{E}^{\sF}$ does so, and in this case we have for $u\in D(\mathcal{E}^{\sC})$
$$
\Gamma^{\sC}(u)(p,x)=\Gamma^{\sB}(u^x)(p)+\textstyle{\frac{1}{f^2(p)}}\Gamma^{\sF}(u^p)(x)\ \mbox{ $\m_{\sC}$-a.e. }.
$$
\end{proposition}

\begin{corollary}\label{anotherstupidcorollary}
Let $\mathcal{E}^{\sC}$ be skew product like in Definition \ref{definitionwarpedproduct}. Then $C^{\infty}_0(\hat{B})\otimes D^2(L^{\sF})\subset D^2(L^{\sC})$ and 
\begin{align}\label{andanotherone}
\left(L^{\sC}u\right)(p,x)=\big(L^{\sB,f^{\sN}}u^x\big)(p)+\textstyle{\frac{1}{f^2(p)}}\left(L^{\sF}u^p\right)(x)\hspace{5pt}\mbox{ for }\m_{\sC}\mbox{-a.e. }(p,x)\in C.
\end{align}
\end{corollary}
\begin{proof}
We consider $u\in C^{\infty}_0(\hat{B})\otimes D^2(L^{\sF})$ and $v\in D(\mathcal{E}^{\sC})$. 
Then $u^x\in C^{\infty}_0(\hat{B})$ for $\m_{\sF}$-almost every $x$ and $u^p\in D^2(L^{\sF})$ for every $p$, and $v^x\in D(\mathcal{E}{\sB,f^{\sN}})$ for $\m_{\sF}$-almost every $x$ and 
$v^p\in D(\mathcal{E}^{\sF})$ for $\vol_{\sB}$ almost every $p$.
Hence
\begin{align*}
\mathcal{E}^{\sB,f^{\sN}}(u^x,v^x)=\big(L^{\sB,f^{\sN}}u^x,v^x\big)_{L^2(f^{\sN}d\vol_{\sB})}\hspace{5pt}\mbox{ and }\hspace{5pt}
\mathcal{E}^{\sF}(u^p,v^p)=\big(L^{\sF}u^p,v^p\big)_{L^2(\m_{\sF})}
\end{align*}for $\m_{\sF}$-almost every $x$ and for $\vol_{\sB}$ almost every $p$.
This and Proposition \ref{okurafubinitype} implies
\begin{align*}
\mathcal{E}^{\sC}(u,v)&=\int_{\sF} \mathcal{E}^{\sB,f^{\sN}}(u^x,v^x)d\m_{\sF}(x)+\int_{\sB}\textstyle{\frac{1}{f^2(p)}}\mathcal{E}^{\sF}(u^p,v^p)f^{\sN}(p)d\vol_{\sB}(p)\\
&=-\int_{\sF} \big(L^{\sB,f^{\sN}}u^x,v^x\big)_{L^2(f^{\sN}\vol_{\sB})}d\m_{\sF}(x)+\int_{\sB}\textstyle{\frac{1}{f^2(p)}}\big(L^{\sF}u^p,v^p\big)_{L^2(\m_{\sF})}f^{\sN}(p)d\vol_{\sB}(p)\\
&=-\int_{\sC}\big(L^{\sB,f^{\sN}}u^x(p)+\textstyle{\frac{1}{f^2(p)}}L^{\sF}u^p(x)\big)v(p,x)d\m_{\sC}(p,x).
\end{align*}
Then we also see that $L^{\sB,f^{\sN}}u^x(p)+f^{-2}(p)L^{\sF}u^p(x)$ is $L^2$-integrable with respect to $\m_{\sC}$. 
First, we consider $u=u_1\otimes u_2 \in C^{\infty}_0(\hat{B})\otimes D^2(L^{\sF})$. Then $u$ is $L^2$-integrable with respect to $\m_{\sC}$ since
\begin{align*}
\big\|L^{\sB,f^{\sN}}u+\textstyle{\frac{1}{f^2}}L^{\sF}u\big\|^2_{L^2(\m_{\sC})}\leq \frac{1}{2}\big\|L^{\sB,f^{\sN}}u_1\big\|^2\big\|u_2\big\|^2_{L^2(\m_{\sF})}+\frac{1}{2}\big\|\textstyle{\frac{u_1}{f}}\big\|^2_{L^2(f^{\sN}d\vol_{\sB})}\big\|L^{\sF}u_2\big\|^2<\infty.
\end{align*}
In particular, we used that $u_1$ is smooth with compact support in $\hat{B}$. Hence, $u_1\otimes u_2\in D^2(L^{\sC})$ and (\ref{andanotherone}) holds. 
In general, any $u\in C^{\infty}_0(\hat{B})\otimes D^2(L^{\sF})$ has the form
\begin{align*}
u=\sum_{i=1}^k u_1^i\otimes u_2^i=\sum_{i=1}^k u^i
\end{align*}
where $L^{\sC}u^i$ is $L^2$-integrable. Then, by linearity of $L^{\sB,f^{\sN}}+\frac{1}{f^2}L^{\sF}$ and by triangle inequality also $L^{\sB,f^{\sN}}u^x+\frac{1}{f^2}L^{\sF}u^p$ is $L^2$-integrable and (\ref{andanotherone}) holds.
\end{proof}
\paragraph{$N$-skew products} We will introduce a slight modification of Definition \ref{skewskew}. The underlying space of $\mathcal{E}^{\sC}$ 
is $B\times F$ equipped with the product topology
$\mathcal{O}_{\sB}\otimes \mathcal{O}_{\sF}$ but in general the intrinsic distance $\de_{\mathcal{E}^{\sC}}$ induces a different topology that we will describe in more detail.
Let us define an equivalence relation on $B\times F$ as follows:
\begin{align*}
(p,x)\sim (q,y)\Longleftrightarrow \Big(p=q\in \partial B\Big)\mbox{ or }\Big(p=q\in B \mbox{ and }x=y\in F\Big).
\end{align*}
Then we can consider the quotient space $B\times F/_{\sim}=:C$ and the corresponding projection map $\pi: B\times F\rightarrow C$. Obviously, we have the following decomposition
\begin{align*}
C=\partial B\
\dot{\cup}\ \hat{B}\times F.
\end{align*}
A subset $V\subset C$ is open if and only if $\pi^{-1}(V)\subset B\times F$ is open. We denote the corresponding topology with $\mathcal{O}_{\sC}$. If $u$ is continuous with respect to
$\mathcal{O}_{\sC}$ then $u\circ \pi=\tilde{u}$ is continuous with respect to $\mathcal{O}_{\sB}\otimes \mathcal{O}_{\sF}$. 
By abuse of notation we will also write $u=\tilde{u}$ when the meaning is clear.
If $\mathcal{E}^{\sF}$ is strongly regular, one can define a family of ``open balls'' that generates the quotient topology, i.e. any open set is a union of elements from this family. 
First, we pick $(p,x)\in \hat{B}\times F$ and we consider $\bar{\epsilon}_{[p,x]}=\inf_{q\in\partial B}\de_{\mathcal{E}^{\sB}}(p,q)$. Then, admissible $\epsilon$-balls around $[p,x]$ are
\begin{align*}
B^{\sC}_{\epsilon}([p,x])=\left\{[q,y]\in C: \de_{\mathcal{E}^{\sB}}(q,p)+\de_{\mathcal{E}^{\sF}}(x,y)<\epsilon\right\}\subset \hat{B}\times F\subset C\hspace{4pt}\mbox{for }0<\epsilon<\bar{\epsilon}_{[p,x]}.
\end{align*}
For $p=[p,x]\in\partial B\subset C$ the corresponding $\epsilon$-balls are
\begin{align*}
B^{\sC}_{\epsilon}([p,x])=\left\{[q,y]\in C: \de_{\mathcal{E}^{\sB}}(q,p)<\epsilon\right\}\subset C\hspace{6pt}\mbox{ for }0<\epsilon<\bar{\epsilon}_{[p,x]}=\infty
\end{align*}
The family of all admissible balls is denoted by 
$$\mathcal{B}=\left\{B^{\sC}_{\epsilon}([p,x]):[p,x]\in C,\ 0<\epsilon<\bar{\epsilon}_{[p,x]}\right\}.$$
It is not hard to check that elements from $\mathcal{B}$ are open with respect to $\mathcal{O}_{\sC}$ and that $\mathcal{B}$ is a generator for $\mathcal{O}_{\sC}$.
We can pushforward the measure $\m_{\sC}$ to $C$ and denote it also with $\m_{\sC}$. $\partial B\subset C$ is a set of measure zero. 
Hence, $C$ keeps its product structure $\m_{\sC}$-almost everywhere. Then we can interpret $\mathcal{E}^{\sC}$ also as a Dirichlet form on $L^2(C,\m_{\sC})=:L^2(\m_{\sC})$ and the previous
Fubini-type results are valid as well.
\begin{definition}[$N$-skew product]\label{definitionwarpedproduct}
Assume $\mathcal{E}^{\sF}$ is a strongly local, regular and strongly regular Dirichlet form and $B$ and $f$ as in Definition \ref{skewskew}.
Consider the admissible space $(C,\mathcal{O}_{\sC},\m_{\sC})$ 
and define a Dirichlet form $\mathcal{E}^{\sC}$ on $L^2(\m_{\sC})$ as in Definition \ref{skewskew}. 
We call $(C,\mathcal{O}_{\sC},\m_{\sC},\mathcal{E}^{\sC})$ the $N$-skew product between $B$, $f$ and $\mathcal{E}^{\sF}$ and 
we will write $\mathcal{E}^{\sC}=\mathcal{E}^{\sB}\times_f^{\sN}\mathcal{E}^{\sF}=B\times_f^{\sN}\mathcal{E}^{\sF}$. 
$B\times_{f}^{\sN}\mathcal{E}^{\sF}$ is strongly local and regular.
\\
\\
Consider the intrinsic distance $\de_{\mathcal{E}^{\sC}}$ of $\mathcal{E}^{\sC}=B\times^{\sN}_{f} \mathcal{E}^{\sF}$ on $C$. 
The topology that is induced by $\de_{\mathcal{E}^{\sC}}$ is denoted by ${\mathcal{O}}_{\de}$. One can check easily the following Lemma. We omit the proof.\end{definition}
\begin{lemma}\label{stronglyregular}
If $\mathcal{E}^{\sF}$ is strongly regular, then $\mathcal{E}^{\sC}=B\times^{\sN}_f\mathcal{E}^{\sF}$ is strongly regular, i.e.  $\mathcal{O}_{\de}=\mathcal{O}_{\sC}$.
Closed $\epsilon$-balls with respect to $\de_{\mathcal{E}^{\sC}}$ are compact if this property holds for $\de_{\mathcal{E}^{\sF}}$.
\end{lemma}
\begin{definition}[$(K,N)$-cones]\label{cones}
Assume the situation of example (\ref{onedimensionalexample}) and let $\mathcal{E}^{\sF}$ be a Dirichlet form as before on some space $F$.
Then the $N$-skew product with respect to $I_K$ and $\sin_K$ is well-defined and called $(K,N)$-cone over $\mathcal{E}^{\sF}$.
\end{definition}
\subsection{Proof of classical $\Gamma_2$-estimates for $N$-skew products}
We fix a regular and strongly local Dirichlet form $\mathcal{E}^{\sF}$ on $L^2(F,\m_{\sF})$ for some admissible space $(F,\m_{\sF})$.
In this subsection we assume there is an admissible algebra $\mathcal{A}^{\sF}$ for $\mathcal{E}^{\sF}$. This enables us to do calculations classically.
In subsection 2.4, we will prove analogous results for Dirichlet forms that satisfy the intrinsic Bakry-Emery curvature-dimension condition.
The advantages of the classical approach are that calculations can be done pointwise and that the structure of formulas and inequalities becomes more clear. 
\begin{theorem}\label{maintheorem}
Let $B$ be a Riemannian manifold (with or without boundary) such that $\hat{B}$ is geodesically convex
and let $\mathcal{E}^{\sB}$ be the associated standard Dirichlet form with Dirichlet boundary conditions that satisfies $BE((d-1)K,d)$. 
Let $f\in D_2(L^{\sB})$ be smooth and $\mathcal{F}K$-concave.
Let $\mathcal{E}^{\sF}$ satisfy $BE((N-1)K_{\sF},N)$ for $N\geq 1$ and $K_{\sF}\in\mathbb{R}$ such that 
\begin{itemize}
\item[]$\Gamma^{\sB}(f)+Kf^2\leq K_{\sF} \mbox{ on } B$.
\end{itemize}
Assume there is an admissible algebra $\mathcal{A}^{\sF}$ for $\mathcal{E}^{\sF}$ and set $\mathcal{A}^{\sC}=C^{\infty}_0(\hat{B})\otimes \mathcal{A}^{\sF}$.
Then the $N$-skew product $\mathcal{E}^{\sC}=B\times_f\mathcal{E}^{\sF}$ satisfies 
\begin{align}\label{4}
\Gamma^{\sC}_2(u)\geq (N+d-1)K \Gamma^{\sC}(u)+\frac{1}{N+d}\left(L^{\sC}u\right)^2
\end{align}
pointwise $\m_{\sC}$-everywhere and for any $u\in\mathcal{A}^{\sC}$.
\end{theorem}
\begin{proof}
Every element of $\mathcal{A}^{\sC}$ is of the form $\sum_{i=1}^k u_1^iu_2^i$ for $k\in\mathbb{N}$, but we will check (\ref{4}) only for elements of the form $u+v=u_1\otimes u_2+v_1\otimes v_2$ where
$u_1,v_1\in C^{\infty}_0(\hat{B})$ and $u_2, v_2\in\mathcal{A}^{\sF}$. The case of arbitrary finite sum follows in the same way.
We compute $\Gamma^{\sC}_2(u)$, $\Gamma_2^{\sC}(v)$, $\Gamma_2^{\sC}(u,v)$ and $\Gamma_2(v,u)$ explicitly $\m_{\sC}$-a.e.\ . A straightforward calculation yields:
\begin{align}
\Gamma^{\sC}_2(u,v)+\Gamma^{\sC}_2(v,u)&=\big(\Gamma_2^{\sB}(u_1,v_1)+\Gamma_2^{\sB}(v_1,u_1)\big)v_2u_2+\textstyle{\frac{u_1v_1}{f^4}}\big(\Gamma_2^{\sF}(u_2,v_2)+\Gamma_2^{\sF}(v_2,u_2)\big)\nonumber\\	
&\hspace{5mm}+\Gamma^{\sB}(u_1,v_2)\textstyle{\frac{1}{f^2}}L^{\sF}(u_2v_2)-\Gamma^{\sB}\big(\textstyle{\frac{u_1}{f^2}},v_1\big)L^{\sF}(u_2)v_2-\Gamma^{\sB}\big(u_1,\textstyle{\frac{v_1}{f^2}}\big)u_2L^{\sF}(v_2)\nonumber\\
&\hspace{5mm}+\big(L^{\sB,f}(\textstyle{\frac{u_1v_1}{f^2}})-L^{\sB,f}(u_1)v_1\textstyle{\frac{1}{f^2}}-u_1L^{\sB,f}(v_1)\textstyle{\frac{1}{f^2}}\big)\Gamma^{\sF}(u_2,v_2)\label{zzz}
\end{align}
We set 
\begin{align*}
\Gamma^{\sC}_2(u,v)+\Gamma^{\sC}_2(v,u)-\big(\Gamma_2^{\sB}(u_1,v_1)+\Gamma_2^{\sB}(v_1,u_1)\big)v_2u_2-\textstyle{\frac{u_1v_1}{f^4}}\big(\Gamma_2^{\sF}(u_2,v_2)+\Gamma_2^{\sF}(v_2,u_2)\big)=:(\mathcal{J})
\end{align*}
When we use the chain rule for $\Gamma^{\sB}$ and $L^{\sB}$
then it yields the following:
\begin{align}
\Gamma^{\sB}\big(\textstyle{\frac{u_1}{f^2}},v_1\big)&=\textstyle{\frac{1}{f^2}}\Gamma^{\sB}(u_1,v_1)-\textstyle{\frac{2u_1}{f^3}}\Gamma^{\sB}(f,v_1)\nonumber\\
L^{\sB}(\textstyle{\frac{u_1v_1}{f^2}})&=\textstyle{\frac{v_1}{f^2}}L^{\sB}(u_1)+\textstyle{\frac{u_1}{f^2}}L^{\sB}(v_1)-\textstyle{\frac{2u_1v_1}{f^3}}L^{\sB}(f)\nonumber\\
&\hspace{10pt}-\textstyle{\frac{4v_1}{f^3}}\Gamma^{\sB}(u_1,f)-\textstyle{\frac{4u_1}{f^3}}\Gamma^{\sB}(v_1,f)  +\textstyle{\frac{6u_1v_1}{f^4}}\Gamma^{\sB}(f) +\textstyle{\frac{2}{f^2}}\Gamma^{\sB}(u_1,v_1)\nonumber\\
\textstyle{\frac{N}{f}}\Gamma^{\sB}(f,\textstyle{\frac{u_1v_1}{f^2}})&=\textstyle{\frac{Nv_1}{f^3}}\Gamma^{\sB}(f,u_1)+\textstyle{\frac{Nu_1}{f^3}}\Gamma^{\sB}(f,v_1)-\textstyle{\frac{2Nu_1v_1}{f^4}}\Gamma^{\sB}(f).\nonumber
\end{align}
We glue this back into $(\mathcal{J})$ and another straightforward calculation yields.
\begin{align}
(\mathcal{J})
&=\textstyle{\frac{2u_1}{f^3}}\Gamma^{\sB}(f,v_1)L^{\sF}(u_2)v_2+\textstyle{\frac{2v_1}{f^3}}\Gamma^{\sB}(f,u_1)L^{\sF}(v_2)u_2-\textstyle{\frac{2u_1v_1}{f^3}}\big(L^{\sB}(f)+\textstyle{\frac{N-1}{f}}\Gamma^{\sB}(f)\big)\Gamma^{\sF}(u_2,v_2)\nonumber\\
&\hspace{5mm}+\underbrace{\big(-\textstyle{\frac{4v_1}{f^3}}\Gamma^{\sB}(u_1,f)-\textstyle{\frac{4u_1}{f^3}}\Gamma^{\sB}(v_1,f)+\textstyle{\frac{4u_1v_1}{f^4}}\Gamma^{\sB}(f) +\textstyle{\frac{4}{f^2}}\Gamma^{\sB}(u_1,v_1)\big)}_{=:(I)(u_1,v_1)}\Gamma^{\sF}(u_2,v_2)\nonumber
\end{align}
In the case $u=v$ we obtain the following simplification
\begin{align}\label{hmmm}
(\mathcal{J})&=\textstyle{\frac{4u_1}{f^3}}\Gamma^{\sB}(f,u_1)L^{\sF}(u_2)u_2-\textstyle{\frac{2u_1^2}{f^3}}\big(L^{\sB}(f)+\textstyle{\frac{N-1}{f}}\Gamma^{\sB}(f)\big)\Gamma^{\sF}(u_2)
+(I)(u_1).
\end{align}
Now, we can compute the $\Gamma_2$-operator for elements of the form $u_1\otimes u_2+v_1\otimes v_2=u+v\in C^{\infty}_0(\hat{B})\otimes \mathcal{A}^{\sF}$ for $\m_{\sC}$-a.e. point $(p,x)\in \hat{C}$:
\begin{align*}
&2\Gamma_2^{\sC}(u+v)(p,x)=2\Gamma_2^{\sB}(u^x)(p)+\textstyle{\frac{1}{f^4(p)}}2\Gamma_2^{\sF}(u^p)(x)\nonumber\\
&\hspace{5mm}+\textstyle{\frac{4}{f^3(p)}}\Gamma^{\sB}(f,u^x)(p)L^{\sF}(u^p)(x)-\textstyle{\frac{2}{f^3(p)}}\big(L^{\sB}(f)(p)+\textstyle{\frac{N-1}{f(p)}}\Gamma^{\sB}(f)(p)\big)\Gamma^{\sF}(u^p)(x)\nonumber\\
&\hspace{5mm}+2(I)(u_1,v_1)(p)\Gamma^{\sF}(u_2,v_2)(x)+(I)(u_1)(p)\Gamma^{\sF}(u_2)(x)+(I)(v_1)(p)\Gamma^{\sF}(v_2)(x)
\end{align*}
We denote the last line in the previous equation with $(II)$. First, assume $u_1(p)v_1(p)\neq 0$. Then, it can be rewritten in the following form
\begin{align*}
(II)&=\textstyle{\frac{2u_1v_1}{f^2}}\big(-4\Gamma^{\sB}(\ln |u_1|,\ln f)-4\Gamma^{\sB}(\ln |v_1|,\ln f)\\
&\hspace{15mm}+4\Gamma^{\sB}(\ln f) +4\Gamma^{\sB}(\ln |u_1|,\ln |v_1|)\big)\Gamma^{\sF}(u_2,v_2)\\
&\hspace{4mm}+\textstyle{\frac{u_1^2}{f^2}}\big(-8\Gamma^{\sB}(\ln |u_1|,\ln f)+4\Gamma^{\sB}(\ln f)\\
&\hspace{15mm}+4\Gamma^{\sB}(\ln |u_1|)\big)\Gamma^{\sF}(u_2)+\textstyle{\frac{v_1^2}{f^2}}\big(-8\Gamma^{\sB}(\ln |v_1|,\ln f)+...\big)\Gamma^{\sF}(v_2)\\
%&=\textstyle{\frac{8u_1v_1}{f^2}}\Gamma^{\sB}(\ln f-\ln |u_1|,\ln f-\ln |v_2|)\Gamma^{\sF}(u_2,v_2)\\
%&\hspace{4mm}+\textstyle{\frac{4u_1^2}{f^2}}\Gamma^{\sB}(\ln f-\ln |u_1|)\Gamma^{\sF}(u_2)+\textstyle{\frac{4v_1^2}{f^2}}\Gamma^{\sB}(\ln f-\ln |v_1|)\Gamma^{\sF}(v_2)\\
&=\textstyle{\frac{8}{f^2}}\left\langle\nabla\ln (\textstyle{\frac{f}{|u_1|}}),\nabla\ln (\textstyle{\frac{f}{|v_1|}})\right\rangle|_p\Gamma^{\sF}(u^p,v^p)(x)\\
&\hspace{4mm}+\textstyle{\frac{4}{f^2}}\left|\nabla\ln (\textstyle{\frac{f}{|u_1|}})\right|^2_p\Gamma^{\sF}(u^p)(x)+\textstyle{\frac{4}{f^2}}\left|\nabla\ln (\textstyle{\frac{f}{|v_1|}})\right|^2_p\Gamma^{\sF}(v^p)(x)
\end{align*}
We choose an orthonormal basis $(e_i)_{1,\dots,d}$ with respect to the Riemannian metric at $TB_p$ and write 
$$\nabla \ln (\textstyle{\frac{f}{|u_1|}})|_p=\sum_{i=1}^{d}a^ie_i\hspace{2mm}\mbox{ and }\hspace{2mm}\nabla \ln (\textstyle{\frac{f}{|v_1|}})|_p=\sum_{i=1}^{d}c^ie_i.$$
Then we obtain
\begin{align}\label{anotherlabel}
\textstyle{\frac{f^2(p)}{4}}(II)&=\sum_{i}^d 2a^ic^i\Gamma^{\sF}(u^p,v^p)(x)+\sum_{i}^d(a^i)^2\Gamma^{\sF}(u^p)(x)+\sum_{i}^d(c^i)^2\Gamma^{\sF}(v^p)(x)\nonumber\\
&=\sum_{i}^d\Gamma^{\sF}(a^iu^p+c^iv^p)(x)\geq 0 \hspace{6pt}\mbox{ for $\m_{\sC}$-almost every $(p,x)$}
\end{align}
since $  \Gamma^{\sF}\geq 0$ $\m_{\sF}$-a.e. In the case where $v_1(p)=0$ and $u_1(p)\neq 0$, we get
\begin{align*}
(II)&=\textstyle{\frac{4}{f^2}}\left\langle\nabla\ln (f/ |u_1|),\nabla v_1\right\rangle|_p\Gamma^{\sF}(u^p,v_2)(x)\\
&\hspace{5mm}+\textstyle{\frac{4}{f^2}}\left|\nabla\ln (f/ |u_1|)\right|^2_p\Gamma^{\sF}(u^p)(x)+\textstyle{\frac{4}{f^2}}\left|\nabla v_1\right|^2_p\Gamma^{\sF}(v_2)(x)
\end{align*}
and when we set $\nabla v_1|_p=\sum_i^d\alpha^i e_i$, similar as before we obtain that
\begin{align}
\textstyle{\frac{f^2(p)}{4}}(II)&=\sum_{i}^d\Gamma^{\sF}(a^iu^p+\alpha^iv_2)(x)\geq 0.
\end{align}
In the same way we can deal with the other cases. If we would consider an arbitrary $u\in C^{\infty}_0(\hat{B})\otimes \mathcal{A}^{\sF}$ of the form $\sum_{j}^k u_{1,j}\otimes u_{2,j}=\sum_j^ku_j$, $\textstyle{\frac{f^2}{4}}(II)$ would take the form
$\textstyle{\sum_{i}^d\Gamma^{\sF}(\sum_{j}^ka^i_ju_{j}^p)\geq 0}$ and all the other calculations are the same.
It follows in any case that
\begin{align*}
 2\Gamma_2(u)(p,x)&\geq 2\Gamma_2^{\sB}(u^x)(p)+\textstyle{\frac{1}{f^4(p)}}2\Gamma_2^{\sF}(u^p)(x)&\nonumber\\
&\hspace{5mm}+\textstyle{\frac{4}{f^3(p)}}\Gamma^{\sB}(f,u^x)L^{\sF}(u^p)-\textstyle{\frac{2}{f^3(p)}}\left(L^{\sB} f(p)+\textstyle{\frac{N-1}{f(p)}}\Gamma^{\sB}(f)(p)\right)\Gamma^{\sF}(u^p)(x)\hspace{6pt}\m_{\sC}\mbox{-a.e.}
\end{align*}
for any $u\in C^{\infty}_0(\hat{B})\otimes \mathcal{A}^{\sF}$. Because of (\ref{hmmm}) we can see that this estimate is sharp and becomes an equality if $u=f\otimes u_2$.
From (\ref{estimatos}) we have 
\begin{align}\label{13}
\Gamma_2^{\sB,f^{\sN}}(u)\geq (d+N-1)K\Gamma^{\sB}(u)+\frac{1}{d}\left(L^{\sB}u\right)^2+\frac{1}{N}\big(\textstyle{\frac{N}{f}}\Gamma^{\sB}(f,u)\big)^2
\end{align}
for any function $u\in C^{\infty}_0(\hat{B})$.
Now we apply the curvature-dimension conditions for $\mathcal{E}^{\sF}$ and $\mathcal{E}^{\sB,f^{\sN}}$, inequality (\ref{13}), 
and the assumptions on $f$. First, we see that $$L^{\sB}f+\frac{N-1}{f}\Gamma^{\sB}(f)\leq -dKf+\frac{N-1}{f}\left(K_{\sF}-Kf^2\right) \hspace{5pt}\mbox{ everywhere in }B.$$
Then it follows that
\begin{align*}
2\Gamma^{\sC}_2(u)&\geq 2 (d+N-1)K\Gamma^{\sB}(u^x)+2\frac{1}{d}\left(L^{\sB}u^x\right)^2+2\frac{1}{N}\big(\textstyle{\frac{N}{f}}\Gamma^{\sB}(f,u^x)\big)^2&\\
&\hspace{5mm}+\textstyle{\frac{2}{f^4}}\Big( (N-1)K_{\sF} \Gamma^{\sF}(u^p)+\textstyle{\frac{1}{N}}\big(L^{\sF}u^p\big)^2\Big)&\nonumber\\
&\hspace{5mm}+\textstyle{\frac{2}{f^3}}\Gamma^{\sB}(f,u^x)2L^{\sF}(u^p)+\textstyle{\frac{2}{f^3}}\Big(dKf-\textstyle{\frac{N-1}{f}}\left(K_{\sF}-Kf^2\right)\Big)\Gamma^{\sF}(u^p)&\\
&= 2 \Big( (d+N-1)K\Gamma^{\sB}(u_1)+\frac{1}{d}\left(L^{\sB}u^x\right)^2+\frac{1}{N}\big(\textstyle{\frac{N}{f}}\Gamma^{\sB}(f,u^x)\big)^2\Big)+\textstyle{\frac{2}{f^2}}(N+d-1)K \Gamma^{\sF}(u^p)\\
&\hspace{5mm}+\textstyle{\frac{1}{N}}\textstyle{\frac{2}{f^4}}\big(L^{\sF}u^p\big)^2+\textstyle{\frac{1}{N}}\textstyle{\frac{2N}{f^3}}\Gamma^{\sB}(f,u^x)2L^{\sF}(u^p)&\\
&= 2 \big( (d+N-1)K\Gamma^{\sB}(u^x)+\frac{1}{d}\left(L^{\sB}u^x\right)^2\big)+\textstyle{\frac{2}{f^2}}(N+d-1)K \Gamma^{\sF}(u^p)&\\
&\hspace{5mm}+ \frac{1}{N}\Big(2\big(\textstyle{\frac{N}{f}}\Gamma^{\sB}(f,u^x)\big)^2+{\frac{2}{f^4}}\left(L^{\sF}u^p\right)^2+\frac{2N}{f^3}\Gamma^{\sB}(f,u^x)2L^{\sF}(u^p)\Big)&\\
&= 2 \Big( (N+d-1)\Gamma^{\sB}(u^x)  + 
\frac{1}{d}(L^{\sB}u^x)^2\Big)&\\
&\hspace{5mm}+\textstyle{\frac{2}{f^2}}(N+d-1)K \Gamma^{\sF}(u^p)
+ \textstyle{\frac{2}{N}}\Big(\textstyle{\frac{N}{f}}\Gamma^{\sB}(f,u^x)+\textstyle{\frac{1}{f^2}}L^{\sF}u^p\Big)^2\hspace{5pt}\m_{\sC}\mbox{-a.e.}\ .&
\end{align*}
We apply the following elementary equality
\begin{equation}
\textstyle{\frac{1}{d}a^2+\frac{1}{N}b^2= \frac{1}{N+d}\left(a+b\right)^2}+\frac{d}{(N+d)N}(b-\frac{N}{d}a)^2\label{elem}
\end{equation}for all $d,N\geq 1$ and for all $a,b\in\mathbb{R}$.
Hence
\begin{align*}
2\Gamma^{\sC}_2(u)&\geq 2 (N+d-1)K\Gamma^{\sB}(u^x)+\frac{1}{f^2}2(N+d-1)K \Gamma^{\sF}(u^p)&\\
&\hspace{5mm}+ \frac{2}{d}\left(L^{\sB} u^x\right)^2+\frac{2}{N}\Big(\textstyle{\frac{N}{f}}\Gamma^{\sB}(f,u^x)+\textstyle{\frac{1}{f^2}}L^{\sF}u^p\Big)^2&\\
&\geq 2(N+d-1)K\Gamma^{\sC}(u)+ \frac{2}{N+d}\left(L^{\sB} u^x+\textstyle{\frac{N}{f}}\Gamma^{\sB}(f,u^x)+\textstyle{\frac{1}{f^2}}L^{\sF}u^p\right)^2\hspace{5pt}\m_{\sC}\mbox{-a.e.}\ .&
\end{align*}
So we have desired inequality
for any $u\in C^{\infty}_0(\hat{B})\otimes \mathcal{A}^{\sF}$.
\end{proof}
\begin{theorem}\label{maintheorem2}
Let $\mathcal{E}^{\sF}$ be a regular and strongly local Dirichlet form and let $\mathcal{A}^{\sF}$ be an admissible algebra for $\mathcal{E}^{\sF}$. 
Assume the $(K,N)$-cone $\mathcal{E}^{\sC}=I_K\times^{\sN}_{\sin_{\sK}}\mathcal{E}^{\sF}$ 
satisfies a $\Gamma_2$-estimate of curvature $NK$ and dimension $N+1$ for $K\in \mathbb{R}$ and $N\geq 1$  on $C_0^{\infty}(\hat{I}_{\sK})\otimes\mathcal{A}^{\sF}$. 
Then $\mathcal{E}^{\sF}$ satisfies $BE(N-1,N)$.
\end{theorem}
\begin{proof}
We assume a $\Gamma_2$-estimate for $(L^{\sC},\mathcal{A}^{\sC})$ and deduce the curvature dimension condition for $(L^{\sF},\mathcal{A}^{\sF})$. 
We have to show that the $\Gamma_2$-estimate holds pointwise a.e. in $F$ for any $u_2\in \mathcal{A}^{\sF}$. 
From calculations in the previous proof we have the identity (\ref{hmmm}) for $\Gamma^{\sC}_2$ in the following form
\begin{align}
\Gamma^{\sC}_2(u_1\otimes u_2)
%&=\Gamma_2^{I_{\sK},\sin_{\sK}^{\sN}}(u_1)u_2^2+\frac{u_1^2}{\sin_{\sK}^4}\Gamma_2^{\sF}(u_2)&\nonumber\\
%&\hspace{5mm}+\frac{u_1}{\sin_{\sK}^3}\Gamma^{I_{\sK}}(\sin_{\sK},u_1) 2L^{\sF}(u_2)u_2-\frac{u_1^2}{\sin_{\sK}^3}\left(L^{I_{\sK}} \sin_{\sK}+\textstyle{\frac{N-1}{\sin_{\sK}}\Gamma^{I_{\sK}}(\sin_{\sK})}\right)\Gamma^{\sF}(u_2)&\nonumber\\
%&\hspace{5mm}+\Big(2\Gamma^{I_{\sK}}(u_1)\frac{1}{\sin_{\sK}^2}-\frac{4u_1}{\sin_{\sK}^3}\Gamma^{I_{\sK}}(\sin_{\sK},u_1)+\frac{2u^2_1}{\sin_{\sK}^4}\Gamma^{I_{\sK}}(\sin_{\sK})\Big)\Gamma^{\sF}(u_2).&\nonumber\\
&=\Gamma_2^{I_{\sK},\sin_{\sK}^{\sN}}(u_1)u_2^2+\frac{u_1^2}{\sin_{\sK}^4}\Gamma_2^{\sF}(u_2)&\nonumber\\
&\hspace{5mm}+\frac{u_1}{\sin_{\sK}^3}\cos_{\sK} \ u_1' 2L^{\sF}(u_2)u_2-\frac{u_1^2}{\sin_{\sK}^3}\left(-K\sin_{\sK}+\textstyle{\frac{N-1}{\sin_{\sK}}\cos_{\sK}^2}\right)\Gamma^{\sF}(u_2)&\nonumber\\
&\hspace{5mm}+\Big({\frac{2}{\sin_{\sK}^2}}(u_1')^2-{\frac{4u_1}{\sin_{\sK}^3}}\cos_{\sK}\ u_1'+{\frac{2u^2_1}{\sin_{\sK}^4}}\cos_{\sK}^2\Big)\Gamma^{\sF}(u_2).&\label{A}
\end{align}
for any $u_1\in C^{\infty}_0(\hat{B})$ and any $u_2\in\mathcal{A}^{\sF}$ $\m_{\sC}$-a.e. in $I_{\sK}\times F$.
We consider some open set $U\subset\hat{I}_{\sK}$ and we choose $u_1\in C_0\su{\infty}(\hat{I}_{\sK})$ such that $u_1=\sin_{\sK}$ on $U$.
%We consider some $r_0\in \hat{I}_{\sK}$ and
% choose $u_1\in C^{\infty}_0(\hat{I}_{\sK})$ such that $u_1=\sin_{\sK}$ in a neighborhood $U$ of $r_0$. 
%Since (\ref{A}) holds $\m_{\sC}$-almost everywhere, we can find $r\in U$ such that (\ref{A}) holds at $(r,x)$ for $\m_{\sF}$-almost every $x$.
%We fix such an $r$ and $(r,x)$.
By the special choice of $u_1$ (\ref{A}) reduces at $(r,x)\in U\times F$ to
\begin{align}
\Gamma_2\su{\sC}(u_1\otimes u_2)&=\Gamma_2^{I_{\sK},\sin_{\sK}^{\sN}}(u_1)u_2^2+\frac{u_1^2}{\sin_{\sK}^4r}\Gamma_2^{\sF}(u_2)&\nonumber\\
&\hspace{5mm}+\frac{u_1}{\sin_{\sK}^3}\cos_{\sK} \cdot u_1' \cdot 2L^{\sF}(u_2)u_2-\frac{u_1^2}{\sin_{\sK}^3}\left(-K\sin_{\sK}+\ \textstyle{\frac{N-1}{\sin_{\sK}}\cos_{\sK}^2}\right)\Gamma^{\sF}(u_2)&\nonumber
\end{align}
for $\m_{\sC}$-a.e. $(r,x)\in U\times F$.
In the case of $(I_{\sK},\sin_{\sK}^{\sN})$ the identity (\ref{estimatos}) in the proof of Proposition \ref{keineahnung} implies
\begin{align}
\Gamma_2^{I_{\sK},\sin_{\sK}^{\sN}}(u_1)&=(u_1'')^2+\textstyle{\frac{N}{\sin_{\sK}}}K\sin_{\sK}\cdot(u_1')^2+\hspace{1pt}{\frac{N}{\sin_{\sK}^2}}\left(\cos_{\sK}\cdot \ u_1'\right)^2\nonumber
\end{align}
everywhere in $\hat{I}_{\sK}$ for $u_1\in C_0\su{\infty}(\hat{I}_{\sK})$.
Hence, we obtain
\begin{align}
\frac{u_1^2}{\sin_{\sK}^4r}\Gamma_2^{\sF}(u_2)&=\Gamma\su{\sC}_2(u_1\otimes u_2)-\Gamma_2^{I_{\sK},\sin_{\sK}^{\sN}}(u_1)u_2^2\nonumber\\
&\hspace{5mm}-\frac{u_1}{\sin_{\sK}^3}\cos_{\sK} \cdot u_1' \cdot 2L^{\sF}(u_2)u_2+\frac{u_1^2}{\sin_{\sK}^3}\left(-K\sin_{\sK}+\ \textstyle{\frac{N-1}{\sin_{\sK}}\cos_{\sK}^2}\right)\Gamma^{\sF}(u_2)\nonumber\\
&= \Gamma_2\su{\sC}(u_1\otimes u_2)-(u_1'')^2u_2^2-\textstyle{\frac{N}{\sin_{\sK}}}K\sin_{\sK}\cdot(u_1')^2u_2^2-\hspace{1pt}{\frac{N'}{\sin_{\sK}^2}}\left(\cos_{\sK}\cdot \ u_1'\right)^2u_2^2\nonumber\\
&\hspace{5mm}-\frac{u_1}{\sin_{\sK}^3}\cos_{\sK} \cdot u_1' \cdot 2L^{\sF}(u_2)u_2+\frac{u_1^2}{\sin_{\sK}^3}\left(-K\sin_{\sK}+\ \textstyle{\frac{N-1}{\sin_{\sK}}\cos_{\sK}^2}\right)\Gamma^{\sF}(u_2)\nonumber
\end{align}
$\m_{\sC}$-a.e. in $U\times F$. On the other side the $\Gamma_2$-estimate for $\mathcal{E}^{\sC}$ gives for any $u_1\in C^{\infty}_0(\hat{B})$ and any $u_2\in\mathcal{A}^{\sF}$
\begin{align*}
\Gamma_2\su{\sC}(u_1u_2)&\geq NK\big((u_1')^2 u_2^2+\textstyle{\frac{u^2_1}{\sin_{\sK}^2}}\Gamma^{\sF}(u_2)\big)+ 
\frac{1}{N+1}\left(L^{I_{\sK}} u_1 u_2+\textstyle{\frac{N}{\sin_{\sK}}}\Gamma^{I_{\sK}}(\sin_{\sK},u_1) u_2+\textstyle{\frac{u_1}{\sin_{\sK}^2}}L^{\sF}u_2\right)^2
\end{align*}
$\m_{\sC}$-a.e.\ .
Since $u_1(r)=\sin_{\sK}r$ and $u_1'(r)=\cos_{\sK}r$, we get after some cancelations
\begin{align}\label{tasse}
\frac{1}{\sin_{\sK}^2}\Gamma_2^{\sF}(u_2)&\geq (N-1)(\cos_{\sK}^2+K\sin_{\sK}^2)\frac{1}{\sin_{\sK}^2}\Gamma^{\sF}(u_2)-\frac{\cos_{\sK}^2}{\sin_{\sK}^2}L^{\sF}(u_2)u_2\nonumber\\
&\hspace{15pt}-\left(-K\sin_{\sK}\right)^2u_2^2-\hspace{1pt}{\frac{N}{\sin_{\sK}^2}}\left(\cos_{\sK}\cdot \ \cos_{\sK}\right)^2u_2^2\nonumber\\
&\hspace{15pt}+ \frac{1}{N+1}\left(-K\sin_{\sK}u_2+\frac{N}{\sin_{\sK}}\cos_{\sK}^2u_2+\frac{1}{\sin_{\sK}}L^{\sF}u_2\right)^2\ \ \m_{\sC}\mbox{-a.e. in }U\times F.
%&= (N-1)\frac{1}{\sin_{\sK}^2}\Gamma^{\sF}(u_2)-\frac{\cos_{\sK}^2}{\sin_{\sK}^2}L^{\sF}(u_2)u_2\nonumber\\
%&\hspace{15pt}-\left(-K\sin_{\sK}\right)^2u_2^2-\hspace{1pt}{\frac{N}{\sin_{\sK}^2}}\left(\cos_{\sK}\cdot \ \cos_{\sK}\right)^2u_2^2\nonumber\\
%&\hspace{15pt}+ \frac{1}{N+1}\left(-K\sin_{\sK}u_2+\frac{N}{\sin_{\sK}}\cos_{\sK}^2u_2+\frac{1}{\sin_{\sK}}L^{\sF}u_2\right)^2.
\end{align}We consider the last term on right side in (\ref{tasse}) in more detail. From the identity (\ref{elem}) we deduce
\begin{align*}
&\frac{1}{N+1}\left(-K\sin_{\sK}u_2+\frac{N}{\sin_{\sK}}\cos_{\sK}^2u_2+\frac{1}{\sin_{\sK}}L^{\sF}u_2\right)^2\\
&=\left(-K\sin_{\sK}u_2\right)^2+\frac{1}{N}\left(\frac{N}{\sin_{\sK}}\cos_{\sK}^2u_2\right)^2+\frac{1}{N}\left(\frac{1}{\sin_{\sK}}L^{\sF}u_2\right)^2+\frac{2}{\sin_{\sK}^2}\cos_{\sK}^2 u_2L^{\sF}u_2\\
&\hspace{15pt}-\frac{1}{(N+1)N}\underbrace{\left(\frac{N}{\sin_{\sK}}\cos_{\sK}^2u_2+\frac{1}{\sin_{\sK}}L^{\sF}u_2+{N}K\sin_{\sK}u_2\right)^2}_{=\frac{1}{\sin_{\sK}^2}\left(L^{\sF}u_2+Nu_2\right)^2}
\end{align*}
It follows
\begin{align}\label{bluuub}
\Gamma_2^{\sF}(u_2)&\geq 
(N-1)\Gamma^{\sF}(u_2)+\frac{1}{N}\left(L^{\sF}u_2\right)^2-\frac{1}{(N+1)N}\left(L^{\sF}u_2+Nu_2\right)^2
\end{align}
at $x$ for $\m_{\sC}$-almost every $(r,x)\in U\times F$. But since (\ref{bluuub}) does not depend on $r\in U$ anymore, we can conclude it holds for $\m_{\sF}$-a.e. $x\in F$.
\smallskip

We fix such a $x$. $\mathcal{E}^{\sF}$ is strongly local. So we can add constants without affecting $L^{\sF}$, $\Gamma^{\sF}$ and $\Gamma_2^{\sF}$. Thus we can replace $u_2$ by $\tilde{u}_2:=u_2+C$, where
$C=-u_2(x)-\frac{1}{N}L^{\sF}u_2(x)$. Then $(L^{\sF}u_2)(x)=(L^{\sF}\tilde{u}_2)(x)$, $\Gamma(u_2)=\Gamma(\tilde{u}_2)$ and $\Gamma_2^{\sF}(u_2)=\Gamma(\tilde{u}_2)$ and  $\left(L^{\sF}\tilde{u}_2+N\tilde{u}_2\right)^2$ vanishes at $x$. Hence, we obtain the desired estimate for $u_2$ at $\m_{\sF}$-almost every $x$ and we obtain the condition $BE(N\!-\!1,N)$ for $F$.
\end{proof}

\subsection{A result on essentially self-adjoint operators}
\begin{definition}[direct sum]
Suppose $(\mathcal{H}_i,\left\|\cdot\right\|_{\mathcal{H}_i})_{i\in\mathbb{N}}$ is a sequence of Hilbert spaces. Its direct sum is a Hilbert space that is given by
\begin{align*}
 \mathcal{H}=\bigoplus_{i=1}^{\infty}\mathcal{H}_i=
 \left\{v:=(v_i)_{i\in\mathbb{N}}:v_i\in \mathcal{H}_i \mbox{ such that }\left\|v\right\|_{\mathcal{H}}^2:=\sum_{i=1}^{\infty}\left\|v_i\right\|_{\mathcal{H}_i}^2< \infty\right\}
\end{align*}
where the inner product is $\sum_{i=1}^{\infty}(v_i,u_i)_{\mathcal{H}_i}=\left(v,u\right)_{\mathcal{H}}$.
Additionally, we introduce
\begin{align*}
\sum_{i=1}^{\infty}\mathcal{H}_i=\Big\{v:=(v_i)_{i\in\mathbb{N}}:v_i\in \mathcal{H}_i \mbox{ and }v_i=0\mbox{ except for finitely many }i\in\mathbb{N}\Big\}.
\end{align*}
\end{definition}
\begin{theorem}\label{firstmain}
Let $\mathcal{E}^{\sF}$ be a regular and strongly local Dirichlet form. 
Assume the spectrum of $-L^{\sF}$ is discrete and its first positive eigenvalue satisfies $\lambda_1\geq N$.
Let $E_i\subset D^2(L^{\sF})$ be the eigenspace that corresponds to the $i$th eigenvalue $\lambda_i$. 
Let $\mathcal{E}^{\sC}=I_K\times^{\sN}_{\sk}\mathcal{E}^{\sF}$ be the $(K,N)$-cone for $K\in\mathbb{R}$ and 
$N\geq 1$ over $\mathcal{E}^{\sF}$ and let $L^{\sC}$ be the corresponding self-adjoint operator. 
Let $\mathcal{A}$ be dense in the domain of $L^{I_{\sK},\sin_{\sK}^{\sN}}$.
Then
\begin{align*}
 \Xi=\left[\mathcal{A}\otimes E_0\right]\oplus\sum_{i=1}^{\infty}C^{\infty}_0(\hat{I}_K)\otimes E_i
\end{align*}
is dense in the domain of $L^{\sC}$ with respect to the graph norm.
\end{theorem}
\begin{proof}
Since $L^{F}$ has a discrete spectrum, there is a spectral decomposition of $L^2(m_{\sF})$ with respect to its eigenvalues.
\begin{align*}
L^2(F,m_{\sF})
%=\bigoplus_{i=0}^{\infty}E_i
=E_0\oplus\bigoplus_{i=1}^{\infty}E_i
%=E_0\oplus \underbrace{\left(L^2(F,m_{\sF})/E_0\right)}_{=:E^{\perp}}
=E_0\oplus E_{\perp}.
\end{align*}
We denote the restriction of $\mathcal{E}^{\sF}$ and $(\cdot,\cdot)_{L^2(\m_{\sF})}$ to $E_{\perp}$ by $\mathcal{E}^{\sF}_{\perp}=\mathcal{E}^{\sF}|_{E_{\perp}\times E_{\perp}}$ and $(\cdot,\cdot)_{E_{\perp}}$
respectively.
Then there is a densely defined, self-adjoint operator $L^{\sF}_{\perp}$ on $(E_{\perp},(\cdot,\cdot)_{E_{\perp}})$ that corresponds to $\mathcal{E}_{\perp}^{\sF}$. It is easy to see 
that $D^2(L^{\sF}_{\perp})=D^2(L^{\sF})\cap E_{\perp}$,
$L^{\sF}_{\perp}u_{\perp}=L^{\sF}u_{\perp}$ and 
 $\mbox{spec}\ L^{\sF}_{\perp}=\mbox{spec}\ L^{\sF}\backslash \left\{\lambda_0\right\}$.
Also $L^2(C,m_{\sC})$ can be decomposed orthogonally into
%\begin{align*}
$
L^2(C,m_{\sC})
%&=\overline{L^2(I_K,\sin^{\scriptscriptstyle{N}}_{\sK} dr)\otimes L^2(F,m_{\sF})}^{L^2(\m_{\sC})}
=U_0\oplus U_{\perp}
$
%\end{align*}
where 
$U_0= L^2(I_K,\sin^{\scriptscriptstyle{N}}_{\sK} dr)\otimes E_0$ and $U_{\perp}
%=\bigoplus_{i=1}^{\infty}U_i
=L^2(I_K,\sin^{\scriptscriptstyle{N}}_{\sK} dr)\otimes E_{\perp}. 
$ 
For $u=u_1\otimes u_2\in U_0$ and $v=v_1\otimes v_2\in U^{\perp}$ with $u_1,v_1\in L^2(I_K,\sin^{\scriptscriptstyle{N}}_{\sK} dr)$, $u_2\in E_0$ and $v_2\in E_{\perp}$ we have
\begin{align*}
\mathcal{E}^{\sC}(u,v)=\mathcal{E}^{I_K,f^{\sN}}(u_1,u_2)\underbrace{\left( u_2,v_2\right)_{L^2(m_{\sF})}}_{=0}+\left( u_1,u_2\right)_{L^2(f^{N-2}d\m_{\sF})}\underbrace{\mathcal{E}^{\sF}(u_2,v_2)}_{=0}=0
\end{align*}
since $\mathcal{E}^{\sF}(u_2,v_2)=-(L^{\sF}u_2,v_2)_{L^2(\m_{\sF})}=0$ and $E_0\perp E_{\perp}$ in $L^2(\m_{\sF})$.
Thus we can decompose $\mathcal{E}^{\sC}$ orthogonally as follows
\begin{align*}
 \mathcal{E}^{\sC}=\mathcal{E}^{\sC}|_{U_0\times U_0}+\mathcal{E}^{\sC}|_{U^{\perp}\times U^{\perp}}=:\mathcal{E}^{\sC}_{0}+\mathcal{E}^{\sC}_{\perp}.
\end{align*}
%For $u_1\otimes u_2\in U^{\perp}$ the form $\mathcal{E}^{\sC}_{\perp}$ is given by
%\begin{align*}
%\mathcal{E}^{\sC}_{\perp}(u_1\otimes u_2)=\mathcal{E}^{I_K}(u_1)\left\|u_2\right\|^2_{L^2(m_{\sF})}+\left\|u_1\right\|^2_{L^2(\sin^{N-2}dr)}\mathcal{E}^{\sF}_{\perp}(u_2)
%\end{align*}
%and for $u_1\otimes u_2\in U_0$ the form $\mathcal{E}^{\sC}_0$ is given by 
%\begin{align*}
%\mathcal{E}^{\sC}_{0}(u_1\otimes u_2)=\mathcal{E}^{I_K}(u_1)\left\|u_2\right\|^2_{L^2(\m_{\sF})}.
%\end{align*}
One checks that $u=u_{0}+u_{\perp}\in D(\mathcal{E}^{\sC})$ if and only if $u_{0}\in D(\mathcal{E}^{\sC}_{\top})$ and $u_{\perp}\in D(\mathcal{E}^{\sC}_{\perp})$ and that 
$u=u_{0}+u_{\perp}\in D^2(L^{\sC})$ if and only if $u_{0}\in D^2(L^{\sC}_0)$ and $u_{\perp}\in D^2(L^{\sC}_{\perp})$ 
%where for example
%\begin{align*}
%D^2(L^{\sC}_{\perp})=\left\{u_{\perp}\in D(\mathcal{E}^{\sC}_{\perp}): L^{\sC}_{\perp}u_{\perp}\in U^{\perp}\right\}
%\end{align*}
and we have $L^{\sC}=L^{\sC}_{0}+L^{\sC}_{\perp}$.
%In the following we will analyse $L^{\sC}_{0}$ and $L^{\sC}_{\perp}$ separately. \smallskip\\
%
%
%
%
%
$L^{\sC}_{\perp}$ is a densely defined operator on $U_{\perp}$ with $$D^2(L^{\sC}_{\perp})=D^2(L^{\sC})\cap U_{\perp}=D^2(L^{\sC})\cap L^2(I_K,\sin^{\scriptscriptstyle{N}}_{\sK} dr)\otimes E_{\perp}.$$ 
$C^{\infty}_0(\hat{I}_K)\otimes D^2(L^{\sF})$ is a subset of $D^2(L^{\sC})$, hence, $C^{\infty}_0(\hat{I}_K)\otimes D^2(L^{\sF}_{\perp})$ is a subset of $D^2(L^{\sC}_{\perp})$.
For $u_1\in C^{\infty}_0(\hat{I}_K)$ and $u_2 \in D^2(L^{\sF}_{\perp})$ we have 
\begin{align*}
L^{\sC}_{\perp}(u_1\otimes u_2) = L^{I_K,\sin^{\sN}}u_1\otimes u_2 + \frac{u_1}{\sin^2}\otimes L^{\sF}_{\perp}u_2.
\end{align*}
For all $i\in\mathbb{N}\backslash \left\{0\right\}$ we set 
$
\tilde{U}_i=U_i\cap C^{\infty}_0(\hat{I}_K)\otimes D^2(L^{\sF}_{\perp})=C^{\infty}_0(\hat{I}_K)\otimes E_i
$
and consider the restriction of $L^{\sC}_{\perp}$ to $\tilde{U}_i$
\begin{align*}
L^{\sC}_{\perp}|_{\tilde{U}_i}=
L^{\sC}|_{C^{\infty}_0(\hat{I}_K)\otimes E_i}&=\big(L^{I_K,\sin^{\scriptscriptstyle{N}}_{\sK}}+\textstyle{\frac{-\lambda_i}{\sin^2}}\big)|_{C^{\infty}_0(\hat{I}_K)}\otimes I|_{E_i}\\
&=\left(L^{I_K}+\textstyle{\frac{N}{\sin}\Gamma^{I_K}(\sin,\cdot)}-\textstyle{\frac{\lambda_i}{\sin^2}}\right)|_{C^{\infty}_0(\hat{I}_K)}\otimes I|_{E_i}=:L^{\sC}_{\perp,i}
\end{align*} 
$L^{\sC}_{\perp,i}$ is a densely defined operator on $L^2(\sin^{\sN}rdr)\otimes E_i$ and we define
$$\sum_{i_1}^{\infty}L^{\sC}_{\perp,i}=:\tilde{L}^{\sC}_{\perp}\hspace{4pt}\mbox{ on }\hspace{4pt}\sum_{i=1}^{\infty}C^{\infty}_0(\hat{I}_K)\otimes E_i.$$
We will show that $\tilde{L}^{\sC}_{\perp}$ is essentially self-adjoint. Then, the unique self adjoint extension of $\tilde{L}^{\sC}_{\perp}$ has to coincide with $L^{\sC}_{\perp}$.
In particular, $\sum_{i=1}^{\infty}C^{\infty}_0(\hat{I}_K)\otimes E_i$ is dense in $D^2(L^{\sC}_{\perp})$ with respect to the graph norm.
It is sufficient 
to show that the operator $(L^{I_K}+{\textstyle\frac{N}{\sin}}\Gamma(\sin,\cdot)-{\textstyle\frac{1}{\sin^2}}\lambda_i)|_{C^{\infty}_0(\hat{I}_K)}$ is essentially self-adjoint for every $\lambda_i\in\spec L^{\sF}$ (see \cite[ch X, problem 1.a]{reedsimon}).
\smallskip

We follow the proof of Theorem X.11 in \cite{reedsimon}.
Consider the unitary transformation 
\begin{align*}
U:L^2(I_K,\sin^{\scriptscriptstyle{N}}_{\sK} dr)\rightarrow L^2(I_K,dr),\ \ \ 
\phi(r)\mapsto \sqrt{\sin^{\scriptscriptstyle{N}}_{\sK}}\phi(r).
\end{align*}
$C^{\infty}_0((0,\infty))$ is invariant under $U$ and $L^{I_K,\sin^{\scriptscriptstyle{N}}_{\sK}}-\textstyle{\frac{\lambda_i}{\sin^2}}$ takes the form
\begin{align*}
U\left(\frac{d^2}{dr^2}+\frac{N}{\sin r}\frac{d\sin r}{dr}\frac{d}{dr}-\frac{\lambda_i}{\sin^2r}\right)U^{-1}=\frac{d^2}{dr^2}-\underbrace{\left(\frac{N^2}{4}\cos^2r-\frac{N}{2}-\lambda_i\right)\frac{1}{\sin^2r}}_{V(r)}.
\end{align*}
We get a Schr\"odinger-type operator defined on  $C^{\infty}_0(\hat{I}_K)$. The question, if such an operator is essentially self-adjoint, is a classical problem from quantum mechanics. It
was answered by Herman Weyl who analyzed the solutions of the following ordinary differential equation
$
 -\phi''+V\phi=\lambda\phi.
$
One says that $V(r)$ is in the limit circle case at $r\in \partial I_K$ (we assume that $\partial I_K=\left\{0,\infty\right\}$ for $K\leq 0$) if for some $\lambda$, all solutions are locally square integrable around $r$. Otherwise we say $V$ is
in the limit point case at $r$.
\begin{theorem}[Weyl's limit point-limit circle criterion]
Let $V$ be a continuous  real-valued function on $\hat{I}_K$. Then, the operator $H=-d^2/dr^2+V$ is essentially self-adjoint on $C^{\infty}_0(\hat{I}_K)$ if and only if $V$ is in the limit point case at any $r\in \partial I_K$.
\end{theorem}
For the particular case that we consider, the limit point case at $\infty$ for $K\leq 0$ is easy to check (see Theorem X.8 in \cite{reedsimon} and the next corollary). 
The case $r=0$ for $K\leq 0$ and $r=0$ and $r=\pi/\sqrt{K}$ for $K>0$ follows from the next theorem.
\begin{theorem}
Let $V$ be continuous on $\hat{I}_{K}$ and positive near $r_0$. If $V(r)\geq \frac{3}{4}\frac{1}{(r-r_0)^2}$ for $r\rightarrow r_0\in\partial I_K\backslash \left\{\infty\right\}$ then $H=-d^2/dr^2+V(r)$ is in the limit point case at $r_0$. 
If for some $\epsilon>0$, $V(r) \leq (\frac{3}{4}-\epsilon)\frac{1}{(r-r_0)^2}$ near $r_0$, then $H$ is in the limit circle case at $r_0$.
\end{theorem}
\proof $\rightarrow$ Theorem X.10 in \cite{reedsimon}.
\smallskip
\\
So far we did not use the spectral gap of $-L^{\sF}$. 
We have $\lambda_i\geq N$ for any $\lambda_i\in\spec -L^{\sF}$.
But then, the operator $-\frac{d^2}{dr^2} - \frac{N\cos}{\sin}\frac{d}{dr} + \frac{\lambda_i}{\sin^2}$ is essentially self-adjoint on $C^{\infty}_0((0,\infty))$ since for $r\rightarrow 0$
\begin{align*}
\left(\frac{N^2}{4}\cos^2r-\frac{N}{2}-\lambda_i\right)\frac{1}{\sin^2r}&\sim\left(\frac{N(N-2)}{4}+\lambda\right)\frac{1}{r^2}\geq \left(\frac{N(N-2)}{4}+N\right)\frac{1}{r^2}
%\geq\left(\frac{N(N+2)}{4}\right)\frac{1}{r^2}
\geq \frac{3}{4r^2}
\end{align*}
where the last inequality holds if $N\geq 1$. Analogously for $r\rightarrow \pi$.
\smallskip
\\
For $\mathcal{E}^{\sC}_{0}$ we can not follow this strategy (see the next Remark).
But since $\mathcal{A}$ is assumed to be dense in the domain of $L^{I_{\sK},\sin_{\sK}^{\sN}}$, we obtain that
\begin{align*}
 \Xi&=\left[\mathcal{A}\otimes E_0\right]\oplus\sum_{i=1}^{\infty}C^{\infty}_0(\hat{I}_K)\otimes E_i\subset D^2(L^{\sC}_{\top})\oplus D^2(L^{\sC}_{\perp})=D^2(L^{\sC})
\end{align*}
is dense in $D^2(L^{\sC})$.
\end{proof}
\begin{remark}\label{delicate}
For
$u_1\otimes c\in C^{\infty}_0(\hat{I}_{\sK})\otimes \mathbb{R}$ we have that
$L^{\sC}_{\top}\left(u_1\otimes c\right)= L^{I_K,\sin_{\sK}^{\sN}}u_1 \otimes c.$
But in general, $L^{I_K,\sin_{\sK}^{\sN}}|_{C^{\infty}_0(\hat{I}_K)}$ is not essentially self-adjoint. More precisely, under the transformation $U$ it becomes
\begin{align*}
 \hat{L}=L^{I_K,\sin_{\sK}^{\sN}}|_{C^{\infty}_0(\hat{I}_K)}=\frac{d^2}{dr^2}-\left(\frac{N^2}{4}\cos^2r-\frac{N}{2}\right)\frac{1}{\sin^2r}
\end{align*}
We see that $$\left(\frac{N^2}{4}\cos^2-\frac{N}{2}\right)\frac{1}{\sin^2}\sim\frac{N(N-2)}{4}\frac{1}{r^2}\geq \frac{3}{4r^2}\hspace{10pt}\mbox{ if }N\geq 3$$ \\
for $r\rightarrow 0$ and $$\left(\frac{N^2}{4}\cos^2-\frac{N}{2}\right)\frac{1}{\sin^2}\leq \frac{3}{4r^2}-\epsilon\hspace{10pt}\mbox{ for some }\epsilon>0\mbox{ if }N<3$$ 
for $r\rightarrow 0$. Analogously for $r\rightarrow \pi$.
Hence, in the case $N\geq 3$ we can choose $C_0^{\infty}(\hat{I}_{\sK})$ as dense subset $\mathcal{A}$ in $D^2(L^{I_{\sK},\sin_{\sK}^{\sN}})$ since the closure of $C^{\infty}_0(\hat{I}_{\sK})$ is the only 
self-adjoint extension. 
On the other hand, in the case $1\leq N <3$ the operator $\hat{L}$ is not essentially self-adjoint. There are more than 
one self-adjoint extensions $A_{\alpha}$ of $\hat{L}$ and
\begin{align*}
D^2(\mbox{CL}(\hat{L}))\subsetneq D(A_{\alpha})\subset D((\hat{L})^*)
\end{align*}
where $\mbox{CL}(\hat{L})$ denotes the closure with respect the graph norm and $(\hat{L})^*$ is the adjoint operator. Hence
$C^{\infty}_0(\hat{I}_{\sK})$ can not be dense in the domain $D^2(L^{I_{\sK},\sin_{\sK}^{\sN}})$ if $1\leq N <3$. 
But we can choose \begin{align*}
\mathcal{A}_{0}:=\bigcup_{t>0}P_t^{I_{\sK},\sin_{\sK}^{\sN}}\big[C^{\infty}(\hat{I}_{\sK})\cap L^2(\sin_{\sK}^{\sN}rdr)\big]\subset D(\Gamma_2^{I_{\sK},\sin_{\sK}^{\sN}})\cap C^{\infty}(\hat{I}_{\sK}).
\end{align*}
where $P_t^{I_{\sK},\sin_{\sK}^{\sN}}$ is the induced semi-group of $L^{I_{\sK},\sin_{\sK}^{\sN}}$.
This is a dense subset in the domain of $L^{I_{\sK},\sin_{\sK}^{\sN}}$ since it is stable with respect to the semigroup, and it consists of functions that are smooth in $\hat{I}_{\sK}$ 
since they solve a parabolic PDE with smooth coefficients in $\hat{I}_{\sK}$.
\end{remark}
\begin{remark}
In the following the set $\Xi$ will play an important role. Therefore, we will consider the cases $\lambda_1\geq 3$ and $\lambda_1\in [1,3)$, that appear in the previous theorem, separately. 
In the first case we choose $\mathcal{A}=C_0^{\infty}(\hat{I}_{\sK})$, in the second case we choose $\mathcal{A}=\mathcal{A}_{0}$. 
\end{remark}
\begin{lemma}\label{chronik}
Consider $\mathcal{E}^{\sF}$ and $I_{\sK}\times^{\sN}_{\sin_{\sK}} \mathcal{E}^{\sF}$ as before. Assume $L^{\sC}$ has a discrete spectrum and let $E_i$ be 
the eigenspace for the $i$-th eigenvalue. Consider $u=\sum_{i=1}^ku_1^i\otimes u_2^i\in \Xi$. Then
\begin{align}\label{semigroup}
P^{\sC}_tu=\sum_{i=0}^kP_t^{I_{\sK},\sin_{\sK}^{\sN},\lambda_i}u_1^i\otimes u_2^i\hspace{4pt}\mbox{ for any }t>0.
\end{align}
where $P_t^{I_{\sK},\sin_{\sK}^{\sN},\lambda_i}$ is the semi-group that is generated by $L^{I_{\sK},\sin_{\sK}^{\sN},\lambda_i}$.
\end{lemma}
\proof$\longrightarrow$
 \cite[proof of Lemma 3.3]{okura}
\begin{remark}
In any case $\lambda_1\geq 1$ 
we also define
$$
\Xi':=\bigcup_{t> 0}P^{\sC}_t\ \Xi=\big[\mathcal{A}_{0}\otimes E_0\big]\oplus \left[\sum_{i=1}^{\infty}P_t^{I_{\sK},\sin_{\sK}^{\sN},\lambda_i}C_0^{\infty}(\hat{I}_{\sK})\otimes E_{\lambda_i}\right].$$
$\Xi'$ is dense in $D^2(L^{\sC})$ and stable with respect to $P^C_t$. $\Xi'$ will provide a suitable class of test function.
\end{remark}
\begin{lemma}
Consider $I_{\sK}\times_{\sin_{\sK}}^{\sN}\mathcal{E}^{\sF}$, the corresponding operator $L^{\sC}$ and 
$u\in \Xi'$ of the form $$u=u_1\otimes u_2=P^{\sC}_t(\tilde{u}_1\otimes u_2)=P_t^{I_{\sK},\sin_{\sK}^{\sN},\lambda}\tilde{u}_1\otimes u_2\in D^2(L^{\sC})$$ where $u_2$ 
is an eigenfunction for the eigenvalue $\lambda\in \left\{0\right\}\cup [N,\infty)$ of $-L^{\sF}$ with $N\geq 1$, and $\tilde{u}_1\in C_0^{\infty}(\hat{I}_{\sK})$ if $\lambda>0$ and $\tilde{u}_1\in \mathcal{A}_0$ if $\lambda=0$. Then
\begin{align*}
L^{\sC}u=L^{I_{\sK},\sin_{\sK}^{\sN},\lambda}u_1\otimes u_2.
\end{align*}
\end{lemma}
\begin{proof}We choose $v=v_1\otimes v_2$ where $v_1\in C^{\infty}_0(\hat{I}_{\sK})\subset D(\mathcal{E}^{I_{\sK},\sin_{\sK}^{\sN},\lambda})$ and $v_2\in D(\mathcal{E}^{\sF})$. Then
\begin{align*}
-(L^{\sC}u,v)_{L^2(\m_{\sC})}&=\int_{\sF} \mathcal{E}^{I_{\sK},\sin_{\sK}^{\sN}}(u^x,v^x)d\m_{\sF}(x)+\int_{I_{\sK}}\textstyle{\frac{1}{f^2(p)}}\mathcal{E}^{\sF}(u^r,v^r)\sin_{\sK}^{\sN}rdr\\
%&=\int_{\sF} \mathcal{E}^{\sB,f^{\sN}}(u_1,v_1)u_2v_2d\m_{\sF}(x)\\
%&\hspace{5mm}+\int_{I_{\sK}}\textstyle{\frac{u_1v_1}{\sin_{\sK}^2(p)}}\big(L^{\sF}u_2,v_2\big)_{L^2(\m_{\sF})}\sin_{\sK}^{\sN}rdr\\
&=\int_{\sF} \Big[\mathcal{E}^{I_{\sK},\sin_{\sK}^{\sN}}(u_1,v_1)+\int_{I_{\sK}}\textstyle{\frac{\lambda u_1v_1}{\sin_{\sK}^2}}\sin_{\sK}^{\sN}rdr\Big] u_2v_2d\m_{\sF}\\
&=\mathcal{E}^{I_{\sK},\sin_{\sK}^{\sN},\lambda}(u_1,v_1)\int_{F}u_2v_2d\m_{\sF}=-(L^{I_{\sK},\sin_{\sK}^{\sN},\lambda}u_1\otimes u_2,v)_{L^2(\m_{\sC})}
\end{align*}
Since $C^{\infty}_0(\hat{I}_{\sK})\otimes D(\mathcal{E}^{\sF})$ is dense in $L^2(\m_{\sC})$, the last identity holds for any $v\in L^2(\m_{\sC})$ and the statement follows.
\end{proof}
\subsection{Intermezzo}
Now, we want to understand the regularity of $P_t^{I_{\sK},\sin_{\sK}^{\sN},\lambda}u$ for $\lambda\geq N\geq 1$, $K>0$ and $u\in C^{\infty}_0(\hat{I}_{\sK})$. This might be done by studying
the corresponding Sturm-Liouville operator 
\begin{align*}
L^{I_{\sK},\sin_{\sK}^{\sN},\lambda_i}|_{C^{\infty}_0(\hat{I}_{\sK})}=\frac{d^2}{dr^2}+N\frac{\cos_{\sK}}{\sin_{\sK}}\frac{d}{dr}-\frac{\lambda_i}{\sin_{\sK}}
\end{align*}
and its eigenfunctions. We will go another way and use the result by Bacher and Sturm from \cite{bastco} that states that Theorem 
\ref{oneone} holds if the underlying space is a weighted Riemannian manifold. Then, we also 
use Theorem \ref{theorembochner}, which connects the Bakry-Emery  condition for Dirichlet forms with the Riemannian curvature-dimension condition to deduce $L^{\infty}$-bounds for the gradient of $P_t^{I_{\sK},\sin_{\sK}^{\sN},\lambda}u$.
\begin{proposition}\label{regularity}
Let $\lambda\geq N\geq 1$ and $K>0$. Consider the essentially self-adjoint operator $L^{I_{\sK},\sin_{\sK}^{\sN},\lambda}|_{C^{\infty}_0(\hat{I}_{\sK})}$ and the corresponding semi-group $P_t^{I_{\sK},\sin_{\sK}^{\sN},\lambda}$.
Then 
\begin{align*}
{\Gamma^{I_{\sK}}(P_t^{I_{\sK},\sin_{\sK}^{\sN},\lambda}u)}=\left((P_t^{I_{\sK},\sin_{\sK}^{\sN},\lambda}u)'\right)^2\in L^{\infty}(I_{\sK},\sin_{\sK}^{\sN}rdr).
\end{align*}
for $u\in C^{\infty}_0(\hat{I}_{\sK})$.
\end{proposition}
\begin{proof}
Let us assume that $K=1$ and we consider the metric measure space 
$$F=(I_{\bar{\sK}},\sin_{\bar{\sK}}^{\sN-1}dr) \ \mbox{ for $\bar{K}\geq 1$ such that $\bar{K}N=\lambda$}.$$
$F$ satisfies the condition $RCD^*(\bar{K}(N-1),N)$.
We have the Dirichlet form 
$$\mathcal{E}^{I_{\bar{\sK}},\sin_{\bar{\sK}}^{\sN-1}}=\Ch^{\sF}\ \mbox{ on } \ L^2(\sin_{\bar{\sK}}^{\sN-1}dr).$$
By Theorem \ref{rosen} it satisfies the Bakry-Emery
condition $BE(\bar{K}(N-1),N)$.
\smallskip

The first non-negative eigenvalue of the corresponding self-adjoint operator
equals $\bar{K}N=\lambda$. An eigenfunction is given by $\cos_{\bar{K}}$ what easily can be checked. 
Since $1\leq \bar{K}$, $F$ also satisfies $RCD^*(N-1,N)$ and we can consider the metric $(1,N)$-cone $[0,\pi]\times_{\sin}^{\sN}F$.
By the result of Bacher and Sturm from \cite{bastco} it
satisfies $CD^*(N,N+1)$ but also $RCD^*(N,N+1)$ because of Corollary \ref{eidechse}. By Theorem \ref{theorembochner} the Cheeger energy 
$\ChC$ of $[0,\pi]\times_{\sin}^{\sN}F$ satisfies $BE(N,N+1)$. It implies a Bakry-Emery gradient estimate
\begin{align}\label{dorf}
|\nabla P_t^{\sC}u|_w^2\leq e^{-2Nt} P_t^{\sC}|\nabla u|_w^2
\end{align}
for $u\in D(\ChC)$. By the main results from Section \ref{sec:versus}
the Cheeger energy of the metric cone coincides with the $N$-skew product $I_1\times_{\sin}^N\ChF$ 
in the sense of Dirichlet forms and we have
\begin{align*}
|\nabla u|^2_w=((u^x)')^2+{\textstyle\frac{1}{\sin^2}}((u^r)')^2=\Gamma^{[0,\pi]}(u^x)+{\textstyle\frac{1}{\sin^2}}\Gamma^{I_{\bar{\sK}}}(u^r).
\end{align*}
In particular, the curvature-dimension condition implies that
the metric $(1,N)$-cone satisfies volume doubling and supports a Poincar\'e inequality. 
Hence, $P^{\sC}_t$ is $L^2\rightarrow L^{\infty}$-ultracontractive by Remark \ref{feller}.
\smallskip

We choose $u=u_1\otimes u_2$ where $u_1\in C^{\infty}_0((0,\pi))$ and $u_2\in E_1$. $E_1$ denotes the eigenspace of $\lambda$.
Lemma \ref{chronik} implies that $P^{\sC}u=P^{[0,\pi],\sin^{\sN},\lambda}u_1\otimes u_2$ and (\ref{dorf}) becomes
\begin{align*}
\Gamma^{[0,\pi]}(P_t^{[0,\pi],\sin^{\sN},\lambda}u_1)u^2_2+\frac{1}{\sin^2}(P_t^{[0,\pi],\sin^{\sN},\lambda}u_1)^2\Gamma^{I_{\bar{\sK}}}(u_2)\leq e^{-2Nt} P_t^{\sC}\Gamma^{\sC}(u)\in L^{\infty}(\m_{\sC}).
\end{align*}
This implies that 
\begin{align*}
\Gamma^{[0,\pi]}(P_t^{[0,\pi],\sin^{\sN},\lambda}u_1)=\big((P_t^{[0,\pi],\sin^{\sN},\lambda}u_1)'\big)^2\in L^{\infty}(\sin^{\sN}rdr)
\end{align*} that is the statement in the case $K=1$.
\end{proof}
\begin{remark}\label{remarkseverywhere}
At this point we can make an important remark on the regularity of test functions $u\in \Xi'$.
Consider a strongly local, regular and strongly regular Dirichlet form $\mathcal{E}^{\sF}$ that satisfies $BE(N-1,N)$ and a volume doubling property and supports a local Poincar\'e inequality. 
Assume that closed balls are compact. 
Then remark \ref{feller} implies $L^2\rightarrow L^{\infty}$-ultracontractivity for $P_t^{\sF}$ and it follows that 
%\begin{align*}
$
P_t^{\sF}\Gamma^{\sF}(u)\in L^{\infty}(\m_{\sF})
$
%\end{align*}
for any $u\in D(\mathcal{E}^{\sF})$. Hence, if we consider eigenfunctions of $L^{\sF}$, the Bakry-Ledoux gradient estimate 
implies 
$$\Gamma^{\sF}(P_t^{\sF}u)=e^{-\lambda t}\Gamma^{\sF}(u)\leq P_t^{\sF}\Gamma^{\sF}(u)\in L^{\infty}(\m_{\sF})$$ 
and especially $u, \Gamma^{\sF}(u)\in L^{\infty}(\m_{\sF})$.
Then, the previous proposition implies for 
$$u=u_1 \otimes u_2 \in P_t^{\sC}\big[C_0^{\infty}(\hat{I}_{\sK})\otimes E_{\lambda}\big]=P_t^{I_{\sK},\sin^{\sN},\lambda}C_0^{\infty}(\hat{I}_{\sK})\otimes E_{\lambda}\ \mbox{ and } \ \lambda\geq 1$$ 
that $u,\ \Gamma^{\sC}(u),\ L^{\sC}u \in L^{\infty}(\m_{\sC})$. The same conclusion holds for 
$u=u_1\otimes u_2 \in \mathcal{A}_0\otimes E_{0}$ because $(I_{\sK},\sin_{\sK}^{\sN}rdr)$ satisfies $RCD^*(N,N+1)$.
Hence, for any $u\in \Xi'$ we have $u,\ \Gamma^{\sC}(u),\ L^{\sC}u \in L^{\infty}(\m_{\sC})$. 
\end{remark}\smallskip
\begin{remark}\label{radon}
Consider $u\in P_t^{I_{\sK},\sin_{\sK}^{\sN},\lambda}C_0^{\infty}(\hat{I}_{\sK})$ for $\lambda\geq 1$. We know that $\Gamma^{I_{\sK}}(u)\in L^{\infty}$. Especially
\begin{align*}
\int_{I_{\sK}}\Gamma^{I_{\sK}}(u,\phi)\sin^{\sN}_{\sK}rdr\leq \left\|u\right\|_{L^{\infty}(\sin_{\sK}^{\sN}rdr)}<\infty
\end{align*}
for any $\phi\in C^{\infty}_0(\hat{I}_{\sK})$ with $\|\sqrt{\Gamma^{I_{\sK}}(\phi)}\|_{L^{\infty}}\leq 1$.
Hence, there exists a Radon measure $-{\bf\Delta}u$ on $I_{\sK}$ such that
\begin{align*}
\int_{I_{\sK}}\phi d(-{\bf\Delta}u)=\int_{I_{\sK}}\Gamma^{I_{\sK}}(u,\phi)\sin^{\sN}_{\sK}rdr=\int_{I_{\sK}}L^{I_{\sK},\sin_{\sK}^{\sN}}\phi u\sin^{\sN}_{\sK}rdr.
\end{align*}
%\end{remark}\smallskip
%\begin{remark}
%$u\in P_t^{I_{\sK},\sin_{\sK}^{\sN},\lambda}C_0^{\infty}(\hat{I}_{\sK})$ implies $u\in D^2(L^{I_{\sK},\sin_{\sK}^{\sN},\lambda})$.
$L^{I_{\sK},\sin_{\sK}^{\sN},\lambda}u\in L^2(\sin_{\sK}^{\sN}rdr)$ and for $\phi\in C^{\infty}_0(\hat{I}_{\sK})$ we obtain
\begin{align*}
-\int_{I_{\sK}}L^{I_{\sK},\sin_{\sK}^{\sN},\lambda}u \phi\sin_{\sK}^{\sN}rdr&=\mathcal{E}^{I_{\sK},\sin_{\sK}^{\sN},\lambda}(u,\phi)\\
&=\int_{I_{\sK}}\Gamma^{I_{\sK}}(u,\phi)\sin_{\sK}^{\sN}rdr+\int_{I_{\sK}}\frac{\lambda}{\sin_{\sK}^2}u\phi\sin_{\sK}^{\sN}rdr\\
&=\int_{I_{\sK}}\phi d(-{\bf\Delta}u)+\int_{I_{\sK}}\frac{\lambda}{\sin_{\sK}^2}u\phi\sin_{\sK}^{\sN}rdr.
\end{align*}
In the case $u\in \mathcal{A}_0$ and $\lambda=\lambda_0=0$ these identities are already true where $d{\bf\Delta}u=L^{I_{\sK},\sin_{\sK}^{\sN}}u \sin^{\sN}_{\sK}rdr$.
%Hence, for any $\phi\in C^{\infty}_0(\hat{I}_{\sK})$.
%\begin{align*}
%-\int_{I_{\sK}}\phi d{\bf\Delta}u=\int_{I_{\sK}}\phi\underbrace{\left(-L^{I_{\sK},\sin_{\sK}^{\sN},\lambda}u -\frac{\lambda}{\sin_{\sK}^2}u\right)}_{=:-L^{I_{\sK},\sin_{\sK}^{\sN}} u}\sin_{\sK}^{\sN}rdr
%\end{align*}
%where $-L^{I_{\sK},\sin_{\sK}^{\sN}} u$ is just a measurable function.
\end{remark}
\subsection{Proof of the Bakry-Emery condition for $(K,N)$-cones}
\noindent
We assume $K>0$. Then, the underlying cone measure $\m_{\sC}$ is finite and constants are integrable.
\begin{theorem}\label{mainmaintheorem} 
Let $\mathcal{E}^{\sF}$ be a strongly local, regular and strongly regular Dirichlet form that satisfies the Bakry-Emery curvature-dimension condition $BE(N-1,N)$ in the sense of Definition \ref{bakrycd2}. 
Assume $\mathcal{E}^{\sF}$ satisfies a volume doubling property property, it admits a local $(2,2)$-Poincar\'e inequality and closed balls are compact.
Assume the spectrum of $L^{\sF}$ is discrete and the first positive eigenvalue of $-L^{\sF}$ satisfies $\lambda_1\geq N$.
Let $\mathcal{E}^{\sC}=I_K\times^{\sN}_{\sk}\mathcal{E}^{\sF}$ be the $(K,N)$-cone for $K>0$ and $N\geq 1$ over $\mathcal{E}^{\sF}$ and let $L^{\sC}$ be the corresponding self-adjoint operator.
Assume also that $\mathcal{E}^{\sC}$ satisfies a volume doubling property and admits a local $(2,2)$-Poincar\'e inequality.
Then $I_{\sK}\times_{\sin_{\sK}}^{\sN}\mathcal{E}^{\sF}$ satisfies $BE(KN,N+1)$.
\end{theorem}
\begin{proof} By Lemma \ref{stronglyregular} $\mathcal{E}^{\sC}$ is a strongly regular Dirichlet form and closed balls with respect to the intrinsic distance $\de_{\mathcal{E}^{\sC}}$ are compact.
Hence, we can make use of Sturm's result from \cite{sturmdirichlet1,sturmdirichlet2,sturmdirichlet3}. Especially, there are the properties of Remark \ref{feller}.
Consider 
\begin{align*}
 \Xi=\left[\mathcal{A}\otimes E_0\right]\oplus\bigg[\sum_{i=1}^{\infty}C^{\infty}_0(\hat{I}_{\sK})\otimes E_i\bigg]
\end{align*}
from Theorem \ref{firstmain}.
In the case $N\geq 3$ we set $\mathcal{A}=C_0^{\infty}(\hat{I}_{\sK})$, 
and in the case $N<3$ we set $\mathcal{A}:=\mathcal{A}_{0}\subset D(\Gamma_2^{I_{\sK},\sin_{\sK}^{\sN}})\cap C^{\infty}(\hat{I}_{\sK})$ like in Remark \ref{delicate}.
Consider $u=\sum_{i=0}^k u_1^i\otimes u_2^i\in \Xi$. We have
\begin{align*}
L^{\sC}u&=L^{\sB,\sin_{\sK}^{\sN}}u^x +\textstyle{\frac{1}{\sin_{\sK}^2}}L^{\sF}u^p\\
&=\sum_{i=0}^k \big(L^{\sB,\sin_{\sK}^{\sN}}u^i_1 u^i_2 +\textstyle{\frac{u_1^i}{\sin_{\sK}^2}}L^{\sF}u^i_2\big)=\sum_{i=0}^k \big(L^{\sB,\sin_{\sK}^{\sN}}u^i_1 +
\textstyle{\frac{\lambda_i }{\sin_{\sK}^2}}u_1^i\big) u^i_2\\
&\in [L^{I_{\sK},\sin_{\sK}^{\sN}}\mathcal{A}\otimes E_0]
\oplus \sum_{i=1}^{\infty}C^{\infty}_0(\hat{I}_{\sK})\otimes E_i\subset D(\mathcal{E}^{\sC}).
\end{align*}
In particular, $\Xi\subset D(\Gamma_2^{\sC})$. 
We remind on the regularity properties of eigenfunctions of $L^{\sF}$  and of test function $\phi\in \Xi'$ (see Remark \ref{remarkseverywhere}). 
Hence, $\Gamma^{\sC}_2(u,v;\phi)$ is well-defined for any $u,v\in \Xi$ and any test function $\phi \in \Xi'$.\medskip
\\
\textbf{1.}
First, we assume $N\geq 3$. 
Let $u=u_1\otimes u_2\in C^{\infty}_0(\hat{I}_{\sK})\otimes E_i$ and $v=v_1\otimes v_2\in C^{\infty}_0(\hat{I}_{\sK})\otimes E_j$ for $i,j>0$.
We take a test function $\phi\in \Xi'$ of the form
$$\phi=\phi_1\otimes\phi_2\in P_t^{I_{\sK},\sin_{\sK}^{\sN},\lambda}C_0^{\infty}(\hat{I}_{\sK})\otimes E_{\lambda}\subset \Xi'\ \mbox{ if } \lambda>0$$
or $\phi\in \mathcal{A}_0\otimes E_0$ if $\lambda=0$ (see the definition of $\Xi'$). We know that $\phi,\ \Gamma^{\sC}(\phi),\ L^{\sC}\phi\in L^{\infty}(\m_{\sC})$.
\begin{align*}
&\Gamma_2^{\sC}(u,v;\phi)=\int_{C}{\textstyle\frac{1}{2}}\Gamma^{\sC}(u,v)L^{\sC}\phi d\m_{\sC}-\int_{C}\Gamma^{\sC}(u,{L^{\sC}v})\phi d\m_{\sC}\\
&=\underbrace{\int_{C}{\textstyle\frac{1}{2}}\big(\Gamma^{I_{\sK}}(u_1,v_1)u_2v_2+{\textstyle\frac{u_1v_2}{\sin_{\sK}^2}} \Gamma^{\sF}(u_2,v_1)\big)L^{\sC}\phi d\m_{\sC}}_{=:(I)}
-\underbrace{\int_{C}\Gamma^{\sC}(u,L^{I_{\sK},\sin_{\sK}^{\sN}}v_1v_2+\textstyle{\frac{v_1}{\sin_{\sK}^2}}L^{\sF}v_2)\phi d\m_{\sC}}_{=:(II)}
\end{align*}
We consider $(I)$:
\begin{align*}
(I)&:=\underbrace{\int_{F}\int_{I_{\sK}}{\textstyle\frac{1}{2}}\Gamma^{I_{\sK}}(u_1,v_1)u_2v_2L^{\sC}\phi\sin_{\sK}^{\sN}rdrd\m_{\sF}}_{=:(I)_1}+\underbrace{\int_{F}\int_{I_{\sK}}{\textstyle\frac{1}{2}}{\textstyle\frac{u_1v_1}{\sin_{\sK}^2}} \Gamma^{\sF}(u_2,v_2)L^{\sC}\phi\sin_{\sK}^{\sN}rdrd\m_{\sF}}_{=:(I)_2}
\end{align*}
One can easily check that the integrals are well-defined. For example, we see that $(I)_1<\infty$ since $\Gamma^{I_{\sK}}(u_1,v_1)\in C^{\infty}_0(\hat{I}_{\sK})$ and $u_2v_2\in L^1(\m_{\sF})$.
We can calculate $(I)_1$ explicitly because of Proposition \ref{okurafubinitype} and Remark \ref{radon}:
\begin{align*}
2(I)_1%&=\int_{F}\int_{I_{\sK}}\Gamma^{I_{\sK}}(u_1,v_1)u_2v_2L^{I_{\sK},\sin_{\sK}^{\sN},\lambda}\phi_1\phi_2\sin_{\sK}^{\sN}rdrd\m_{\sF}\\
&=\int_{I_{\sK}}\Gamma^{I_{\sK}}(u_1,v_1)L^{I_{\sK},\sin_{\sK}^{\sN},\lambda}\phi_1\sin_{\sK}^{\sN}rdr\int_{\sF}u_2v_2\phi_2d\m_{\sF}\\
&=\bigg(\int_{I_{\sK}}\Gamma^{I_{\sK}}(u_1,v_1)d{\bf\Delta}\phi_1-\int_{I_{\sK}}\Gamma^{I_{\sK}}(u_1,v_1){\textstyle\frac{\phi_1\lambda}{\sin_{\sK}^2}}\sin_{\sK}^{\sN}rdr\bigg)\int_{\sF}u_2v_2\phi_2d\m_{\sF}\\
&=\bigg(\int_{I_{\sK}}L^{I_{\sK},\sin_{\sK}^{\sN}}\Gamma^{I_{\sK}}(u_1,v_1)\phi_1\sin_{\sK}^{\sN}rdr-\int_{I_{\sK}}\Gamma^{I_{\sK}}(u_1,v_1){\textstyle\frac{\phi_1\lambda}{\sin_{\sK}^2}}\sin_{\sK}^{\sN}rdr\bigg)\int_{\sF}u_2v_2\phi_2d\m_{\sF}\\
&=\int_{C}L^{I_{\sK},\sin_{\sK}^{\sN}}\Gamma^{I_{\sK}}(u_1,v_1)u_2v_2\phi d\m_{\sC}+\int_{C}\Gamma^{I_{\sK}}(u_1,v_1){\textstyle\frac{1}{\sin_{\sK}^2}}u_2v_2\phi_1L^{\sF}\phi_2 d\m_{\sC}\hspace{30cm}
\end{align*}
We remark that
$u_2v_2\in D(L_1^{\sF})$ for any $u_2,v_2\in D_2(L^{\sF})$
%since $u_2,v_2 \in D(\Gamma_2^{\sF})$ and $u_i, \Gamma^(u_i), L^{\sF}u_i \in L^{\infty}(\m_{\sF})$ for $i=1,2$ 
(for example see \cite[Section I.4, Theorem 4.2.2]{hirsch}) and we have
\begin{align}\label{maus}
L_1^{\sF}(u_2v_2)=L^{\sF}u_2  v_2 +  u_2L^{\sF}v_2 + \Gamma^{\sF}(u_2,v_2)\ \ \ \&\ \ \ \int_{\sF}\phi L_1^{\sF}u d\m_{\sF}=\int_{\sF}uL^{\sF}\phi d\m_{\sF}
\end{align}
if $u\in D(L^{\sF}_1)$ and $\phi\in D_2(L^{\sF})\cap L^{\infty}(\m_{\sF})$.
The operator $L_1^{\sF}$ with domain $D(L_1^{\sF})\subset L^1(\m_{\sF})$ is the smallest closed extension 
of $L^{\sF}$ to $L^1(\m_{\sF})$ and there is an associated semi group $P_{t,1}^{\sF}:L^1(\m_{\sF})\rightarrow L^1(\m_{\sF})$.
The second equation in (\ref{maus}) comes from 
\begin{align*}
\int_{\sF} P_{t,1}^{\sF}u \cdot \phi d\m_{\sF}=\int_{\sF} u \cdot P_t^{\sF}\phi d\m_{\sF} \ \ \mbox{for any }u\in L^1(\m_{\sF})\mbox{ and }\phi \in L^{\infty}(\m_{\sF})
\end{align*}
that follows for instance from the existence of a bounded, continuous heat kernel (see Remark \ref{feller}) and Fubini's theorem.
Next, we consider $(I)_2$. Similar as before we obtain
\begin{align*}
2(I)_2%&=\int_{F}\int_{I_{\sK}}{\textstyle\frac{u_1v_2}{\sin_{\sK}^2}} \Gamma^{\sF}(u_2,v_2)L^{I_{\sK},\sin_{\sK}^{\sN},\lambda}\phi_1\phi_2\sin_{\sK}^{\sN}rdr d\m_{\sF}\\
&=\int_{I_{\sK}}{\textstyle\frac{u_1v_1}{\sin_{\sK}^2}}L^{I_{\sK},\sin_{\sK}^{\sN},\lambda}\phi_1\sin_{\sK}^{\sN}rdr \int_{F}\Gamma^{\sF}(u_2,v_2)\phi_2d\m_{\sF}\\
&=\bigg(\int_{I_{\sK}}L^{I_{\sK},\sin_{\sK}^{\sN}}\Big({\textstyle\frac{u_1v_1}{\sin_{\sK}^2}}\Big)\phi_1\sin_{\sK}^{\sN}rdr-\int_{I_{\sK}}{\textstyle\frac{u_1v_1}{\sin_{\sK}^2}}\lambda\phi_1\sin_{\sK}^{\sN}rdr\bigg)\int_{\sF}\phi_2\Gamma^{\sF}(u_2,v_2)d\m_{\sF}\\
&=\int_{C}L^{I_{\sK},\sin_{\sK}^{\sN}}\Big({\textstyle\frac{u_1v_1}{\sin_{\sK}^2}}\Big)\Gamma^{\sF}(u_2,v_2)\phi d\m_{\sC}+\int_{C}{\textstyle\frac{u_1v_1}{\sin_{\sK}^2}}\phi_1\Gamma^{\sF}(u_2,v_2)L^{\sF}\phi_2d\m_{\sC}\hspace{30cm}
\end{align*}
Then, we consider $(II)$
\begin{align*}
(II)=\underbrace{\int_{F}\int_{I_{\sK}}\Gamma^{\sC}(u,L^{I_{\sK},\sin_{\sK}^{\sN}}v_1v_2)\phi \sin_{\sK}^{\sN}rdr d\m_{\sF}}_{=:(II)_1}+\underbrace{\int_{F}\int_{I_{\sK}}\Gamma^{\sC}(u,\textstyle{\frac{v_1}{\sin_{\sK}^2}}L^{\sF}v_2)\phi \sin_{\sK}^{\sN}rdr d\m_{\sF}}_{=:(II)_2}.
\end{align*}
We can also calculate $(II)_1$ and $(II)_2$:
\begin{align*}
(II)_1=\int_{C}\Gamma^{I_{\sK}}(u_1,L^{I_{\sK},\sin_{\sK}^{\sN}}v_1)u_2v_2\phi d\m_{\sC}+\int_{F}\int_{I_{\sK}}{\textstyle\frac{u_1L^{I_{\sK},\sin_{\sK}^{\sN}}v_1}{\sin_{\sK}^2}}\Gamma^{\sF}(u_2,v_2)\phi\sin_{\sK}^{\sN}rdr d\m_{\sF}\hspace{30cm}
\end{align*}
\begin{align*}
(II)_2=\int_{F}\int_{I_{\sK}}\Gamma^{I_{\sK}}(u_1,{\textstyle\frac{v_1}{\sin_{\sK}^2}})u_2L^{\sF}v_2\phi \sin_{\sK}^{\sN}rdr d\m_{\sF}+\int_{F}\int_{I_{\sK}}{\textstyle\frac{v_1u_1}{\sin_{\sK}^4}}\Gamma^{\sF}(u_2,L^{\sF}v_2)\phi \sin_{\sK}^{\sN}rdr d\m_{\sF}\hspace{30cm}
\end{align*}
Finally, we obtain
\begin{align}\label{krebs}
\Gamma_2^{\sC}(u,v,\phi)
%&=(I)_1+(I)_2+(II)_1+(II)_2\nonumber\\
&=\int_{C}\Gamma_2^{I_{\sK},\sin_{\sK}^{\sN}}(u_1,v_1)u_2v_2\phi d\m_{\sC}+\int_{C}{\textstyle\frac{u_1v_1}{\sin_{\sK}^2}}\Gamma_2^{\sF}(u_2,v_2,\phi_1)\phi_2d\m_{\sC}\nonumber\\
&\hspace{40pt}+\int_{C}\Big[{\textstyle\frac{1}{2}}\Gamma^{I_{\sK}}(u_1,v_1){\textstyle\frac{1}{\sin_{\sK}^2}}L^{\sF}(u_2v_2)+{\textstyle\frac{1}{2}}L^{I_{\sK},\sin_{\sK}^{\sN}}\big({\textstyle\frac{u_1v_1}{\sin_{\sK}^2}}\big)\Gamma^{\sF}(u_2,v_2)\nonumber\\
&\hspace{110pt}-{\textstyle\frac{u_1L^{I_{\sK},\sin_{\sK}^{\sN}}v_1}{\sin_{\sK}^2}}\Gamma^{\sF}(u_2,v_2)-\Gamma^{I_{\sK}}(u_1,{\textstyle\frac{v_1}{\sin_{\sK}^2}})u_2L^{\sF}v_2\Big]\phi d\m_{\sC}
\end{align}
$\Gamma_2^{\sC}(u,v;\phi)$ and also the right hand side in the last equation is linear in $\phi$. Hence, the last equation also holds for general $\phi=\sum_{i=1}^{k}\phi_1^i\otimes\phi_2^i\in \Xi'$.
$
\Gamma_2^{\sC}(u,v;\phi)+\Gamma_2^{\sC}(v,u;\phi)
$
looks exactly like the weak version of equation (\ref{zzz}) in the proof of Theorem \ref{maintheorem} where we prove the classical $\Gamma_2$-estimate. Hence, we can proceed exactly like in the proof of Theorem \ref{maintheorem} and
 we obtain the sharp $\Gamma_2$-estimate in a weak form for $u\in \Xi$ and test functions $\phi \in \Xi'$ with $\phi\geq 0$.\medskip
 \\
\textbf{2.}
Now, we deal with the case $1\leq N <3$.
We compute $\Gamma_2^{\sC}(u,v;\phi)$ exactly like in the case $N\geq 3$ but
we have to consider the case when $u_1\otimes u_2\in \mathcal{A}_0\ \otimes E_0$ and $v_1\otimes v_2\in C^{\infty}_0(\hat{I}_{\sK})\otimes E_i$ for $i>0$ separately. 
Any other case is already covered by the previous paragraph.
We recall that
$u_2=const=m\in\mathbb{R}$ because of Remark \ref{feller}. 
We can compute $\Gamma^{\sC}_2(u,v;\phi)$ for $\phi\in \Xi'$ exactly
like in the previous paragraph since terms of the form $u_1v_1$, $\Gamma^{I_{\sK}}(u_1,v_1)$ and 
$L^{I_{\sK},\sin_{\sK}^{\sN}}(u_1v_1)$ are in $C^{\infty}_0(\hat{I}_{\sK})$. 
We obtain again formula (\ref{krebs}).
\smallskip

The only case that we still have to check is $u=v\in \mathcal{A}_0\ \otimes E_0$. 
It is not covered, yet, since $u_1^2\notin C^{\infty}_0(\hat{I}_{\sK})$.
First, let $\phi=\phi_1\otimes\phi_2\in P_t^{I_{\sK},\sin_{\sK}^{\sN},\lambda}\otimes E_{i}$. 
We know that $u=u_1\otimes m \in D(\mathcal{E}^{\sC})$, $u_1\in \mathcal{A}_0\subset D(\Gamma_2^{I_{\sK},\sin_{\sK}^{\sN}})$
and
$
\Gamma^{\sC}(u)=\Gamma^{I_{\sK}}(u_1)m^2. 
$
Hence,
\begin{align}
\Gamma_2^{\sC}(u;\phi)&=\int_{\sC} {\textstyle \frac{1}{2}}L^{\sC}\phi \Gamma^{\sC}(u)d\m_{\sC}-\int_{C}\Gamma^{\sC}(u,L^{\sC}u)\phi d\m_{\sC}\nonumber\\
&=\int_{\sC} {\textstyle \frac{1}{2}}L^{I_{\sK},\sin_{\sK}^{\sN},\lambda_i}\phi_1\phi_2\Gamma^{I_{\sK}}(u_1)m^2\sin_{\sK}^{\sN}rdr d\m_{\sF}\nonumber\\
&\hspace{2cm}-\int_{C}\Gamma^{I_{\sK}}(u_1,L^{I_{\sK},\sin_{\sK}^{\sN}}u_1)m^2\phi_1\phi_2 \sin_{\sK}^{\sN}rdr d\m_{\sF} \nonumber\\
%=-\int_{\sC} \Gamma^{\sC}(\phi, \Gamma^{\sC}(u))d\m_{\sC}-\int_{C}\Gamma^{\sC}(u,L^{\sC}u)\phi d\m_{\sC}\nonumber\\
%&=-\int_{\sC} \Gamma^{\sC}(\phi_1\phi_2,\Gamma^{I_{\sK}}(u)m^2)d\m_{\sC}-\int_{C}\Gamma^{\sC}(u,L^{\sC}u)\phi d\m_{\sC}\nonumber\\
%&=-\int_{I_{\sK}} \Gamma^{I_{\sK}}(\phi_1,\Gamma^{I_{\sK}}(u))\sin_{\sK}^{\sN}rdr \int_{\sF} m^2\phi_2 d\m_{\sF}-\int_{F}\int_{I_{\sK}}\Gamma^{I_{\sK}}(u_1,L^{I_{\sK},\sin_{\sK}^{\sN}}u_1)m^2\phi \sin_{\sK}^{\sN}rdr d\m_{\sF}\nonumber\\
%&=\int_{I_{\sK}} \Big(L^{I_{\sK},\sin_{\sK}^{\sN}}\Gamma^{I_{\sK}}(u)) -\Gamma^{\sC}(u,L^{\sC}u)\Big)\phi_1\sin_{\sK}^{\sN}rdr\int_{\sF} m^2\phi_2 d\m_{\sF}\nonumber\\
&=\int_{\sF}\!\!m^2\phi_2d\m_{\sF}\!\!\int_{I_{\sK}}\!\Big[{\textstyle \frac{1}{2}}\Gamma^{I_{\sK}}(u_1)L^{I_{\sK},\sin_{\sK}^{\sN},\lambda_i}\phi_1 -\Gamma^{I_{\sK}}(u_1,L^{I_{\sK},\sin_{\sK}^{\sN}}u_1)\phi_1\Big]\sin_{\sK}^{\sN}rdr\nonumber
\end{align}
Since $\phi_2$ is an eigenfunction of $L^{\sF}$, the right hand side is $0$ unless $\lambda_i=0$ and $\phi_2\neq 0$.
We conclude that $\Gamma_2^{\sC}(u;\phi)\neq 0$ for $\phi\in \Xi'$ only if $\phi_2=const\neq 0$. In any case:

\begin{align}\label{lefthand}
\Gamma_2^{\sC}(u;\phi)
%\nonumber\\
%&=\int_{\sF}m^2\phi_2d\m_{\sF}\int_{I_{\sK}} \Big(\Gamma^{I_{\sK}}(u_1)L^{I_{\sK},\sin_{\sK}^{\sN}}\phi_1 -\Gamma^{I_{\sK}}(u_1,L^{I_{\sK},\sin_{\sK}^{\sN}}u_1)\phi_1\Big)\sin_{\sK}^{\sN}rdr\nonumber\\
=\int_{\sF}m^2\phi_2d\m_{\sF}\Gamma_2^{I_{\sK},\sin_{\sK}^{\sN}}(u_1;\phi_1).
\end{align}
This is just (\ref{krebs}) where we replace $\Gamma_2^{I_{\sK},\sin_{\sK}^{\sN}}(u_1)\phi_1$ by $\Gamma_2^{I_{\sK},\sin_{\sK}^{\sN}}(u_1;\phi_1)$.
%Hence, the inequality for $u\in \mathcal{A}_0\otimes E_0$ and $\phi\in \Xi'$ that yields the Bakry-Emery condition is just the Bakry-Emery condition for $\mathcal{E}^{I_{\sK},\sin_{\sK}^{\sN}}$.
But we can proceed like at the end of the previous paragraph.
Because $\mathcal{E}^{I_{\sK},\sin^{\sN}_{\sK}}$ satisfies
$BE(KN,N+1)$ we can bound (\ref{lefthand}) by
\begin{align*}
\Gamma_2^{\sC}(u;\phi)&\geq m^2\int_{\sF}\int_{I_{\sK}}\big[KN\phi \Gamma^{I_{\sK},\sin_{\sK}^{\sN}}(u_1)+\frac{1}{N+1}(L^{I_{\sK},\sin_{\sK}^{\sN}}u_1)^2\phi \big]\sin_{\sK}^{\sN}rdr d\m_{\sF}
%&=KN\int_{\sC}\phi \Gamma^{\sC}(u)d\m_{\sC}+\frac{1}{N+1}\int_{\sC}(L^{\sC}u)^2\phi d\m_{\sC}
\end{align*}
if $\phi\geq 0$. 
%This inequality is again just a special case of (\ref{krebs}).
Hence, for $u\in\Xi$ and $\phi\in \Xi'$ with $\phi\geq 0$ we have the desired $\Gamma_2$-estimate.
\medskip
\\
\textbf{3.}
We extend this estimate to any function $u\in D(\Gamma_2)$. We choose a sequence $u_n\in \Xi$ that converges to $u\in D(\Gamma_2)$ in $D^2(L^{\sC})$.
Then we obtain that
\begin{align}
&\int_{C}\Gamma^{\sC}(u_n,u_n)L^{\sX}\phi dm_{\sC}\rightarrow \int_{C}\Gamma^{\sC}(u,u)L^{\sC}\phi dm_{\sC},\ \ \ \
\int_{C}\Gamma^{\sC}(u_n,u_n)\phi d\m_{\sC}\rightarrow \int_{C}\Gamma^{\sC}(u,u)\phi d\m_{\sC},\nonumber\\
&\hspace{4cm}\int_{C} L^{\sC}u_n\phi d\m_{\sC} \rightarrow \int_{C}L^{\sC}u\phi d\m_{\sC}.
\end{align}
for $\phi\in D^{b,2}_+(L^{\sC})\cap \Xi$. 
We still need to show convergence of $\int_{C}\Gamma^{\sC}(u_n,L^{\sC}u_n)\phi d\m_{\sC}$.
Since $u_n,L^{\sC}u_n,\phi\in D(\mathcal{E}^{\sC})$ and $\phi,\Gamma^{\sC}(\phi)\in L^{\infty}(\m_{\sC})$, we can apply the Leibniz rule (\ref{leibnizkeks}) for $\Gamma^{\sC}$.
 We obtain
\begin{align*}
\int_{C}\Gamma^{\sC}(u_n,L^{\sC}u_n)\phi d\m_{\sC}&=\int_{C}\Gamma^{\sC}(u_n,L^{\sC}u_n\phi)d\m_{\sC}-\int_{C}\Gamma^{\sC}(u_n,\phi)L^{\sC}u_n d\m_{\sC}\\
&=-\underbrace{\int_{C}(L^{\sC}u_n)^2\phi d\m_{\sC}}_{\longrightarrow \int_{C}(L^{\sC}u)^2\phi d\m_{\sC}}-\int_{C}\Gamma^{\sC}(u_n,\phi)L^{\sC}u_n d\m_{\sC}.
\end{align*}
Consider the second term on the right hand side.
\begin{align*}
&\left|\int_{C}\Gamma^{\sC}(u_n,\phi)L^{\sC}u_nd\m_{\sC}-\int_{C}\Gamma^{\sC}(u,\phi)L^{\sC}u d\m_{\sC}\right|\\
&\leq\int_{C}\left|\Gamma^{\sC}(u_n,\phi)L^{\sC}(u_n-u)\right|+\left|\Gamma^{\sC}(u_n-u,\phi)L^{\sC}u\right|d\m_{\sC}\\
&\leq \left\|\Gamma^{\sC}(\phi)\right\|_{L^{\infty}}\bigg( \underbrace{\int_{C}\Gamma^{\sC}(u_n)d\m_{\sC}}_{\rightarrow \int_{C}\Gamma^{\sC}(u)d\m_{\sC}}\underbrace{\int_{\sC}(L^{\sC}(u_n-u))^2d \m_{\sC}}_{\rightarrow 0}+\underbrace{\int_{C}\Gamma^{\sC}(u_n-u)d\m_{\sC}}_{\rightarrow 0}\int_{\sC}(L^{\sC}u)^2d \m_{\sC}\bigg)
\end{align*}
Since $\phi \in \Xi'$, we have that $\left\|\Gamma^{\sC}(\phi)\right\|_{L^{\infty}}<\infty$. It follows that 
\begin{align*}
 \int_{C}\Gamma^{\sC}(u_n,\phi)L^{\sC}u_n d\m_{\sC}\rightarrow \int_{C}\Gamma^{\sC}(u,\phi)L^{\sC}u d\m_{\sC}\hspace{5pt}\mbox{ for }u_n\rightarrow u \mbox{ in }D^2(L^{\sC})
\end{align*}
and consequently
\begin{align*}
\int_{C}\Gamma^{\sC}(u_n,L^{\sC}u_n)\phi d\m_{\sC}\rightarrow \int_{C}\Gamma^{\sC}(u,L^{\sC}u)\phi d\m_{\sC}\hspace{5pt}\mbox{ for }u_n\rightarrow u \mbox{ in }D^2(L^{\sC})
\end{align*}
for any $u\in D(\Gamma_2^{\sC})$ and for any test function $\phi \in \Xi'$ with $\phi\geq 0$. \medskip
\\
\textbf{4.}
Finally, we show that the $\Gamma_2$-estimate holds for any admissible test function $\phi\in D^{b,2}_+(L^{\sC})$. 
Since we assume $K>0$, the measure $\m_{\sC}$ is finite and we can assume that $\phi\geq {{M}}>0$ for some 
positive constant ${{M}}\in D^2(L^{\sC})$. 
Consider a sequence $\phi_n\in \Xi$ that converges to $\phi$ in $D^2(L^{\sC})$. Then, we
also have $P_t^{\sC}\phi\geq {{M}}$ and $P_t^{\sC}\phi_n\rightarrow P_t^{\sC}\phi$ in $D^2(L^{\sC})$ for all $t>0$.
Since we assume that $\mathcal{E}^{\sC}$ satisfies volume doubling and supports a Poincar\'e inequality, is strongly regular and admits that closed balls are compact (see Lemma \ref{stronglyregular})
there is an upper bound for the heat kernel (see \cite[Corollary 4.2]{sturmdirichlet3}, Remark \ref{feller}) 
that is equivalent to $L^2\rightarrow L^{\infty}$-ultracontractivity of the semigroup $P^{\sC}_t$ (see \cite[Chapter 14.1]{grigoryanheat} and Remark \ref{feller}). 
Hence, $P_t^{\sC}\phi_n\rightarrow P_t^{\sC}\phi$ and $L^{\sC}P_t^{\sC}\phi_n\rightarrow L^{\sC}P_t^{\sC}\phi$ in $L^{\infty}(\m_{\sC})$. Since $P_t^{\sC}\phi\geq {{M}}>0$, 
we deduce that $P_t^{\sC}\phi_n\in D^{b,2}_+(L^{\sC})\cap\Xi'$ for $n$ sufficiently big.
Then, the results from the previous paragraphs state that 
\begin{align*}
\int_{C}\big({\textstyle \frac{1}{2}}\Gamma^{\sC}(u)L^{\sC}P^{\sC}_t\phi_n-\Gamma^{\sC}(u,L^{\sC}u)P^{\sC}_t\phi_n \big)d\m_{\sC}\geq \int_{\sC}\big(KN\Gamma^{\sC}(u) +\textstyle{\frac{1}{N+1}}(L^{\sC}u)^2\big)P^{\sC}_t\phi_n\m_{\sC}.
\end{align*}
Hence, if $n\rightarrow \infty$
\begin{align*}%\label{unibonn}
\int_{C}\big({\textstyle \frac{1}{2}}\Gamma^{\sC}(u)L^{\sC}P^{\sC}_t\phi-\Gamma^{\sC}(u,L^{\sC}u)P^{\sC}_t\phi \big)d\m_{\sC}\geq \int_{\sC}\big(KN\Gamma^{\sC}(u)P^{\sC}_t\phi +\textstyle{\frac{1}{N+1}}(L^{\sC}u)^2P^{\sC}_t\phi\big)\m_{\sC}
\end{align*}
for $u\in D(\Gamma_2)$ and $\phi\in D^{b,2}_+$ with $\phi\geq {{M}}>0$ because of the $L^{\infty}$-convergence of $P_t^{\sC}\phi_n$ and $P_t^{\sC}L^{\sC}\phi_n$. Then we also let $M\rightarrow 0$ and the inequality holds for any test function
of the form $P_t^{\sC}\phi$ where
$\phi\in D^{b,2}_+$. Finally, by application of Lebesgue's dominated convergence theorem one can check that $P_t^{\sC}\phi$ and $P_t^{\sC}L^{\sC}\phi$ converges to $\phi$ and $L^{\sC}\phi$, respectively, w.r.t. weak-$*$ convergence if 
$\phi \in D^{b,2}_+(L^{\sC})$, and we obtain the $\Gamma_2$-estimate for any $u\in D(\Gamma_2^{\sC})$ and for any $\phi \in D^{b,2}_+(L^{\sC})$.
\end{proof}

\section{Preliminaries on the calculus for metric measure spaces}
\subsection{The curvature-dimension condition}\label{curvdim}
\begin{assumption}\label{assmms}
Let $(X,\de_{\sX})$ be a complete and separable metric space, 
and $\m_{\sX}$ a locally finite Borel measure on $(X,\de_{\sX})$ with full support.
That is,  for all $x\in X$ and all sufficiently small $r>0$ the volume $\m_{\sX}(B_r(x))$ of balls centered at $x$ is positive and finite. We assume that $X$ has more than one point. A triple $(X,\de_{\sX},\m_{\sX})$ will be called \emph{metric measure space}.
\smallskip\\
$(X,\de_{\sX})$ is called \textit{intrinsic} or \textit{length space} if $\de_{\sX}(x,y)=\inf \mbox{L} (\gamma)$ for all $x,y\in X$, 
where the infimum runs over all curves $\gamma$ in $X$ connecting $x$ and $y$. $(X,\de_{\sX})$ is called \textit{strictly intrinsic} or \textit{geodesic space}
if every two points $x,y\in X$ are connected by a curve $\gamma$ with $\de_{\sX}(x,y)=\mbox{L}(\gamma)$.
Distance minimizing curves of constant speed are called \textit{geodesics}.
$(X,\de_{\sX})$ is called \textit{non-branching} 
if for every tuple $(z,x_0,x_1,x_2)$ of points in $X$ for which $z$ is a midpoint of $x_0$ and $x_1$ as well as of $x_0$ and $x_2$, it follows that $x_1=x_2$.
For basic facts about optimal transport and Wasserstein geometry we refer to \cite{viltot}.
An intrinsic metric space which is complete and locally compact, is strictly intrinsic, i.e. a geodesic space (\cite[Theorem 2.5.23 ]{bbi}).\end{assumption}

\begin{definition}[\cite{bast}, Reduced curvature-dimension condition]\label{CDstern}
A metric measure space $(X,\de_{\sX},\m_{\sX})$ satisfies the condition $CD^*(K,N)$ for $K\in\mathbb{R}$ and $N\in[1,\infty)$ if for each pair
$\mu_0,\mu_1\in\mathcal{P}^2(X,\de_{\sX},\m_{\sX})$ there exists an optimal
coupling $q$ of $\mu_0=\rho_0\m_{\sX}$ and $\mu_1=\rho_1\m_{\sX}$ and a geodesic 
$\mu_t=\rho_t \m_{\sX}$ in $\mathcal{P}^2(X,\de_{\sX},\m_{\sX})$ connecting them such that
\begin{equation}\label{scheisswetter}
\int_{X}\rho_t^{-1/N'}\rho_td\m_{\sX}\ge\int_{X\times
X}\!\big[\sigma^{(1-t)}_{K,N'}(\de_{\sX})\rho^{-1/N'}_0(x_0)+
\sigma^{(t)}_{K,N'}(\de_{\sX})\rho^{-1/N'}_1(x_1)\big]dq(x_0,x_1)
\end{equation}
for all $t\in (0,1)$ and all $N'\geq N$ where $\de_{\sX}:=\de_{\sX}(x_0,x_1)$.
In the case $K>0$, the \textit{volume distortion coefficients} $\sigma^{(t)}_{K,N}(\cdot)$
for  $t\in (0,1)$  are defined by
$$\sigma_{K,N}^{(t)}(\theta)=\textstyle{\frac {\sin\left(\sqrt{ K/{N}}\theta t\right)} {\sin\left(\sqrt{K/{N}}\theta\right)} }$$
if $0\le\theta< \scriptstyle{\sqrt{\frac{N}K}\pi}$ and by $\sigma_{K,N}^{(t)}(\theta)=\infty$ if $\theta\geq\scriptstyle{\sqrt{\frac{N}K}\pi}$.
In the case $K\leq0$ an analogous definition applies with an appropriate replacement of
$\sin\scriptscriptstyle{\left(\sqrt{\frac K{N-1}}-\right)}$.
\end{definition}
\begin{remark}[\cite{stugeo2}]\label{CD}
A metric measure space $(X,\de_{\sX},\m_{\sX})$ satisfies the curvature-dimension condition $CD(K,N)$ for $K\in\mathbb{R}$ and $N\in[1,\infty)$ if we replace in Definition \ref{CDstern} the coefficients
$\sigma^{(t)}_{K,N}(\theta)$ by 
\begin{align*}
\tau_{{K},N}^{(t)}(\theta)=
\begin{cases} 
\infty & \mbox{ if }\ {K}\theta^2> (N-1)\pi^2,\medskip\\
t^{1/N}\cdot\sigma_{{K},N-1}^{(t)}(\theta)^{1-1/N} & \mbox{ if }\ {K}\theta^2\leq(N-1)\pi^2 \ \ \&\ \ N>1,\medskip \\
t & \mbox{ if }\ {K}\theta^2\leq 0 \ \ \&\ \ N=1.
\end{cases}
\end{align*}
This is the original condition that was introduced by Sturm in \cite{stugeo2}. 
By definition a single point satisfies $CD(K,1)$ for any ${K}>0$, and ${K}>0$ and $N=1$ can only appear in this case.
\end{remark}
\begin{theorem}[Doubling property, \cite{bast}]\label{doubling}
For a metric measure space $(X,\de_{\sX},\m_{\sX})$ that satisfies
$CD^*(K,N)$ for some $K\in\mathbb{R}$ and $N\geq 1$, the doubling property holds on each bounded subset $X'\subset\supp\m$. In
particular each bounded closed subset is compact and $(X,\de_{\sX},\m_{\sX})$ is locally compact.
If $K\geq 0$ or $N=1$ the doubling constant is $\leq 2^{\sN}$. %Otherwise, it can be estimated in terms of $K,N$ and the diameter $L$ of $X'$ as follows
\end{theorem}
\begin{theorem}[Hausdorff dimension, \cite{bast}]\label{hausdorff}
For a metric measure space $(X,\de_{\sX},\m_{\sX})$ that satisfies $MCP(K,N)$ for some $K\in\mathbb{R}$ and $N\geq 1$, the support of $\m_{\sX}$ has Hausdorff dimension $\leq N$.
\end{theorem}
\begin{remark}\label{spätabendsinfinland}
The condition $CD^*$ implies that the metric space $(X,\de_{\sX})$ is geodesic.
\end{remark}
If $(X,\de_{\sX},\m_{\sX})$ is non-branching then the reduced curvature-dimension condition $CD^*({K},N)$ implies the  measure contraction property $MCP({K},N)$ 
by a result of Cavalletti and Sturm \cite{cavallettisturm} where ${K}\in\mathbb{R}$ if $N>1$ and ${K}= 0$ if $N=1$.
There are two different definitions of the measure contraction property by Ohta in \cite{ohtpro} and by Sturm in \cite{stugeo2}. 
The latter is more restrictive and implies the former. In a non-branching situation the definitions coincide. We give the definition in the sense of Ohta.
%\end{remark}
\begin{definition}[Measure contraction property, \cite{ohtmea}]
Let $(X,\de_{\sX},\m_{\sX})$ be a non-branching metric measure space.
% such that $X\neq \left\{\mbox{single point}\right\}$. 
Then it satisfies the measure contraction property $MCP({K},N)$ if for any $x\in X$, 
for any measurable subset $A\subset X$ with $\m_{\sX}(A)<\infty$ (and $A\subset B_{\pi\scriptscriptstyle{\sqrt{{(N-1)}/{{K}}}}}(x)$ if ${K}>0$), there exists a $L^2$-Wasserstein geodesic $\Pi$ 
such that $\delta_x=(e_0)_*\Pi$ and $\m_{\sX}(A)^{-1}\m_{\sX}=(e_1)_*\Pi$ and
\begin{align*}
d\m_{\sX}\geq (e_t)_*\left(\tau_{{K},N}^{(t)}(\mbox{L}(\gamma))^{N}\m_{\sX}(A)d\Pi(\gamma)\right).
\end{align*} 
By definition a single point satisfies $MCP({K},1)$ for any ${K}>0$, and ${K}>0$ and $N=1$ can only appear in this case.
\end{definition}
\noindent
A corollary of the measure contraction property is the Bonnet-Myers Theorem.
%\begin{theorem}[Cavalletti, Sturm, \cite{cavallettisturm}]\label{cavalletti}
%Let $(X,\de_{\sX},\m_{\sX})$ be a nonbranching metric measure space. Then $CD^*(K,N)$ implies the measure contraction property $MCP(K,N)$.
%
%\end{theorem}
\begin{theorem}[Generalized Bonnet-Myers Theorem, \cite{ohtmea}]\label{bonnet}
Assume that a metric measure space  $(X,\de_{\sX},\m_{\sX})$ satisfies satisfies $MCP(K,N)$ for some $K\in\mathbb{R}$ and $N\geq 1$. 
Then the diameter of $(X,\de_{\sX})$ is bounded by $\pi\textstyle{\sqrt{N-1/K}}$.
In particular, a metric measure space that is nonbranching and satisfies the reduced curvature-dimension condition $CD^*(K,N)$ for $K>0$ has bounded diameter by $\pi\sqrt{N-1/K}$.
\end{theorem}
\begin{remark}
One can check that the generalized Bonnet-Myers Theorem is an almost immediate consequence of the condition $CD(K,N)$. But despite this fact
the reduced curvature-dimension is more suitable for many applications. If the metric measure space is a Riemannian manifold, the reduced and non-reduced condition are equivalent
and one conjectures that this should hold also in a more general setting.
In any case, there are the following implications 
\begin{eqnarray*}CD(K,N)\Rightarrow CD^*(K,N),\hspace{5pt}CD^*(K,N)\Leftrightarrow CD_{loc}(K,N),\hspace{5pt}CD^*(0,N)\Leftrightarrow CD(0,N)
\end{eqnarray*}
(see \cite{bast}) where the definition of $CD_{loc}(K,N)$ can be found for example in \cite{stugeo2}.
\end{remark}

\subsection{First order calculus for metric measure spaces}\label{firstordercalculus}
%\subsection{Poincar\'e and Sobolev inequalities in the setting of metric measure spaces}
Let $(X,\de_{\sX},\m_{\sX})$ be a metric measure space. 
We recall the concept of upper gradient.
Let $\gamma: J\rightarrow (X,\de_{\sX})$ be absolutely continuous curve in the sense of \cite{agsgradient}. 
Then, $\gamma$ has a well-defined metric speed 
$
|\dot{\gamma}(t)|=\lim_{h\rightarrow 0}\frac{1}{h}\de_{\sX}(\gamma_t,\gamma_{t+h})
$
such that $|\dot{\gamma}(\cdot)|\in L^1(J,dt)$.
We denote with $AC^p(J,X)$ the set of all absolutely continuous curves that are defined on $J$ and such that the metric speed is in $L^p(J,dt)$. Then $AC^1(J,X)$ is the set of absolutely continuous curves.
 A Borel function $g:X\rightarrow [0,\infty]$ is an upper gradient of a continuous function
$u:X\rightarrow \mathbb{R}$ if for any absolutely continuous curve $\gamma:[0,1]\rightarrow X$ we have
\begin{align}\label{uppergradient}
|u(\gamma_0)-u(\gamma_1)|\leq \int_0^1g(\gamma(t))|\dot{\gamma}(t)|dt.
\end{align}
We say that a metric measure space $(X,\de_{\sX},\m_{\sX})$ supports a weak local $(q,p)$-Poincar\'e inequality with $1\leq p \leq q < \infty$ if (\ref{ichwerdnarrisch}) 
in Definition \ref{doub} holds for any 
continuous $u$ on $X$, 
for any $x\in X$ and $r>0$ such that $\m_{\sX}(B_r(x))>0$, and 
any upper gradient $g$ of $u$. We just have to replace $\sqrt{\Gamma(u)}$ by upper gradients $g$ of $u$.
The statements of remark \ref{remarkpoincare} hold as well.

\begin{definition}\label{localslope}
Let $u:X\rightarrow \mathbb{R}$ be a continuous function. The local slope (or local Lipschitz constant or pointwise Lipschitz constant) is the Borel function $\lip$ given by
\begin{align*}
\lip u(x)=\limsup_{y\rightarrow x}\frac{|u(y)-u(x)|}{\de_{\sX}(x,y)}.
\end{align*}
$\lip u$ is an upper gradient for $u$ \cite[proposition 1.11]{cheegerlipschitz}.
\smallskip

T. Rajala and M.-K. von Renesse proved the following results that we state only for $K\geq 0$.
\end{definition}
\begin{theorem}[\cite{rajala2}]\label{poincare}
Suppose that $(X,\de_{\sX},\m_{\sX})$ satisfies $CD(K,N)$ with $K\geq 0$ and $N\geq 1$. Then $(X,\de_{\sX},\m_{\sX})$ supports a weak local $(1,1)$-Poincar\'e inequality.
\end{theorem}
\begin{theorem}[\cite{renessepoincare}]\label{renesse}
Suppose that $(X,\de_{\sX},\m_{\sX})$ is nonbranching and satisfies $MCP(K,N)$ with $K\geq 0$ and $N\geq 1$. Then $(X,\de_{\sX},\m_{\sX})$ supports a weak local $(1,1)$-Poincar\'e inequality.
\end{theorem}
\begin{remark}\label{22poincare}
If a metric measure spaces satisfies a doubling property, Hajlasz and Koskela proved in \cite{koskela} that a weak local $(1,p)$-Poincar\'e inequality 
also implies a $(q,p)$-Poincar\'e inequality for $q<\scriptstyle{\frac{pN}{N-p}}$ if the doubling constant satisfies $C\leq 2^{\sN}$. This is the case if the space satisfies 
the condition $CD(0,N)$. In particular, $(X,\de_{\sX},\m_{\sX})$ supports a weak local $(2,2)$-Poincar\'e inequality. In the following we say that a metric measure space $X$ supports
a (weak) local Poincar\'e inequality if it supports a weak local $(1,1)$-Poincar\'e inequality.
\end{remark}
\smallskip
We want to define Sobolev spaces and a notion of modulus of a gradient on a suitable class of functions. There are several authors that gave different definitions (see \cite{cheegerlipschitz, shan, haj}).
Here, we follow the approach of Ambrosio, Gigli and Savar\'e. 
Their main result from \cite{agslipschitz} (see also \cite{agsheat}) states that under the Assumption \ref{assmms} most of the different approaches coincide and give us 
the same notion of Sobolev space and modulus of a gradient. The key is a non-trivial approximation by Lipschitz functions that we will use as starting point for our presentation. 
For any Borel function $u:X\rightarrow \mathbb{R}$ in $L^2(X,\m_{\sX})$ the Cheeger energy $\ChX(u)$ is defined by 
\begin{align}\label{liggett}
 \ChX(u)=\frac{1}{2}\inf\left\{\liminf_{h\rightarrow \infty}\int_{\sX}\left(\lip u_h\right)^2d\m_{\sX}: u_h \mbox{ Lipschitz}, \left\|u_h-u\right\|_{L^2(X,\m_{\sX})}\rightarrow 0\right\}.
\end{align}
Then the $L^2$-Sobolev space is given by $D(\ChX)=\left\{u\in L^2(\m_{\sX}):\ChX(u)<\infty\right\}.$
The associated norm is
$
\left\|u\right\|^2_{D(\ChX)}=\left\|u\right\|_{L^2}^2+2\ChX(u).
$
An important fact is that $\Ch$ is not a quadratic form in general. 
\begin{definition}
Let $(X,\de_{\sX},\m_{\sX})$ be a metric measure space. 
If the Cheeger energy $\ChX$ is a quadratic form, we call $(X,\de_{\sX},\m_{\sX})$ infinitesimal Hilbert.
\smallskip
\\
Another result from \cite{agslipschitz} is that $\ChX$ can be represented by
\begin{align}\label{shan}
\Ch^{\sX}(u)=\frac{1}{2}\int_{\sX}|\nabla u|_w^2d\m_{\sX}&\hspace{10pt}\mbox{ if }u\in D(\Ch)
\end{align}
and $+\infty$ otherwise where $|\nabla u|_w:X\rightarrow [0,\infty]$ is Borel measurable and called the minimal weak upper gradient of $u$. 
The notion of minimal weak upper gradient is motivated by the following definitions that we take from \cite{agsriemannian}.
\smallskip\\
We say that 
$u:X\rightarrow \mathbb{R}\cup \left\{\infty\right\}$ is ``Sobolev along almost every curve'' if $u\circ\gamma$ coincides a.e. in $[0,1]$ and in $\left\{0,1\right\}$ with an absolutely continuous map 
$u_{\gamma}:[0,1]\rightarrow \mathbb{R}$ for almost every curve $\gamma$.
The definition of the property "for almost every curve $\gamma$" can be found in \cite{agsriemannian}. We will not state it because it will not be used in the sequel.
\end{definition}
\begin{definition}
For $u$ that is Sobolev along almost every curve, a $\m_{\sX}$-measurable function $G:X\rightarrow[0,\infty]$ is a \textit{weak} upper gradient of $u$ if
\begin{align*}
|u(\gamma_0)-u(\gamma_1)|\leq \int_0^1G(\gamma(t))|\dot{\gamma}(t)|dt\hspace{5pt}\mbox{ for almost every curve $\gamma$}.
\end{align*}
Then any function $u\in D(\Ch)$ is Sobolev along a.e. curve and $\left|\nabla u\right|_{w}$ is a minimal weak upper gradient in the following sense:
$
\mbox{If $G$ is a weak upper gradient of $u$, then } \left|\nabla u\right|_{w}\leq G \hspace{5pt}\mbox{ $\m_{\sX}$-a.e. in $X$}.
$
\end{definition}
\begin{remark}
An upper gradient $g$ for some continuous $u$ is also a weak upper gradient in the sense of the previous definition. The converse is in general not true.
Hence, we have
$
\left|\nabla u\right|_w\leq \left|\lip u\right| \hspace{5pt}\mbox{a.e.}\ ,
$
but no equality in general. If we assume a doubling property and a local Poincar\'e inequality, there is the following result of Cheeger.
\end{remark}
\begin{theorem}\label{cheegerlipschitz}
If $(X,\de_{\sX},\m_{\sX})$ is a complete and intrinsic metric measure space that provides a doubling property and a local $(1,2)$-Poincar\'e inequality, then for any function $u:X\rightarrow \mathbb{R}$ 
that is locally Lipschitz, we have $\lip u=|\nabla u|_w$ $\m_{\sX}$-a.e.\ . 
\end{theorem}
\proof $\longrightarrow$ \cite{cheegerlipschitz}.
\smallskip
\\
The minimal weak upper gradient provides a stability property.
\begin{theorem}[Stability theorem, \cite{agslipschitz}]\label{stabilitytheorem}
Let $(X,\de_{\sX},\m_{\sX})$ be a complete and separable metric measure space. Let $u_n\in D(\Ch^{\sX})$ such that $u_n \rightarrow u\in L^2(\m_{\sX})$ pointwise a.e. and assume $|\nabla u_n|_w\in L^2(\m_{\sX})$ converges weakly to $g\in L^2(\m_{\sX})$.
Then $u\in D(\Ch^{\sX})$ and $g\geq|\nabla u|_w$ $\m_{\sX}$-a.e.\ .
\end{theorem}
\proof $\longrightarrow$ Theorem 5.3 in \cite{agslipschitz}.
\begin{remark}
In the introduction of \cite{agslipschitz} the authors remark that a complete and separable metric measure space 
$(X,\de_{\sX},\m_{\sX})$ whose balls have finite measure (hence, it fits in our Assumption \ref{assmms}), supports 
a weak local $(1,1)$-Poincar\'e inequality with constants $C>0$ and $\lambda\geq 1$ if and only if it holds for any Lipschitz function $u$ and 
upper gradient $\lip u$. 
%More precisely, for any Lipschitz function $u$ on $X$, 
%for any $x\in X$ and $r>0$ such that $\m_{\sX}(B_r(x))>0$, it holds
%\begin{align}\label{kaffeekocher}
%\int_{B_r(x)}|u-\bar{u}_{B_r(x)}| d\m_{\sX}\leq Cr\int_{B_{\lambda r}(x)}\lip ud\m_{\sX}.
%\end{align}
\end{remark}

\begin{remark}\label{locher}
If we assume the Cheeger energy $\ChX$ of $(X,\de_{\sX},\m_{\sX})$ to be a quadratic form, it yields a 
strongly local Dirichlet form $(\ChX,D(\ChX))$ on $L^2(X,\m_{\sX})$ where
the set of Lipschitz functions is dense in $D(\ChX)$ with respect to the Energy norm (see Proposition 4.10 in \cite{agsriemannian}). Additionally, if we assume
that the space $X$ is compact, Lipschitz function are dense in $C_0(X)$ with respect to uniform convergence by application of the Stone-Weierstra\ss\ Theorem.
Hence, $\ChX$ is a regular Dirichlet form and Lipschitz functions are a core. 
\end{remark}
\subsection{The Riemannian curvature-dimension condition}\begin{definition}[\cite{giglistructure, erbarkuwadasturm}, Riemannian curvature-dimension condition]
A metric measure space $(X,\de_{\sX},\m_{\sX})$ satisfies the Riemannian curvature-dimension condition $RCD^*(K,N)$ if 
it is infinitesimal Hilbert and satisfies the condition $CD^*(K,N)$.
\end{definition}
The definition of Riemannian curvature bounds was first introduced by Ambrosio, Gigli and Savar\'e in \cite{agsriemannian} 
for infinite dimension bounds in terms of the \textit{evolution variational inequality}. The finite dimensional counterpart was first considered by Gigli in \cite{giglistructure} 
where Laplace comparison estimates have been proved. 
The coherence of the finite and infinite dimensional setting was proved in \cite{agmr}. 
Finally, 
Erbar, Kuwada and Sturm established a unified definition of Riemannian curvature bounds in \cite{erbarkuwadasturm} in terms of a modified \textit{EV inequality} that depends also on a dimensional parameter.
We will not give the definition of EVI since it will not be used in this article.
\begin{proposition}[\cite{agsriemannian}]\label{stlo}
Assume $(X,\de_{\sX},\m_{\sX})$ satisfies $RCD^*(K,N)$ for $K\in\mathbb{R}$ and $N\geq 1$. 
Then the intrinsic distance $d_{\ChX}$ coincides with $\de_{\sX}$.
\end{proposition}

\begin{theorem}\label{mcp}
Let $(X,\de_{\sX},\m_{\sX})$ be a metric measure space that satisfies $RCD^*(K,N)$ for $K\in\mathbb{R}$ and $N\geq 1$. 
Then the space satisfies the measure contraction property $MCP(K,N)$.
\end{theorem}
\proof The theorem is a corollary of several results by Cavalletti, Gigli, Sturm and Rajala and can be found in this form in \cite{giglirajalasturm}. 
%The difference between this result and Theorem \ref{cavalletti} is that the non-branching-assumption is not needed anymore.
\begin{remark}\label{reg}
 Under the condition $RCD^*(K,N)$ several regularity  properties for the Markov semi-group $P_t$ 
 have been obtained in \cite{agsriemannian}. If $u\in D(\mathcal{E}^{\sX})$ and $\Gamma(u)\in L^{\infty}$, $P_tu$ has a
 Lipschitz representative, denoted by $\tilde{P}_tu$ (\cite[Theorem 6.1, Theorem 6.2]{agsriemannian}). Especially, any $u\in D(\mathcal{E}^{\sX})$ with $\Gamma(u)\in L^{\infty}$
 has a Lipschitz representative $\tilde{u}$ such that 
$
 \left|\nabla \tilde{u}\right| \leq \||\nabla u|_w\|_{L^{\infty}}.
$
 Under stronger conditions, namely $L^2\rightarrow L^{\infty}$-ultracontractivity, we even have 
 that $\tilde{P}_tu$ is Lipschitz for any $f\in L^2$ (\cite[Remark 6.4]{agsriemannian}). Especially, this is the case when the space satisfies $RCD^*(K,N)$.
 \\
 \\
We introduce the following regularity assumption for metric measure spaces $(X,\de_{\sX},\m_{\sX})$. Because of Remark \ref{reg} and Remark \ref{spätabendsinfinland} these properties are necessarily satisfied by $RCD^*$-spaces.
\end{remark}
\begin{assumption}\label{TheAss}
$(X,\de_{\sX},\m_{\sX})$ is a geodesic metric measure space satisfying $\supp\m_{\sX}=X$. 
In addition, every $u\in D(\ChX)$ with $|\nabla u|_w\leq 1$ a.e. 
admits a 1-Lipschitz representative.
\\
\\
The main result of Erbar, Kuwada and Sturm in \cite{erbarkuwadasturm} is
\end{assumption}
\begin{theorem}\label{theorembochner}
Let $(X,\de_{\sX},\m_{\sX})$ be a metric measure space that satisfies the condition $RCD^*(K,N)$. Then
\begin{itemize}
 \item[(1)] $BE(K,N)$ holds for $(\ChX,D(\ChX))$.
\end{itemize}
Moreover, if $(X,\de_{\sX},\m_{\sX})$ is a metric measure space that is infinitesimal Hilbert, satisfies the Assumption \ref{TheAss} and $(\ChX,D(\ChX))$ satisfies the condition $BE(K,N)$
then 
\begin{itemize}
 \item[(2)] $(X,\de_{\sX},\m_{\sX})$ satisfies $CD^*(K,N)$, i.e.
the condition $RCD^*(K,N)$.
\end{itemize}
\end{theorem}
\proof $\rightarrow$ \cite[Theorem 7, Theorem 4.1, Theorem 4.3, Theorem 4.8, Proposition 4.9]{erbarkuwadasturm}. 
\\
\\
The following theorem of Koskela and Zhou in \cite{koskelazhou} will be important later.
\begin{theorem}\label{theoremkoskelazhou}
Let $\mathcal{E}^{\sX}$ be a regular, strongly local and strongly regular symmetric Dirichlet form on $L^2(X,\m_{\sX})$. Suppose $(X,\de_{\mathcal{E}^{\sX}},\m_{\sX})$ satisfies a doubling property. 
Then $\lip(X)\subset D_{loc}(\mathcal{E})$, $\Gamma^{\sX}(u)$ exists for any $u\in \lip(X)$ and 
$\Gamma^{\sX}(u)\leq \lip(u)^2\hspace{4pt}\m_{\sX}\mbox{-a.e.}\ .$
\end{theorem}\begin{proposition}\label{nice}
Let $(X,\de_{\sX},\m_{\sX})$ be a metric measure space that satisfies $RCD^*(K,N)$ for $K>0$ and $N>1$. 
Then the spectrum of the associated Laplace operator $L^{\sX}$ is discrete and the first non-zero eigenvalue is $\geq K\frac{N}{N-1}$.
\end{proposition}
\begin{proof}
The condition $CD(K,N)$ implies a $(2^*,2)$-Sobolev inequality of the form
\begin{align}\label{sob}
\left(\int_{B_x(R)}u^{2^*}d\m_{\sX}\right)^{\frac{2}{2^*}}\leq A\int_{B_x(R)}u^2d\m_{\sX}+B\int_{B_x(R)}(\lip u)^2 d\m_{\sX},
\end{align}
where $2^*=\frac{2N}{N-2}$ for any 
Lipschitz function $u$ that is supported in a ball $B_x(R)$ (Theorem 30.23 in \cite{viltot}). Then, a Rellich-Kondrachov compactness Theorem is implied by results of 
Hailasz and Koskela \cite[Theorem 8.1]{koskela}.
Finally, the proof of the first statement works by induction exactly as for Riemannian manifolds (see \cite{berard}). 
\smallskip 
\\
The second statement directly comes from the Bakry-Emery characterization of the Riemannian curvature-dimension condition. 
Choose any eigenfunction $u$ with eigenvalue $\lambda$. Since $X$ is compact, an admissible test function is $\phi=1$. 
Then the condition $BE(K,N)$ implies
\begin{align*}
0&\geq \int_{\sX}\Gamma^{\sX}(u,L^{\sX}u)d\m_{\sX}+K\int_{\sX}\Gamma^{\sX}(u)d\m_{\sX}+\frac{1}{N}\int_{\sX}\left(L^{\sX}u\right)^2d\m_{\sX}\\
&=-\lambda\int_{\sX}\Gamma^{\sX}(u,u)d\m_{\sX}+\lambda K\int_{\sX}ud\m_{\sX}+\frac{\lambda^2}{N}\int_{\sX}u^2d\m_{\sX}=\int_{\sX}u^2d\m_{\sX}\left(-\lambda^2+\lambda K+\textstyle{\frac{\lambda^2}{N}}\right)
\end{align*}
from which follows $\lambda\geq K\frac{N}{N-1}$.
\end{proof}
\begin{remark}\label{onedimensionalcase}
The conclusion of the previous theorem is also true if ${K}=0$ and $N=1$. Then $\lambda_1\geq 1$. It follows since in this case
$F\simeq\lambda\mathbb{S}^1$ or $F\simeq \lambda [0,\pi]$ for some $0<\lambda\leq 1$. The diameter bound implies that $F$ is compact and there are points $x,y\in F$ 
such that $\diam_{\sF}=\de_{\sF}(x,y)$ and there is at least one geodesic between $x$ and $y$. Hence, 
the Hausdorff dimension has to be $1$ and $F$ consists of finitely many geodesic segments that connect $x$ and $y$ since the measure is assumed to be locally finite. 
But the curvature-dimension condition implies that there can be at most two geodesics.
\end{remark}
\section{$(K,N)$-cones and the Riemannian curvature dimension condition}
\subsection{Warped products and $(K,N)$-cones for metric measure spaces}
%We need to clarify if the notion of warped product that we used for Dirichlet forms coincides ``in some sense'' with the classical notion from metric geometry that we introduce now.
Let $(B,d_{\sB})$ and $(F,d_{\sF})$ be metric spaces that are complete, locally compact and (strictly) intrinsic. Let $f: B\rightarrow \mathbb{R}_{\geq0}$ be a locally Lipschitz function. 
We call a curve $\gamma=(\alpha,\beta)$ in $B\times F$ 
admissible if $\alpha$ and $\beta$ are absolutely continuous in 
$B$ and $F$ respectively. In that case we define
\begin{displaymath}
\mbox{L}(\gamma)=\int_0^1\sqrt{|\dot{\alpha}(t)|^2+(f\circ \alpha)^2(t) |\dot{\beta}(t)|^2}dt.
\end{displaymath}
$\mbox{L}$ is
a length-structure on the class of admissible curves
(for details see \cite{bbi} and \cite{albi0}). 
Then we can define a pseudo-distance between $(p,x)$ and $(q,y)$ by
$$
\inf \mbox{L}(\gamma)=: \de_{\sC}((p,x),(q,y))\in[0,\infty)
$$
where the infimum is taken over all admissible curves $\gamma$ that connect $(p,x)$ and $(q,y)$. 
The induced metric space $B\times_{f}F$ is called warped product of $(B,d_{\sB})$, $(F,d_{\sF})$ and $f:B\rightarrow\mathbb{R}_{\geq 0}$.
One can see that its topology coincides with the topology that was introduced for warped products in the setting of Dirichlet forms (see section 3.1).
By definition $B\times_f F$ is an intrinsic metric space. 
Completeness and local compactness follow from the corresponding properties
of $B$ and $F$. A suitable measure is given by $d\m_{\sC}=f^{\sN} d\m_{\sB}\otimes d\m_{\sF}$ and the corresponding metric measure space $(B\times_f F,\m_{\sC})=B\times_f^{\sN} F$ is called
$N$-warped product.
\begin{definition}[$(K,N)$-cones]\label{kncone}
For a metric measure space $(F,\de_{\sF},\m_{\sF})$ the $(K,N)$-cone is a metric measure space defined as follows:
\begin{itemize}
\item[$\diamond$]
$\Con_K(F):=
\begin{cases} F\times \left[0,\pi/\scriptstyle{\sqrt{K}}\right]/(F\times \left\{0,\pi/\scriptstyle{\sqrt{K}}\right\})&\mbox{ if }K>0\\
						   F\times [0,\infty)/(F\times \left\{0\right\})  &\mbox{ if }K\leq 0
                                                    \end{cases}
$
\item[$\diamond$] For $(x,s),(x',t)\in \Con_K(F)$
\begin{align*}&\de_{\Con_K}((x,s),(x',t))\\
&:=\begin{cases}\ck^{-1}\left(\ck(s)\ck(t)+K\sk(s)\sk(t)\cos\left({d}(x,x')\wedge\pi\right)\right))&\mbox{ if }K\neq 0\\
                                                  \sqrt{s^2+t^2-2st\cos\left({d}(x,x')\wedge\pi\right)}&\mbox{ if }K=0.
                                                 \end{cases}
\end{align*}
\item[$\diamond$] $\m^{\sN}_{\Con_K}=\sin_{\sK}^{\sN}tdt\otimes\m_{\sF}=\m_{\sC}$
\end{itemize}
where $\sk(t)$ is defined as in Example \ref{onedimensionalexample} and $\ck(t)=\cos(\sqrt{K}t)$ for $K>0$ and $\ck(t)=\cosh(\sqrt{-K}t)$ for $K<0$. The triple $(\Con_K(F),\de_{\Con_K},\m^{\sN}_{\Con_K})$ is denoted by $\Con_{K,N}(F)$.
If $\diam F\leq \pi$ and $F$ is a length space, the $(K,N)$-cone coincides with the $N$-warped product $I_K\times_{\sin_K}^{\scriptscriptstyle{N}} F$.
\end{definition}
\begin{remark}
The $(K,N)$-cone with respect to $(F,\de_{\sF})$ is an intrinsic (resp. strictly intrinsic) metric spaces if and only if $(F,\de_{\sF})$ is intrinsic (resp. strictly intrinsic) at distances less than $\pi$
(see \cite[Theorem 3.6.17]{bbi} for $K=0$). 
We also remark that away from the singularity points the $(K,N)$-cone metric is locally bi-Lipschitz equivalent to the euclidean product metric between $|\cdot-\cdot|$ and $\de_{\sF}$. 
\end{remark}
\begin{remark}
The $(K,N)$-cone is a metric measure space and we can consider its Cheeger energy that we denote with $\ChC$.
In general the intrinsic distance of 
$\ChC$ does not need to coincide with $\de_{\Con_K}$. Similar, we denote the Cheeger energy of $I_K\times_{\sin_K} ^{\scriptscriptstyle{N}}F$ with $\ChW$.
It is not clear if $\ChC$ will coincide with the $(K,N)$-cone $\mathcal{E}^{\sC}=B\times^{\sN}_{\sin_K}\ChF$ in the sense of Dirichlet forms.
In the first case, 
we define the length structure that yields an intrinsic distance that determines the Cheeger energy. In the second case, we define a Dirichlet form on $L^2({B}\times F,f^{\sN}d\vol_{\sB}\m_{\sF })$ that determines 
the intrinsic distance. We will address this problem in the next section.
\end{remark}
\subsection{On the relation between metric cones and cones in the sense of Dirichlet forms}\label{sec:versus}
Let $\ChF$ be the Cheeger energy of $(F,\de_{\sF},\m_{\sF})$ that is a locally compact, length metric measure space. If we assume that $\diam F\leq \pi$, then we have $I_K\times_{\sin_K}^{\sN}F=\Con_{N,K}(F)$.
We can also consider $I_K\times_{\sin_K}^{\sN} \ChF=\mathcal{E}^{\sC}$ in the sense of Dirichlet forms like in section 2 and 3 where $\ChF=\Ee$. 
We denote with $\Gamma^{\sC}$ the $\Gamma$-operator of $\mathcal{E}^{\sC}$ and with $|\nabla u|_w$ the minimal weak upper gradient with respect to $\ChC$.
The underlying topological space of $\mathcal{E}^{\sC}$ is by definition $I_K\times F/_{\sim}$ as it was defined in section 3.1. If $\mathcal{E}^{\sF}$ is strongly local, regular and strongly regular and closed balls
are compact, the same properties also hold for $\mathcal{E}^{\sC}$.
We want to analyze the intrinsic distance of $\mathcal{E}^{\sC}$ in more detail. The key result is the following proposition.
\begin{proposition}\label{dontknow}
Let $(F,\de_{\sF},\m_{\sF})$ be a locally compact length metric measure space that satisfies volume doubling and supports a local Poincar\'e inequality. Assume $\diam F\leq \pi$.
Then $D(I_K\times_{\sin_K}^{\sN}\ChF)\subset D(\ChC)$ and for any $u\in D(I_K\times_{\sin_K}^{\sN}\ChF)$ we have
\begin{align}
\left|\nabla u\right|^2_w&\leq \Gamma^{I_K,\sin_K^{\sN}}(u^x)+\textstyle{\frac{1}{\sin_K^2}}|\nabla u^p|_w\hspace{5pt}\m_{\sC}\mbox{-a.e.}
\end{align}
where $u^x(r)=u(r,x)$ and $u^r(x)=u(r,x)$. Especially, the result holds if $(F,\de_{\sF},\m_{\sF})$ satisfies the condition $RCD^*(N-1,N)$.
\end{proposition}
\begin{proof}We follow the proof of Lemma 6.12 in \cite{agsriemannian} and use the following elementary lemma in \cite{agsgradient}:
\begin{lemma}
Let $d(s,t):(0,1)^2\rightarrow \mathbb{R}$ be a map that satisfies
\begin{align*}
|d(s,t)-d(s',t)|\leq |v(s)-v(s')|,\hspace{5pt}|d(s,t)-d(s,t')|\leq |v(t)-v(t')|
\end{align*}
for any $s,t,s',t'\in (0,1)$, for some locally absolutely continuous map $v:(0,1)\rightarrow \mathbb{R}$ and let $\delta(t):=d(t,t)$. Then $\delta$ is locally absolutely continuous in $(0,1)$ and
\begin{align*}
\frac{d}{dt}\delta|_t\leq \limsup_{h\rightarrow 0}\frac{d(t,t)-d(t-h,t)}{h}+\limsup_{h\rightarrow 0}\frac{d(t,t+h)-d(t,t)}{h}\hspace{5pt}dt\mbox{-a.e. in }(0,1).
\end{align*}
\end{lemma}
\proof $\rightarrow$\cite[Lemma 4.3.4]{agsgradient}\\
\\
\textbf{1.}\ 
Consider $u\in C^{\infty}_0(\hat{I}_K)\otimes \lip(F)$. $u$ is also Lipschitz with respect to $\Con_{N,K}(F)$. 
Let $\gamma=(\alpha,\beta): [0,1] \rightarrow (\hat{I}_{\sK})_{\varepsilon}\times F$ 
be a curve in $AC^2(\Con_{\sK,\sN}(F))$ where $(\hat{I}_{\sK})_{\varepsilon}=I_{\sK}\backslash (\partial I_{\sK})_{\varepsilon}.$ Then,
one can check that $\alpha\in AC^2(\hat{I}_{\sK})$ and $\beta\in AC^2(F,\de_{\sF})$ and there is $g\in L^2((0,1),dt)$ such that 
\begin{align*}
\de_{F}(\beta_t,\beta_s)\leq \int_s^tg(\tau)d\tau\hspace{10pt}\mbox{and}\hspace{10pt}|\alpha_t-\alpha_s|&\leq \int_s^tg(\tau)d\tau
\end{align*}
For $K>0$ we have the following estimates (and similar for any $K\in\mathbb{R}$).
\begin{align}
\de_{\Con_K}((r,y),(r,x))&=\cos_K^{-1}(\cos^2r+\sin^2r\cos|x,y|)\nonumber\\
&=\cos^{-1}_K(1-\sin^2r(1-\cos|x,y|))\leq \cos^{-1}_K(1-\textstyle{\frac{1}{2}}\sin^2r|x,y|^2)\label{itsch}\\
\cos_K^{-1}(1-\textstyle{\frac{1}{2}}x^2)&= x+o(x^2)\hspace{5pt}\mbox{for}\hspace{5pt}x\rightarrow 0\label{ni}\end{align}
Then we can see that 
\begin{align*}
|u(\alpha_s,\beta_t)-u(\alpha_s,\beta_{t'})|&\leq L \de_{\Con_K}((\alpha_s,\beta_t),(\alpha_s,\beta_{t'}))\\
&\leq L\cos^{-1}_K(1-{\textstyle\frac{1}{2}}\sin^2\alpha_s\de_{\sF}(\beta_t,\beta_{t'})^2)\leq C \de_{\sF}(\beta_t,\beta_{t'})\leq C \int_s^tg(\tau)d\tau\\
|u(\alpha_s,\beta_t)-u(\alpha_{s'},\beta_{t})|&\leq L \de_{\Con_K}((\alpha_s,\beta_t),(\alpha_s,\beta_{t}))\leq L|\alpha_{s}-\alpha_{s'}| \leq C \int_s^tg(\tau)d\tau
\end{align*}
where $L$ is a Lipschitz constant of $u$ and $C>0$ is another constant. Hence, we can apply the lemma from above and obtain
\begin{align*}
\left|\frac{d}{dt} (u\circ \gamma)(t)\right|
%\\
&\leq \limsup_{h\rightarrow 0}\frac{|u(\alpha_{t-h},\beta_t)-u(\alpha_t,\beta_t)|}{h}+\limsup_{h\rightarrow 0}\frac{|u(\alpha_t,\beta_{t+h})-u(\alpha_t,\beta_t)|}{h}
\end{align*}
for a.e. $t\in[0,1]$. 
By definition of the local Lipschitz constant $\lip$ and by the elementary estimate $2ab\leq a^2+b^2$ for any $a,b\in\mathbb{R}$, it follows that
\begin{align*}
\left|\frac{d}{dt} (u\circ \gamma)(t)\right|&\leq\lip u^{\beta_t}(\alpha_t)|\dot{\alpha}(t)|+\lip u^{\alpha_t}(\beta_t)|\dot{\beta}(t)|\\
&\leq
{\sqrt{(\lip u^{\beta_t})^2(\alpha_t)+\textstyle{\frac{1}{\sin_K^2(\alpha_t)}}(\lip u^{\alpha_t})^2(\beta_t)}}
\sqrt{|\dot{\alpha}(t)|^2+\sin^2_K(\alpha_t)|\dot{\beta_t}|^2} =:G(\gamma(t))|\dot{\gamma}(t)|
\end{align*}
for a.e. $t\in [0,1]$. If we want to check that $G$ is a weak upper 
gradient of $u$, we only need to consider curves like above since $u$ has compact support in $\hat{I}_{\sK}\times_{\sin_{\sK}}^{\sN} F$. Hence, integration with respect to $t$ on both sides shows
that $G
$ is a weak upper gradient of $u$. It follows
\begin{align}
\left|\nabla u\right|_w(r,x)&\leq G(r,x)\label{pf0}
\hspace{5pt}\m_{\sC}\mbox{-a.e.}\ .
\end{align}
Since $(F,\de_{\sF},\m_{\sF})$ satisfies a volume doubling property and supports a local Poincar\'e inequality,
Cheeger's theorem (Theorem \ref{cheegerlipschitz}) states that $\lip u^{r}=|\nabla u^r|_w$. Then the square of the right hand side of (\ref{pf0}) $\m_{\sC}$-a.e. equals 
\begin{align*}
%|\nabla u^x|_w^2+\textstyle{\frac{1}{\sin_K^2}}|\nabla u^r|^2_w=
((u^x)')^2+\textstyle{\frac{1}{\sin_K^2}}|\nabla u^r|^2_w=\Gamma^{I_K,\sin_K^{\sN}}(u^x)+\textstyle{\frac{1}{\sin_K^2}}|\nabla u^r|^2_w=\Gamma^{\sC}(u).
\end{align*}
\textbf{2.}\ 
By the definition of skew products, $C^{\infty}_0(\hat{I}_K)\otimes D(\mathcal{E}^{\sF})$ is a dense subset of $D(I_K\times_{\sin_K}^{\sN}\mathcal{E}^{\sF})$. 
Hence, for any $u\in D(I_K\times_{\sin_K}^{\sN}\mathcal{E}^{\sF})$ there is a sequence $u_n\in C^{\infty}_0(\hat{I}_K)\otimes D(\mathcal{E}^{\sF})$ that
converges to $u$ with respect to the energy norm of $I_K\times_{\sin_K}^{\sN}\mathcal{E}^{\sF}$, and since $\Gamma^{\sC}$ is a continuous bilinear form
% from $D(I_K\times_{\sin_K}^{\sN}\mathcal{E}^{\sF})^2$ to $L^1(\m_{\Con_{\sN,\sK}})$
, we will find a subsequence such that 
\begin{align*}
\Gamma^{I_K,\sin_K^{\sN}}(u_{n_i}^x)+\textstyle{\frac{1}{\sin_K^2}}|\nabla u_{n_i}^r|^2_w=\Gamma^{\sC}(u_{n_i})\longrightarrow\Gamma^{\sC}(u)=\Gamma^{I_K,\sin_K^{\sN}}(u^x)+\textstyle{\frac{1}{\sin_K^2}}|\nabla u^r|^2_w \hspace{10pt}\m_{\sC}\mbox{-a.e.}
\end{align*}
The left hand side of (\ref{pf0}) converges weakly in $L^2(\m_{\sC})$ (after taking another subsequence) and the limit is a weak upper gradient of $u$. 
This follows from the stability theorem for minimal weak upper gradients in \cite{agslipschitz} (see Theorem \ref{stabilitytheorem}). 
More precisely, we can argue as follows. Since $|\nabla u_n|_w\in L^2(\m_{\sX})$ is a bounded sequence, 
we find a subsequence that converges weakly to $g\geq|\nabla u|_w\in L^2(\m_{\sX})$. In particular, we have convergence for any test function 
%it follows that
%\begin{align*}
%\int_{C}\left|\nabla u\right|_w(r,x)\phi(r,x)d\m_{\sC}(r,x)&\leq \int_{C}\sqrt{(\lip u^{x})^2(r)+\textstyle{\frac{1}{\sin_K^2}}(\lip u^{r})^2(x)}\phi(r,x)d\m_{\sC}(r,x)
%\end{align*}
%for any 
$\phi\in L^2(\m_{\sC})$ such that $\phi\geq 0$.
Hence, inequality (\ref{pf0}) is preserved $\m_{\sC}$-a.e. an we have
\begin{align}
\left|\nabla u\right|^2_w(r,x)&\leq \Gamma^{I_K,\sin_K^{\sN}}(u^x)(r)+\textstyle{\frac{1}{\sin_K^2}}|\nabla u^r|_w^2(x)\hspace{10pt}\m_{\sC}\mbox{-a.e.}
\end{align}
and in particular, $u\in D(I_K\times_{\sin_K}^{\sN}\Ch_{\sF})$ implies $u\in D(\Ch_{K,F})$.
\end{proof}

\begin{lemma}\label{books}
Let $(F,\de_{\sF},\m_{\sF})$ satisfy $RCD^*(N-1,N)$  for $N\geq 1$ and $\diam F\leq \pi$. Let $\Con_{\sN,\sK}(F)$ be the corresponding $(K,N)$-cone for $K\geq 0$.
Then $\Con_{\sN,\sK}(F)$ satisfies a volume doubling property and supports a local Poincar\'e inequality. 
\end{lemma}
\begin{proof} We assume $N>1$ since the case $N=1$ is already clear by Remark \ref{onedimensionalcase}.
We will use the following theorem of Ohta from \cite{ohtpro}.
\begin{theorem}\label{mcpfor0cones}
{If the metric measure space $(F,\de_{\sF},\m_{\sF})$ satisfies $MCP(N-1,N)$ in the sense of Ohta and if $\diam_{\sF}\leq \pi$ then the associated $(0,N)$-cone satisfies $MCP(0,N+1)$ in the sense of Ohta.}
\end{theorem}
\noindent
Hence, in the case $K=0$ we proceed as follows. When $(F,\de_{\sF},\m_{\sF})$ satisfies $RCD^*(N-1,N)$, Theorem \ref{mcp} implies $MCP(N-1,N)$ 
that implies a measure contraction property $MCP(0,N+1)$ for $\Con_{\sN,0}(F)$ in the sense of Ohta.
In particular, $\Con_{\sN,\sK}(F)$ satisfies a volume doubling property by results of Ohta in \cite{ohtmea} and supports a local Poincar\'e inequality by Theorem \ref{renesse} and Theorem \ref{mcpfor0cones}.
The latter follows since the condition $RCD^*(N-1,N)$ implies that for
every $x\in {F}$ and $\m_{\sX}$-a.e. $y\in{F}$ there is a unique geodesic. This property is inherited by the cone since $\diam_{\sF}\leq\pi$. Hence, $MCP$ \`a la Ohta is the same as $MCP$ \`a la Sturm and we can apply Theorem \ref{renesse} by von Renesse. 
\smallskip

The case $K>0$ can be covered in the same way. Assume without loss of generality that $K=1$. By following straightforwardly Ohta's proof of Theorem \ref{mcpfor0cones} in \cite{ohtpro} we can prove the analogous result for $(1,N)$-cones where one should use the following formula for the projection of a geodesic $\gamma=(\alpha,\beta):[0,1]\rightarrow \Con_{N,1}(F)\backslash \left\{\mbox{singularities}\right\}$ to $[0,\pi]$.
\begin{align*}
\cos\alpha(t)=\sigma_{1,1}^{(1-t)}(\mbox{L}(\gamma))\cos\alpha(0)+\sigma_{1,1}^{(t)}(\mbox{L}(\gamma))\cos\alpha(1).
\end{align*}
Alternatively, one can use Theorem \ref{mcpfor0cones} directly and compare the metric and the measure of the spherical cone around the origin with the metric of the Euclidean cone around the origin. 
More precisely, one can find constants $m,M>0$ such that
\begin{align*}
\frac{1}{M}\de_{\Con_{\sK}}\leq \de_{\Con_{0}} \leq \frac{1}{m}\de_{\Con_{\sK}}\ \mbox{ and } \ \frac{1}{M}\sin_{\sK}^Nr\leq r^N \leq \frac{1}{m}\sin_{\sK}^Nr.
\end{align*}
From this estimates one can easily deduce the doubling property and the Poincar\'e inequality in a neighborhood of the origin from the corresponding results for 
the $0$-cone. Away from the singularities the same argument works by comparison with the direct product $(I_{\sK}\times F, \de_{Eukl}\times \de_{\sF},\mathcal{L}^1\otimes \m_{\sF})$. 
\end{proof}
\begin{theorem}\label{intrinsic}
Let $(F,\de_{\sF},\m_{\sF})$ be a metric measure space satisfying $RCD^*(N-1,N)$ for $N\geq1$ and $\diam F\leq \pi$.
Then the intrinsic distance $\de_{\mathcal{E}^{\sC}}$ of $\mathcal{E}^{\sC}=I_{\sK}\times^{\sN}_{\sin_{\sK}}\Ch^{\sF}$ coincides with $\de_{\Con_{\sK}}$.
\end{theorem}
\begin{proof} 
By remark \ref{diameterbound} we know that in any case $\diam F\leq \pi$, thus $I_{\sK}\times^{\sN}_{\sin_{\sK}} F=\Con_{\sN,\sK}(F)$.
We only check the case $K>0$.
\smallskip\\ \textbf{1.}\ 
We know from Proposition \ref{dontknow} that $D(I_{\sK}\times_{\sin_{\sK}}^{\sN}\ChF)\subset D(\ChC)$ and for any $u\in D(I_{\sK}\times_{\sin_{\sK}}^{\sN}\ChF)$
\begin{align}
\left|\nabla u\right|^2_w&\leq \Gamma^{I_{\sK},\sin_{\sK}^{\sN}}(u^x)+\textstyle{\frac{1}{\sin_{\sK}^2}}|\nabla u^r|_w\hspace{10pt}\m_{\sC}\mbox{-a.e.}\label{pf}
\end{align}
where $u^x(r)=u(r,x)$ and $u^r(x)=u(r,x)$.
For the intrinsic distance of $\mathcal{E}^{\sC}$ we need to consider $u\in\mathcal{L}_{C,loc}=C(I_{\sK}\times F/_{\sim},\mathcal{O}_{\sC})\cap \mathcal{L}_{loc}$ where 
\begin{align*}
\mathcal{L}_{loc}:=\left\{\psi\in D_{loc}(I_{\sK}\times_{\sin_{\sK}}^{\sN}\ChF):\sqrt{\Gamma^{\sC}(\psi)}\leq 1 \m_{\sC}\mbox{-a.e. in }I_{\sK}\times F/_{\sim}\right\}.
\end{align*}
One has to prove that $u$ is $1$-Lipschitz with respect to $\de_{\Con_{\sK}}$. We will follow an argument that was suggested to the author by Tapio Rajala. 
\smallskip

First, $\Gamma^{\sC}(u)\leq 1$ $\m_{\sC}$-a.e. implies $|\nabla u|_w\leq 1$ $\m_{\sC}$-a.e. by (\ref{pf}). 
$|\nabla u|_w$ is a weak upper gradient and $\Con_{\sN,\sK}(F)$ satisfies the measure contraction property $MCP(N,N+1)$ by the proof of the previous lemma.
Consider two points $p,q\in \Con_{\sN,\sK}(F)$, $B_{\epsilon}(q)\subset \Con_{\sN,\sK}(F)$, $\mu_0=\m_{\sC}(B_{\epsilon}(q))^{-1}\m_{\sC}|_{B_{\epsilon}(q)}$ 
and the unique optimal displacement interpolation $\mu_t$ between $\mu_0=\mu$ and $\mu_1=\delta_p$. Let $\Pi$ be the corresponding dynamical transference plan. 
Because of the measure contraction property $(\mu_t)_{t\in [0,t_0]}$ is a $2$-test plan for any $t_0<1$. Hence 
\begin{align*}
%|f(p)-\int f\mu_0|&\leq 
\int|u(\gamma_1)-u(\gamma_0)|d\Pi(\gamma)\leq \int\int_0^1|\nabla u|_w(\gamma(t))\mbox{L}(\gamma)dtd\Pi(\gamma)\leq \de_{W}(\delta_p,\mu_1).
\end{align*}
where $\de_{W}$ is the $L^2$-Wasserstein metric of $\Con_{\sN,\sK}(F)$.
In the last inequality we just use that $\mu_t\leq C(t)\m_{\sC}$ for some $C(t)>0$ and any $t<1$ and $|\nabla u|_w\leq 1$ $\m_{\sC}$-a.e.\ .
If $\epsilon \rightarrow 0$, we obtain
\begin{align*}
|u(p)-u(q)|&\leq\de_{W}(\delta_p,\delta_q)=\de_{\Con_{\sK}}(p,q).
\end{align*}
This yields
\begin{align}\label{keinname2}
\de_{\mathcal{E}^{\sC}}((s,y),(r,x))=\sup\left\{u(s,y)-u(r,x):u\in\mathcal{L}_{C,loc}
\right\}
\leq \de_{\Con_{\sK}}((r,x),(s,y))
\end{align}
for all $(r,x),(s,y)\in \Con_{\sN,\sK}(F)$.
\smallskip\\ \textbf{2.}\ 
On the other hand, we define $g((p,x))=\de_{\Con_{\sK}}((p,x),(q,y))$ for some $(q,y)\in I_{\sK}\times_{\sin_{\sK}} F$ where 
\begin{align*}
\de_{\Con_{\sK}}((p,x),(q,y))=\ck^{-1}\underbrace{\left(\ck(p)\ck(q)+K\sk(p)\sk(q)\cos\de_{\sF}(x,y)\right)}_{=:h(p,x)}.
\end{align*}
$h\in D_{loc}(\mathcal{E}^{I_{\sK},\sin_{\sK}^{\sN}})\otimes D(\ChF)$ since 
$\cos_{\sK}, \sin_{\sK}\in D_{loc}(\mathcal{E}^{I_{\sK},\sin_{\sK}^{\sN}})$ and $\cos\de_{\sF}(\cdot,q), 1\in D(\ChF)$. 
We can calculate $\Gamma^{\sC}(g)$ explicitly. We get 
\begin{align*}
\Gamma^{\sC}(g)&=\left(\left(\cos_{\sK}^{-1}\right)'(h(p,x))\right)^2\Gamma^{\sC}(h)(p,x)=\frac{1}{1-h^2(p,x)}\Gamma^{\sC}(h)(p,x)
\end{align*}
Then, a straightforward calculation using the chain rule and $\Gamma^{\sF}(\de_{\sF}(\cdot,y))\leq 1$ yields
\begin{align*}
&\Gamma^{\sC}(h)(p,x)=\Gamma^{I_{\sK}}(\cos_{\sK}p\cos_{\sK}q)+2\Gamma^{I_{\sK}}(\cos_{\sK}p\cos_{\sK}q,\sin_{\sK}p\sin_{\sK}q)\cos\de_{\sF}(x,y)\\
&\ \ \ \ \ \ \ +\Gamma^{I_{\sK}}(\sin_{\sK} p\sin_{\sK} q)\cos^2\de_{\sF}(x,y)+\frac{\sin_{\sK}^2q\sin_{\sK}^2p}{\sin_{\sK}^2p}\Gamma^{\sF}(\cos\de_{\sF}(x,y))
%&=\cos^2_{\sK}q\sin^2_{\sK} p-2\cos_{\sK}p\cos_{\sK}q\sin_{\sK}p\sin_{\sK}q\cos\de_{\sF}(x,y)\\
%&\hspace{40pt}+\cos^2_{\sK}p\sin^2_{\sK}q\cos^2\de_{\sF}(x,y)+\sin^2_{\sK}q\sin^2\de_{\sF}(x,y)underbrace{\Gamma^{\sF}(\de_{\sF}(x,y))}_{\leq 1}\\
%
\leq
%\cos^2_{\sK}q\sin^2_{\sK} p-2\cos_{\sK}p\cos_{\sK}q\sin_{\sK}p\sin_{\sK}q\cos\de_{\sF}(x,y)\\
%&\hspace{40pt}+\underbrace{\cos^2_{\sK}p}_{1-\sin_{\sK}^2p}\sin^2_{\sK}q\cos^2\de_{\sF}(x,y)+\sin^2_{\sK}q\sin^2\de_{\sF}(x,y)\\
%&=\cos^2_{\sK}q\sin^2_{\sK} p+\sin^2_{\sK}q-2\cos_{\sK}p\cos_{\sK}q\sin_{\sK}p\sin_{\sK}q\cos\de_{\sF}(x,y)\\
%&\hspace{40pt}-\sin_{\sK}^2p\sin^2_{\sK}q\cos^2\de_{\sF}(x,y)\\
%&=
%1-\cos^2_{\sK}q\cos^2_{\sK} p-2\cos_{\sK}p\cos_{\sK}q\sin_{\sK}p\sin_{\sK}q\cos\de_{\sF}(x,y)\\
%&\hspace{40pt}-\sin_{\sK}^2p\sin^2_{\sK}q\cos^2\de_{\sF}(x,y)=
1-h^2(p,x)
\end{align*}
Hence
$\Gamma^{\sC}(g)\leq 1$, $g\in \mathcal{L}_{C,loc}$ and 
$$g((p,x))-g((q,y))=g((p,x))=\de_{\Con_{\sK}}((p,x),(q,y))\leq \de_{\mathcal{E}^{\sC}}((p,x),(q,y))$$
by definition of $\de_{\mathcal{E}^{\sC}}$.
Hence, we obtain that $\de_{\mathcal{E}^{\sC}}=\de_{\Con_{\sK}}$.
% on $\hat{I}_{\sK}\times F$. But since $\de_{\Con_{\sK}}$ is uniformly continuous with respect to itself, 
%there exists a unique $\de_{\Con_{\sK}}$-continuous extension of $\de_{\mathcal{E}^{\sC}}$ that has to coincide with $\de_{\Con_{\sK}}$ on $I_{\sK}\times_{\sin_{\sK}}^{\sN} $.
\end{proof}
\begin{corollary}\label{maincor} 
Let $(F,\de_{\sF},\m_{\sF})$ be a metric measure space satisfying $RCD^*(N-1,N)$ for $N\geq1$ and $\diam F\leq \pi$. 
Then $I_{\sK}\times_{\sin_{\sK}}^{\sN}\ChF=\ChC$.
\end{corollary}
\begin{proof}
$\de_{\Con_{\sK}}=\de_{\mathcal{E}^{\sC}}$ by Theorem \ref{intrinsic} and $\de_{\Con_{\sK}}$ induces the topology of the underlying space $I_{\sK}\times F/\sim$. 
Theorem \ref{books} implies the doubling property for $\Con_{N,K}(F)$. Then, by Theorem \ref{theoremkoskelazhou} and Proposition \ref{dontknow} we get that any Lipschitz function $u$ with respect to $\de_{\Con_{K}}$ is in 
$D_{loc}(I_{\sK}\times_{\sin_{\sK}}^{\sN}\ChF)$ and
\begin{align}\label{together}
\Gamma^{\sC}(u)=\lip(u)=\left|\nabla u\right|^2_w \hspace{10pt}\m_{\sC}\mbox{-a.e.}\ .
\end{align}
By the definition of the Cheeger energy this implies the result.
%We can integrate (\ref{together}) with respect to $\m_{\sC}$ and we obtain for 
%any Lipschitz function $u\in D(\ChC)$
%$$\mathcal{E}^{\sC}(u)\leq \ChC(u)\in[0,\infty).$$
%By the relaxation formula (\ref{liggett}) and by the lower semi-continuity of the Dirichlet form $\mathcal{E}^{\sC}$ we obtain that 
%$D(\ChC)\subset D(\mathcal{E}^{\sC})$. Together with Proposition \ref{dontknow} this yields the result.
\end{proof}
\paragraph{The case when $\Con_{N,K}(F)$ satisfies $RCD^*(KN,N+1)$}
%The main results of this paragraph will be that the metric measure space $(F,\de_{\sF},\m_{\sF})$ is infinitesimal Hilbertian, $\de_{\ChF}=\de_{\sF}$, $\de_{\ChC}=\de_{\Con_{\sK}}$ and $\ChC=I_{\sK}\times_{\sin_{\sK}}^{\sN}\ChF$ provided that 
%$\Con_{\sN,\sK}(F)$ satisfies $RCD(KN,N+1)$. In particular, $\ChF$ is strongly regular.
\begin{remark}\label{diameterbound}
We remind the reader on a result by Bacher and Sturm from \cite{bastco}. They show the following. If the $(K,1)$-cone over some $1$-dimensional space satisfies $CD(0,2)$ then the 
diameter of the underlying space is bounded by $\pi$. It is easy to see that their proof can be extended to any $(K,N)$-cone of any dimension bound $N$ and any parameter $K$. Thus, 
if $\Con_{N,K}(F)$ satisfies $RCD^*(KN,N+1)$, then $\diam F\leq \pi$, and
again $\Con_{N,K}(F)=I_{\sK}\times_{\sin_{\sK}}^{\sN} F$.
\end{remark}

\begin{lemma}\label{intrinsic2}
Let $(F,\de_{\sF},\m_{\sF})$ be a metric measure space that satisfies a volume doubling property and supports a Poincar\'e inequality. Assume 
$\Con_{\sN,\sK}(F)$ satisfies $RCD^*(KN,N+1)$ for $K\geq 0$ and $N\geq 1$.
Then the intrinsic distance $\de_{\mathcal{E}^{\sC}}$ of $\mathcal{E}^{\sC}=I_{\sK}\times^{\sN}_{\sin_{\sK}}\Ch^{F}$ coincides with $\de_{\Con_{\sK}}$.
\end{lemma}

\begin{proof}
Since $F$ satisfies a volume doubling property, supports a local Poincar\'e inequality and is infinitesimal Hilbertian, we can apply Proposition \ref{dontknow}. 
Then, we have for any $u\in D(I_{\sK}\times_{\sin_{\sK}}^{\sN}\ChF)$ 
\begin{align}
\left|\nabla u\right|^2_w&\leq \Gamma^{\sC}(u)\hspace{5pt}\m_{\sC}\mbox{-a.e.}\ .
\end{align}
$\Con_{\sN,\sK}(F)$ satisfies a Riemannian curvature-dimension condition. Hence, $|\nabla u|_w\leq \sqrt{\Gamma^{\sC}(u)}\leq 1$ $\m_{\sC}$-a.e. implies 
$u$ is $1$-Lipschitz and (\ref{keinname2}) holds. We can proceed as in the proof of Theorem \ref{intrinsic} and obtain that $\de_{\mathcal{E}^{\sC}}=\de_{\Con_{\sK}}$.
\end{proof}
\begin{corollary}\label{maincor2}
Let $(F,\de_{\sF},\m_{\sF})$ be a metric measure space that satisfies a volume doubling property and supports a Poincar\'e inequality. Assume 
$\Con_{\sN,\sK}(F)$ satisfies $RCD^*(KN,N+1)$ for $K\geq 0$ and $N\geq 1$. Then $I_{\sK}\times_{\sin_{\sK}}^{\sN}\ChF=\ChC$.
\end{corollary}

\begin{proof} We can follow the proof of Corollary \ref{maincor}.
\end{proof}

\begin{lemma}\label{stilldontknow}
Let $(F,\de_{\sF},\m_{\sF})$ be a metric measure space. Assume the $(K,N)$-cone $\Con_{\sN,\sK}(F)$ satisfies $RCD^*(KN,N+1)$ for $N\geq 1$ and $K\geq 0$. 
Then $(F,\de_{\sF},\m_{\sF})$ satisfies a volume doubling property, supports a local Poincar\'e inequality and $(F,\de_{\sF},\m_{\sF})$ is infinitesimal Hilbertian.
\end{lemma}
\begin{proof}We prove the result for $K>0$. The general case follows in the same way.
Consider
$$x\in F\mapsto (1,x)\in \left\{1\right\}\times F\subset \Con_{\sN,\sK}(F).$$
We can find constants $M>m>0$ such that 
\begin{align*}
\frac{1}{M}\de_{\Con_{\sK}}\leq \max\left\{|\cdot-\cdot|,\de_{\sF}\right\}\leq \frac{1}{m}\de_{\Con_{\sK}}\ \& \ \ \frac{1}{M}\sin_{\sK}^{\sN}\leq 1 \leq \frac{1}{m}\sin_{\sK}^{\sN}
\end{align*}
in an $\epsilon$-neighborhood of $\left\{1\right\}\times F\subset \Con_{\sN,\sK}(F)$. On the one hand, from this we can easily deduce the volume doubling property for $F$.
Pick a point $x\in F$ and let $r>0$. Then
\begin{align*}
4r\m_{\sF}(B_{2r}(x))&\leq \frac{1}{m}\sin_{\sK}^{\sN}rdr\otimes \m_{\sF}([-2r+1,2r+1]\times B_{2r}(x))\\
&\leq \frac{1}{m}\m_{\sC}(B_{2Mr}(1,x))\leq \frac{1}{m}C\left(\frac{2M}{m}\right)^{\sN}\m_{\sC}(B_{mr}(1,x)) \\
&\leq  \frac{1}{m}C\left(\frac{2M}{m}\right)^{\sN}\m_{\sC}([-r+1,r+1]\times B_{r}(x))\\
&\leq  2^NC\left(\frac{M}{m}\right)^{\sN+1}2r\m_{\sF}(B_{2r}(x)).
\end{align*}
We used the volume doubling property of $\Con_{\sN,\sK}(F)$ in the third inequality.
We also obtain that the space $F$ supports a weak local Poincar\'e inequality because of the bi-Lipschitz invariance of this property. For example, we can follow the
method that is provided in Section 4.3 of \cite{bjoern}.
\smallskip

Now, we will check that $F$ is infinitesimal Hilbertian.
For any Lipschitz function $u$ on $\Con_{\sN,\sK}(F)$ we see
\begin{align}
(\lip u)(r,x)
&=\limsup_{(s,y)\rightarrow (r,x)} \frac{|u(s,y)-u(r,x)|}{\de_{\Con_{\sK}}((s,y),(r,x))}\nonumber\\
&\geq\limsup_{(s,y)\rightarrow (r,x),r=s} \frac{|u(r,y)-u(r,x)|}{\de_{\Con_{\sK}}((r,y),(r,x))}\nonumber\\
&\geq\limsup_{y\rightarrow x} \frac{|u^r(y)-u^r(x)|}{\sin_{\sK}(r)|x,y|}= \frac{1}{\sin_{\sK}(r)}\lip u^r(x).\label{rrrr}
\end{align}
The second last inequality comes from (\ref{itsch}) and (\ref{ni}).
Following the steps in paragraph 1 of the proof of Proposition \ref{dontknow} we can see that (\ref{pf0}) holds for 
$C_{\sK}^{\infty}(\hat{I}_{\sK})\otimes u$ where $u\in \lip(F)$.  There, we did not use that $F$ is infinitesimal Hilbertian. By locality of the minimal weak upper gradient (\ref{pf0}) also holds for 
$1\otimes u$. Then (\ref{pf0}) and (\ref{rrrr}) imply 
\begin{align}
\lip (1\otimes u)(r,x)=\textstyle{\frac{1}{\sin_{\sK}r}}\lip u(x)\hspace{5pt}\mbox{for a.e. }r\in [0,\pi/\sqrt{K}]\mbox{ and $\m_{\sF}$-a.e. }x\in F.\nonumber
\end{align}
Then, $\int_{\sF}\lip u d\m_{\sF}$ has to be a quadratic form on $\lip(F)$ and its form closure is by definition the Cheeger energy $\ChF$. 
\end{proof}
\begin{lemma}\label{kronkorken}
Let $(F,\de_{\sF},\m_{\sF})$ be a metric measure space and $\Con_{N,K}(F)$ satisfies $RCD^*(KN,N+1)$ for $K\geq 0$ and $N\geq 1$.
Then $\de_{\ChF}=\de_{\sF}$. In particular, $\ChF$ is strongly regular.
\end{lemma}
\begin{proof}
We assume $K>0$. The case $K=0$ follows in the same way.
Consider $u(\cdot)=\de_{\sF}(x,\cdot)\in D(\ChF)$. $u$ satisfies $|\nabla u|_w\leq 1$. Hence, $\de_{\ChF}\geq \de_{\sF}$.
The converse inequality is obtained as follows. 
%From Proposition \ref{stlo} follows that $\de_{\ChC}=\de_{\Con_{\sK}}$. 
Consider $u\in D_{loc}(\ChF)\cap C(F)$ with $|\nabla u|_w \leq 1$. 
Let $\epsilon>0$. We choose $\delta>0$ such that $\scriptstyle{\frac{1}{\sin_{\sK}r}}\leq 1+\epsilon$ if $r\in
B_{2\delta}(\pi/2)$. 
Let $u_1\in C_0^{\infty}(\hat{I}_{\sK})$ such that $u_1\leq 1$ and $u_1|_{B_{\delta}(\pi/2)}=1$.
$u_1\otimes u\in D_{loc}(\mathcal{E}^{\sC})$ and 
$$|\nabla (u_1\otimes u)|^2=(u_1')^2u^2+{\textstyle \frac{u_1^2}{\sin_{\sK}^2}}|\nabla u|^2_w\in L^{\infty}(\m_{\sC})$$ 
In particular, it follows that $|\nabla (u_1\otimes u)|=\frac{1}{\sin_{\sK}}|\nabla u|_w\leq 1+\epsilon$ on $B_{\delta}(\pi/2)\times F$.
Since $\Con_{N,K}(F)$ satisfies $RCD^*(KN,N+1)$, this implies that $u_1\otimes u$ admits a Lipschitz representative 
and the Lipschitz constant is locally less than $1+\epsilon$ on some neighborhood of $\pi/2\times F$. This can be seen from standard arguments like in paragraph $1$ of the proof of Proposition \ref{intrinsic}. Hence, for any $x,y\in F$ such that
$\de_{\sF}(x,y)$ is small, we have
\begin{align*}
|u(x)-u(y)|\leq (1+\epsilon)\de_{\Con_{\sK}}((\pi/2,x),(\pi/2,y))\leq (1+\epsilon)\de_{\sF}(x,y).
\end{align*}
It follows that $\de_{\ChF}\leq (1+\epsilon)\de_{\sF}$ locally. Now, $\de_{\sF}$ is geodesic by the remark directly after Definition \ref{kncone}. We can conclude that
$\de_{\ChF}\leq (1+\epsilon)\de_{\sF}$ globally, and since $\epsilon>0$ was arbitrary, we have $\de_{\ChF}\leq \de_{\sF}$.
\end{proof}
\noindent
We summarize the results of this section in the following corollary.
\begin{corollary}\label{eidechse}
Let $(F,\de_{\sF},\m_{\sF})$ be a metric measure space and $K\geq 0$. Assume
\begin{itemize}
 \item[1.] $(F,\de_{\sF},\m_{\sF})$ satisfies $RCD^*(N-1,N)$ for $N\geq 1$ and $\diam_{\sF}\leq \pi$, or
 \item[2.] $\Con_{N,K}(F)$ satisfies $RCD^*(KN,N+1)$ for $N\geq 1$.
\end{itemize}
Then $\ChC=I_{\sK}\times_{\sin_{\sK}}^{\sN}\ChF$, $\de_{\ChF}=\de_{\sF}$ and $\de_{\ChC}=\de_{\Con_{\sN,\sK}(F)}$.\end{corollary}
%\begin{remark}
%In particular, $F$ is infinitesimal Hilbertian if $\Con_{\sN,\sK}(F)$ satisfies $RCD^*(KN,N+1)$, and $\Con_{\sN,\sK}(F)$ is infinitesimal Hilbertian if $F$ satisfies $RCD^*(N-1,N)$.
%\end{remark}
\subsection{Proof of the main theorem}
\begin{proof1}
The Cheeger energy $\ChF$ of $(F,\de_{\sF},\m_{\sF})$ is a strongly local, regular and strongly regular Dirichlet form that satisfies $BE(N-1,N)$ by Theorem \ref{theorembochner}. 
By Proposition \ref{nice} the spectrum of the associated
Laplace operator $L^{\sF}$ is discrete and the first positive eigenvalue of $-L^{\sF}$ satisfies $\lambda_1\geq {N}$.  
By Theorem \ref{intrinsic} and Corollary \ref{maincor} we know that $I_{\sK}\times_{\sin_{\sK}}^{\sN} \ChF=\ChC$ and $\de_{\Con_{\sK}}=\de_{\ChC}$. 
Lemma \ref{books} states that $\Con_{N,K}(F)$ satisfies a volume doubling property and supports a local Poincar\'e inequality.  
By Lemma \ref{stronglyregular} $\ChC$ is strongly regular and closed balls in $\Con_{N,K}(F)$ are compact since $(F,\de_{\sF},\m_{\sF})$ and its Cheeger energy $\ChF$ satisfy the corresponding
properties.
Hence, if $K>0$, we can apply Theorem \ref{mainmaintheorem} and $I_{\sK}\times_{\sin_{\sK}}^{\sN}\ChF$ satisfies $BE(K N,N+1)$. 
\smallskip

Finally, we want to apply the backward direction of Theorem \ref{theorembochner}. 
Results of Sturm from \cite{sturmdirichlet2} (see Remark \ref{feller} and Remark \ref{22poincare}) state a Feller property for the corresponding semigroup $P^{\sC}_t$ of $\Con_{K,N}(F)$ .
Thus, we can apply Theorem 3.15 from \cite{agsbakryemery} that states that in this case any $u\in D(\mathcal{E}^{\sC})$ with $\sqrt{\Gamma^{\sC}(u)}\in L^{\infty}(\m_{\sC})$ 
has a continuous representative. Consequently, any such $u\in D(\mathcal{E}^{\sC})$ is Lipschitz continuous with respect to 
the intrinsic distance of $I_{\sK}\times_{\sin_{\sK}}^{\sN}\ChF$ that coincides with $\de_{\Con_{\sK}}$ by Corollary \ref{eidechse}.
Thus, the regularity Assumption \ref{TheAss} is satisfied and Theorem \ref{theorembochner} yields the condition $RCD^*(NK,N+1)$ if $K>0$.
\smallskip

The case $K=0$ follows from the case $K>0$. 
The rescaled space $\Con_{\sN,K/n^2}(F)$ converges with respect to pointed measured Gromov-Hausdorff convergence to $\Con_{\sN,0}(F)$ if $n\rightarrow \infty$. 
Hence, $\Con_{N,0}(F)$ satisfies $RCD^*(0,N+1)=RCD(0,N+1)$ by the stability property of the condition $RCD$ under measured Gromov-Hausdorff convergence.\qed
\end{proof1}
\begin{proof2}
First, let us consider the case $N\geq 1$. Remark \ref{diameterbound} states that $\diam F\leq \pi$ and $\Con_{\sN,\sK}(F)=I_{\sK}\times^{\sN}_{\sin_{\sK}}F$ in any case when $N\geq 1$.
We need to check the condition $RCD^*(N-1,N)$ for $(F,\de_{\sF},\m_{\sf})$. Corollary \ref{eidechse} implies that $(F,\de_{\sF},\m_{\sF})$ is infinitesimal Hilbertian. 
By Proposition \ref{intrinsic2} and Corollary \ref{maincor2}
the intrinsic distance of $\mathcal{E}^{\sC}=I_{\sK}\times_{\sin_{\sK}}^{\sN}\ChF$ is the $K$-cone distance $\de_{\Con_{\sK}}$ and the Cheeger energy of the $(K,N)$-cone coincides with $\mathcal{E}^{\sC}$. Theorem \ref{theorembochner} implies
the condition $BE(KN,N+1)$ for $I_{\sK}\times_{\sin_{\sK}}^{\sN}\mathcal{E}^{\sF}$. 

\smallskip
%We remark that also in this case $L^{\sF}$ has a discrete spectrum. This follows from Proposition \ref{nice0} since $F$ satisfies the right volume doubling property and 
%supports a Poincar\'e inequality by Lemma \ref{stilldontknow}. Then, one can also check that the eigenfunctions satisfy nice regularity properties. 
%Namely, an eigenfunction $u$ is continuous because of Remark \ref{feller} and $u,\Gamma^{\sF}(u),L^{\sF}u\in L^{\infty}$ because of the Bakry-Ledoux gradient estimate for $\Con_{\sN,\sK}(F)$.
%This can be proven as in the proof of Proposition \ref{regularity}. 

One can check that $C^{\infty}_0(\hat{I}_{\sK})\otimes D(\Gamma_2^{\sF})\subset D(\Gamma_2^{\sC})$ and $1\otimes D^{b,2}_+(L^{\sF})\subset D^{b,2}_+(L^{\sC})$. 
Hence, we can again derive formula (\ref{krebs}) in precisely the same way as in the proof of Theorem \ref{mainmaintheorem} 
for $u_1\otimes u_2 \in C^{\infty}_0(\hat{I}_{\sK})\otimes D(\Gamma_2^{\sF})$ and $1\otimes \phi_2\in 1\otimes D^{b,2}_+(L^{\sF})$.
Now, we can follow the proof of Theorem \ref{maintheorem} and we obtain
%Then, the estimate (\ref{bluuub}) in the proof of Theorem \ref{maintheorem2} can also be derived in a weak form in the current context.
\begin{align}\label{nescafe}
&\int_{\sF}L^{\sF}\phi\Gamma^{\sF}(u)d\m_{\sF}-\int_{\sF}\Gamma^{\sF}(u,L^{\sF}u)\phi d\m_{\sF}\nonumber\\
&\geq (N-1)\int_{\sF}\Gamma^{\sF}(u)\phi d\m_{\sF}+\frac{1}{N}\int_{\sF}\left(L^{\sF}u_2\right)^2\phi d\m_{\sF}
-\frac{1}{(N+1)N}\int_{F}\left(L^{\sF}u_2+NK_{\sF}u_2\right)^2\phi d\m_{\sF}
\end{align}
for any $u\in  D(\Gamma_2^{\sF})$ and any $\phi\in D^{b,2}_+(L^{\sF})$.
We want to deduce $RCD^*(N-1,N)$ for $F$. However, we cannot apply the argument of Theorem \ref{maintheorem2} directly since pointwise estimates for the Bochner inequality do not make sense.
But like in the proof of Theorem \ref{theorembochner} (more precisely, see Proposition 4.7 in \cite{erbarkuwadasturm}), we get a gradient estimate of the following type:
\begin{align}\label{bla}
&\left|\nabla P^{\sF}_t u_2\right|^2+\frac{c(t)}{N}\left(\left|L^{\sF}P^{\sF}_tu_2\right|^2-{\textstyle\frac{1}{(N+1)}}P^{\sF}_t\left(L^{\sF}u_2+NK_{\sF}u_2\right)^2\right)\leq e^{-2Kt}P_t^{\sF}\left|\nabla u_2\right|^2
\end{align}
$\m_{\sF}$-a.e. in $F$ for any $u_2\in D^2(L^{\sF})$. We sketch the argument briefly.
Consider 
\begin{align*}
h(s):=e^{-2(N-1)s}\int_{\sF}P_s^{\sF}\phi|\nabla P_{t-s}^{\sF}u_2|^2d\m_{\sF}.
\end{align*}
One estimates the derivative of $h$ as:
\begin{align*}
&h'(s)=2e^{-2(N-1)s}\!\!\int_{\sF}\!\!\big(-\textstyle{(N-1)}P_s\phi|\nabla P_{t-s}^{\sF}u_2|^2+\textstyle{\frac{1}{2}}L^{\sF}P_s\phi|\nabla P_{t-s}u_2|^2\nonumber\\
&\hspace{70mm}-P_s\phi\Gamma^{\sF}(P_{t-s}u_2,L^{\sF}P_{t-s}u_2)\big) d\m_{\sF}\\
&\geq 2e^{-2(N-1)s}\int_{\sF}P_s\phi\left(\textstyle{\frac{1}{N}}\left(L^{\sF}P_{t-s}u_2\right)^2-\textstyle{\frac{1}{(N+1)N}}\left(L^{\sF}P_{t-s}u_2+NK_{\sF}P_{t-s}u_2\right)^2\right)d\m_{\sF}\\
&\geq 2e^{-2(N-1)s}\int_{\sF}\phi\left(\textstyle{\frac{1}{N}}\left(L^{\sF}P_{t}u_2\right)^2-\textstyle{\frac{1}{(N+1)N}}P_t\left(L^{\sF}u_2+NK_{\sF}u_2\right)^2\right)d\m_{\sF}
\end{align*}
where we used (\ref{nescafe}) in the first and Jensen's inequality in the second inequality. 
Finally, we integrate $h'$ from $0$ to $t$ and the rest of the proof is exactly the same as in Proposition 4.9 in \cite{erbarkuwadasturm}.
\smallskip

We remark that
$F$ satisfies a doubling property and supports a local Poincar\'e inequality and
by Lemma \ref{kronkorken} we have that $\de_{\ChF}=\de_{\sF}$, which implies that
$\ChF$ is strongly local. Thus, by the results of Remark \ref{feller}
the associated semigroup is Feller and has a continuous kernel. Then we proceed as follows.
For $u_2\in D^2(L^{\sF})$ we consider $P_s^{\sF}\left(L^{\sF}u_2+NK_{\sF}u_2\right)=L^{\sF}P_s^{\sF}u_2+NK_{\sF}P_s^{\sF}u_2=:v_2$ and for $x\in F$ we define $v_{2,x}=v_2-v_2(x)$. 
$v_{2,x}$ is continuous on $F$ and $v_{2,x}(x)=0$. We consider $P_t^{\sF}(v_{2,x}^2)$ that is jointly continuous in $z\in F$ and $t\geq 0$. For instance,
this follows since $v_{2,x}^2\in C(F)\cap L^{\infty}(\m_{\sF})$ and since we have a nice upper bound for the heat kernel associated to $\ChF$. Then, to prove that $P_t^{\sF}(v_{2,x}^2)$ is jointly continuous, we can copy the proof of the corresponding result in $\mathbb{R}^n$. It holds that $P_0^{\sF}(v_{2,x}^2)(x)=v_{2,x}(x)=0$.
Hence, for any $\epsilon > 0$ and any $x\in F$ there is $\delta_x>0$ and $\tau_x>0$ such that $|P_t^{\sF}(v_{2,x}^2)(y)|<\epsilon $ for any $y\in B_{\delta_x}(x)$ and $0<t<\tau_x$. 
Since $F$ is compact, there is a finite collection $(x_i)_{i=1}^k$ of points such that $B_{\delta_{x_i}}(x_i)_{i=1,...,k}$ is a covering of $F$. We set $\tau=\min_{i=1,...,k} \tau_i$.
\\
Now we choose $x_i\in F$ with $B_{\delta_i}(x_i)$ and we set $\delta_i=\delta_{x_i}$. 
Consider $$P_s^{\sF}u_2-NK_{\sF}P_s^{\sF}u_2(x_i)-L^{\sF}P_s^{\sF}u_2(x_i)=:\bar{v}_{2,x_i}\in D^2(L^{\sF})$$ and insert it in (\ref{bla}) for $t<\tau$.
\begin{align*}
&\left|\nabla P^{\sF}_t \bar{v}_{2,x_i}\right|^2+c(t)\frac{1}{N}\left(L^{\sF}P^{\sF}_t\bar{v}_{2,x_i}\right)^2\nonumber\\
&\hspace{30mm}-\frac{c(t)}{N(N+1)}\underbrace{P^{\sF}_t\left(L^{\sF}\bar{v}_{2,x_i}+NK_{\sF}\bar{v}_{2,x_i}\right)^2}_{(*)}\leq e^{-2Kt}P_t^{\sF}\left|\nabla \bar{v}_{2,x_i}\right|^2.
\end{align*}
We can see that
\begin{align*}
(*)
&=P^{\sF}_t\left(L^{\sF}P_s^{\sF}u_2+NK_{\sF}P_s^{\sF}u_2-NK_{\sF}P_s^{\sF}u_2(x_i)-L^{\sF}P_s^{\sF}u_2(x_i)\right)^2\\
&=P^{\sF}_t\underbrace{\left(v_{2}-v_{2}(x_i)\right)^2}_{(v_{2,x_i})^2}
\end{align*}
For any $y\in B_{\delta_i}(x_i)$ we get $|(*)(y)|=|P^{\sF}_t\left(\tilde{v}_{2,x_i}^2\right)(y)|<\epsilon$.
From that and since $\bar{v}_{2,x_i}$ differs form $P_s^{\sF}u_2$ only by a constant, we get for any $0<t<\tau$ and $\m_{\sF}$-a.e. $y\in B_{\delta_i}(x_i)$
\begin{align*}
\left|\nabla P^{\sF}_t P_s^{\sF}u_2\right|^2(y)+\frac{c(t)}{N}\left(\left|L^{\sF}P^{\sF}_t P_s^{\sF}u_2\right|^2(y)-\frac{1}{N+1}\epsilon\right)\leq e^{-2Kt}P_t^{\sF}\left|\nabla P_s^{\sF}u_2\right|^2(y).
\end{align*}
The last inequality does not depend on $x_i$ anymore
and since $\epsilon>0$ is arbitrary, we obtain 
\begin{align*}
\left|\nabla P^{\sF}_t P_s^{\sF}u_2\right|^2+\frac{c(t)}{N}\left|L^{\sF}P^{\sF}_t P_s^{\sF}u_2\right|^2\leq e^{-2Kt}P_t^{\sF}\left|\nabla P_s^{\sF}u_2\right|^2
\end{align*}
$\mbox{ for }0<t<\tau\mbox{ and }\m_{\sF}\mbox{-a.e.}$ for $u_2\in D^2(L^{\sF})$. Then we can also let $s$ go to $0$
\begin{align*}
\left|\nabla P^{\sF}_t u_2\right|^2+\frac{c(t)}{N}\left|L^{\sF}P^{\sF}_t u_2\right|^2\leq e^{-2Kt}P_t^{\sF}\left|\nabla u_2\right|^2 \mbox{ for }0<t<\tau
\end{align*}
and finally, we can follow the proof of Theorem 4.8 in \cite{erbarkuwadasturm} to obtain the condition $BE(N-1,N)$. 
Now, similar like in the previous theorem, this implies $RCD^*(N-1,N)$ for $(F,\de_{\sF},\m_{\sF})$. 
We only need to check the Assumption \ref{TheAss}. 
The condition $RCD^*(KN,N+1)$ for $\Con_{N,K}(F)$ implies that $u\in D(\ChC)$ with $\Gamma^C(u)\in L^{\infty}(\m_{\sC})$ admits a Lipschitz representative, and
Theorem \ref{maincor} says that $I_{\sK}\times^{\sN}_{\sin_{\sK}}\ChF=\ChC$. These two statements imply that also $v\in D(\ChF)$ with $\Gamma^{\sF}(v)\in L^{\infty}(\m_{\sF})$ 
admits a Lipschitz representative with respect to $\de_{\sF}$.
\smallskip

For the case $N\in [0,1)$ we argue by contradiction. First, $F$ has to be discrete. Otherwise, we would find a geodesic $\gamma$ in $F$, and 
consequently the cone over $\mbox{Im}(\gamma)$ would be a $2$-dimensional subset of $\Con_{\sN,\sK}(F)$. This contradicts the condition $RCD^*(KN,N+1)$ for $\Con_{\sN,\sK}(F)$ that implies 
that the Hausdorff dimension of $\Con_{\sN,\sK}(F)$ cannot be bigger than $N+1<2$. 

The second observation is that any pair of points $x,y\in F$ satisfies $\de_{\sF}(x,y)=\pi$. Otherwise, there are points $x,y\in F$ with $\de_{\sF}(x,y)<\pi$. 
It follows that there is no continuous curve between $(1,x)$ and $(1,y)$ in $\Con_{\sN,\sK}(F)$ such that its length is $\epsilon$-close
to $\de_{\Con_{\sK}}((1,x),(1,y))$ for $\epsilon>0$ sufficienty small. A continuous curve that connects $(1,x)$ and $(1,y)$ consists of the segments 
that connect each of this points with the nearest origin and its length is $\de_{\Con_{\sK}}(o,(1,x))+\de_{\Con_{\sK}}(o,(1,y))>\de_{\Con_{\sK}}((1,x),(1,y))$. 
But $\Con_{\sN,\sK}(F)$ satisfies a curvature-dimension condition. Therefore, it has to be an intrinsic metric space what contradicts the previous observation. 

We observe that $F$ can have at most two points. Otherwise we will find
an optimal transport between absolutely continuous measures in $\Con_{\sN,\sK}(F)$ that is essentially branching what contradicts the $RCD^*$-condition. For instance, 
assume there are three points. The geodesics between $(s,x)$, $(t,y)$ and $(r,z)$ for $s,t,r\leq 1$ consist exactly of segments that connect the origin. Hence, 
one can consider an absolutely continuous measure that is concentrated on one segment and transported to an absolutely continuous measure that is concentrated equally on the two other segments. 
Finally, in the case where $F$ is just one point we see that $I_{\sK}\times_{\sin_{\sK}}^NF=(I_{\sK},\sin^N_{\sK})$. Otherwise, if $F$ has two points with distance $\pi$, 
$N$ has to be $0$ and we see that $\Con_{0,\sK}(F)=\frac{1}{\sqrt{K}}\mathbb{S}^1$. \qed
\end{proof2}

\subsection{Proof of the maximal diameter theorem}
As an application of the previous results we can prove a spherical splitting theorem for $RCD^*$-spaces by application of the
splitting theorem for $RCD^*(0,N)$ spaces. For Riemannian manifolds with non-negative Ricci curvature bounds 
this was proven by Cheeger and Gromoll in \cite{cheegergromoll}. Recently, N. Gigli proved the result for general $RCD^*(0,N)$-spaces:
\begin{theorem}[N. Gigli, \cite{giglisplittingshort}]
Let $(X,\de_{\sX},\m_{\sX})$ be a metric measure space that satisfies $RCD(0,N+1)$ for $N\geq 0$ and contains a geodesic line. 
Then $(X,\de_{\sX},\m_{\sX})$ is isomorphic to the euclidean product of the Euclidean line $(\mathbb{R},\de_{Eucl}, \mathcal{L}^1)$ and another metric measure space $(X',\de_{\sX'},\m_{\sX'})$ such that
\begin{itemize}
\item[(1)] $(X',\de_{\sX'},\m_{\sX'})$ is $RCD(0,N)$ if $N\geq 1$,
\item[(2)] $X'$ is just a point if $N\in [0,1)$.
\end{itemize}
Here "isomorphic" means that there is a measure preserving isometry.
\end{theorem}
\begin{proof4}
$(F,\de_{\sF},\m_{\sF})$ is compact with $\diam F\leq \pi$. Therefore, $\Con_0^{N+1}(F)=[0,\infty)\times_r^{N+1}F$ and by Theorem \ref{oneone} $\Con_0^{N+1}(F)$ satisfies 
$RCD^*(0,N+2)=RCD(0,N+2)$. Since $\de_{\sF}(x,y)=\pi$, there is a geodesic line in $\Con_0^{N+1}(F)$ by definition of $\de_{\Con_{0}}$.
Thus, by the first part of Gigli's Theorem, $\Con_{N+2,0}(F)=:X$ splits into $X=\mathbb{R}\times X'$ where $X'=(X',\de_{\sX'},\m_{\sX'})$ denotes a metric measure space
that satisfies $RCD^*(0,N+1)$.
One can easily see that $X'$ is a metric cone over $F'=F\cap X'$, that $F'$ is a geodesic space and 
that $F'$ embeds geodesically in $F$.
\smallskip

Consider $(1,f),(1,g)$ in $\left\{1\right\}\times F$. 
We find $r,s>0$, $i,j\in [-1,1]$ and $f',g'\in F'$ such that
\begin{align*}
\de_{\sX}((1,f),(1,g))^2=2-2\cos \de_{\sF}(f,g)=r^2+s^2-2rs\cos\de_{\sF'}(f',g')+|i-j|^2.
\end{align*}
Because the metric on $X$ is precisely given by the metric $l^2$-product of $|\cdot-\cdot|$ and $\de_{\sX'}$, the Pythagorean theorem holds. Hence $i^2+r^2=1$. It follows that
\begin{align*}
\cos\de_{\sF}(f,g)=ij+(1-i^2)^{\frac{1}{2}}(1-j^2)^{\frac{1}{2}}\cos\de_{\sF'}(f',g').
\end{align*}
There are unique numbers $\theta,\phi\in [0,\pi]$ such that $i=\cos\theta$ and $j=\cos\phi$.
Thus, there is an isometry between $(F,\de_{\sF})$ and the metric $1$-cone with respect to $F'$. In particular, $F$ is also a topological suspension in the sense of Ohta's topological splitting result in \cite{ohtpro}
and the measure has the form $d\m_{\sF}=\sin^{\sN}d\m_{\sF'}$
for some Borel measure $\m_{\sF'}$ on $F'$. Hence, $F$ is a $(K,N)$-cone over $(F',\de_{\sF'},\m_{\sF'})$ in the sense of Definition \ref{kncone}. 
Finally, Theorem \ref{twotwo} yields the result.\qed
\end{proof4}

\begin{proof3}
First, we consider $x_0,y_0\in F$ with $\de_{\sF}(x_0,y_0)=\pi$.
The maximal diameter theorem implies that $F$ is a spherical suspension with respect to some metric measure spaces $F'_0$ that satisfies $RCD(N-2,N-1)$ 
where the pair $(x_0,y_0)$ corresponds to the two origins of $I_{\sK}\times^{\sN}_{\sin_{\sK}}F'$. 
If we consider another pair $(x_1,y_1)$, we obtain another suspension structure.
Hence, we find a loop $s:[0,2\pi]/_{\left\{0\sim 2\pi\right\}}\rightarrow S$ in $F$ that is geodesic for small distances and intersects with $F'$ at $x_1$ since $\de_{\sF}(x_0,x_1)=\frac{\pi}{2}$. 
But this also implies that $\de_{\sF}(y_0,y_1)=\frac{\pi}{2}$ and $y_1\in F'$. Since $F'$ embeds geodesically into $F$, we have $\de_{\sF'}(x_1,y_1)=\pi$
\smallskip

Then, we also obtain for any other pair $x_i,y_i\in F$ for $i\geq 1$ that $x_i,y_i\in F'$, $\de_{\sF'}(x_i,y_i)=\pi$ and $\de_{\sF'}(x_i,x_j)=\frac{\pi}{2}$ for $i\neq j$.
Hence, we can proceed by induction and the second part of the maximal diameter theorem tells us that that after finitely many steps no further decomposition is possible and $F=\mathbb{S}^k$ for some $k\in \mathbb{N}$.
But then, $n-1=N=k$.\qed
\end{proof3}
\begin{ak}
The author wants to thank Karl-Theodor Sturm for proposing this interesting problem, stimulating discussions during the work on this article and important comments and remarks. 
I also want to thank the anonymous reviewer for carefully reading the article and for helpful comments.
\end{ak}
\bibliography{new}

\begin{thebibliography}{10}

\bibitem{albi0}
Stephanie~B. Alexander and Richard~L. Bishop.
\newblock Warped products of {H}adamard spaces.
\newblock {\em Manuscripta Math.}, 96(4):487--505, 1998.

\bibitem{agmr}
Luigi Ambrosio, Nicola Gigli, Andrea Mondino, and Tapio Rajala.
\newblock Riemannian ricci curvature lower bounds in metric measure spaces with
  $σ$-finite measure.
\newblock {\em http://arxiv.org/abs/1207.4924}.

\bibitem{agsbakryemery}
Luigi Ambrosio, Nicola Gigli, and Giuseppe Savar{\'e}.
\newblock Bakry-{E}mery curvature-dimension condition and {R}iemannian {R}icci
  curvature bounds.
\newblock {\em http://arxiv.org/abs/1209.5786}.

\bibitem{agsheat}
Luigi Ambrosio, Nicola Gigli, and Giuseppe Savar{\'e}.
\newblock Calculus and heat flow in metric measure spaces and applications to
  spaces with {R}icci bounds from below.
\newblock {\em http://arxiv.org/abs/1106.2090}.

\bibitem{agslipschitz}
Luigi Ambrosio, Nicola Gigli, and Giuseppe Savar{\'e}.
\newblock Density of {L}ipschitz functions and equivalence of weak gradients in
  metric measure spaces.
\newblock {\em http://arxiv.org/abs/1111.3730}.

\bibitem{agsriemannian}
Luigi Ambrosio, Nicola Gigli, and Giuseppe Savar{\'e}.
\newblock Metric spaces with {R}iemannian {R}icci curvature bounded from below.
\newblock {\em http://arxiv.org/abs/1109.0222}.

\bibitem{agsgradient}
Luigi Ambrosio, Nicola Gigli, and Giuseppe Savar{\'e}.
\newblock {\em Gradient flows in metric spaces and in the space of probability
  measures}.
\newblock Lectures in Mathematics ETH Z\"urich. Birkh\"auser Verlag, Basel,
  second edition, 2008.

\bibitem{amsbochner}
Luigi Ambrosio, Andrea Mondino, and Giuseppe Savar\'e.
\newblock Nonlinear diffusion equations and curvature conditions in metric
  measure spaces.
\newblock {\em in preparation}.

\bibitem{anderson}
Michael~T. Anderson.
\newblock Metrics of positive {R}icci curvature with large diameter.
\newblock {\em Manuscripta Math.}, 68(4):405--415, 1990.

\bibitem{bast}
Kathrin Bacher and Karl-Theodor Sturm.
\newblock Localization and tensorization properties of the curvature-dimension
  condition for metric measure spaces.
\newblock {\em J. Funct. Anal.}, 259(1):28--56, 2010.

\bibitem{bastco}
Kathrin Bacher and Karl-Theodor Sturm.
\newblock Ricci bounds for euclidean and spherical cones.
\newblock {\em Singular Phenomena and Scaling in Mathematical Models}, pages
  3--23, 2014.

\bibitem{bakryledouxliyau}
Dominique Bakry and Michel Ledoux.
\newblock A logarithmic {S}obolev form of the {L}i-{Y}au parabolic inequality.
\newblock {\em Rev. Mat. Iberoam.}, 22(2):683--702, 2006.

\bibitem{berard}
Pierre~H. B{\'e}rard.
\newblock {\em Spectral geometry: direct and inverse problems}, volume 1207 of
  {\em Lecture Notes in Mathematics}.
\newblock Springer-Verlag, Berlin, 1986.
\newblock With appendixes by G{\'e}rard Besson, and by B{\'e}rard and Marcel
  Berger.

\bibitem{bjoern}
Anders Bj{\"o}rn and Jana Bj{\"o}rn.
\newblock {\em Nonlinear potential theory on metric spaces}, volume~17 of {\em
  EMS Tracts in Mathematics}.
\newblock European Mathematical Society (EMS), Z\"urich, 2011.

\bibitem{hirsch}
Nicolas Bouleau and Francis Hirsch.
\newblock {\em Dirichlet forms and analysis on {W}iener space}, volume~14 of
  {\em de Gruyter Studies in Mathematics}.
\newblock Walter de Gruyter \& Co., Berlin, 1991.

\bibitem{bbi}
Dmitri Burago, Yuri Burago, and Sergei Ivanov.
\newblock {\em A course in metric geometry}, volume~33 of {\em Graduate Studies
  in Mathematics}.
\newblock American Mathematical Society, Providence, RI, 2001.

\bibitem{cavallettisturm}
Fabio Cavalletti and Karl-Theodor Sturm.
\newblock Local curvature-dimension condition implies measure-contraction
  property.
\newblock {\em J. Funct. Anal.}, 262(12):5110--5127, 2012.

\bibitem{cheegerlipschitz}
Jeff Cheeger.
\newblock Differentiability of {L}ipschitz functions on metric measure spaces.
\newblock {\em Geom. Funct. Anal.}, 9(3):428--517, 1999.

\bibitem{almostrigidity}
Jeff Cheeger and Tobias~H. Colding.
\newblock Lower bounds on {R}icci curvature and the almost rigidity of warped
  products.
\newblock {\em Ann. of Math. (2)}, 144(1):189--237, 1996.

\bibitem{cheegergromoll}
Jeff Cheeger and Detlef Gromoll.
\newblock The splitting theorem for manifolds of nonnegative {R}icci curvature.
\newblock {\em J. Differential Geometry}, 6:119--128, 1971/72.

\bibitem{erbarkuwadasturm}
Matthias Erbar, Kazumasa Kuwada, and Karl-Theodor Sturm.
\newblock On the equivalence of the {E}ntropic curvature-dimension condition
  and {B}ochner's inequality on metric measure spaces.
\newblock {\em http://arxiv.org/abs/1303.4382}.

\bibitem{fukushimaoshima}
Masatoshi Fukushima and Y{\=o}ichi {\=O}shima.
\newblock On the skew product of symmetric diffusion processes.
\newblock {\em Forum Math.}, 1(2):103--142, 1989.

\bibitem{fukushima}
Masatoshi Fukushima, Yoichi Oshima, and Masayoshi Takeda.
\newblock {\em Dirichlet forms and symmetric {M}arkov processes}, volume~19 of
  {\em de Gruyter Studies in Mathematics}.
\newblock Walter de Gruyter \& Co., Berlin, extended edition, 2011.

\bibitem{giglistructure}
Nicola Gigli.
\newblock On the differential structure of metric measure spaces and
  applications.
\newblock {\em http://arxiv.org/abs/1205.6622}.

\bibitem{giglisplittingshort}
Nicola Gigli.
\newblock An overview on the proof of the splitting theorem in a non-smooth
  context.
\newblock {\em http://arxiv.org/abs/1305.4854}.

\bibitem{giglikuwadaohta}
Nicola Gigli, Kazumasa Kuwada, and Shin-Ichi Ohta.
\newblock Heat flow on {A}lexandrov spaces.
\newblock {\em Comm. Pure Appl. Math.}, 66(3):307--331, 2013.

\bibitem{giglirajalasturm}
Nicola Gigli, Tapio Rajala, and Karl-Theodor Sturm.
\newblock Optimal maps and exponentiation on finite dimensional spaces with
  {R}icci curvature bounded from below.
\newblock {\em http://arxiv.org/abs/1305.4849}.

\bibitem{grigoryanheat}
Alexander Grigor'yan.
\newblock {\em Heat kernel and analysis on manifolds}, volume~47 of {\em AMS/IP
  Studies in Advanced Mathematics}.
\newblock American Mathematical Society, Providence, RI, 2009.

\bibitem{grigoryan}
Alexander Grigor'yan and Jun Masamune.
\newblock Parabolicity and stochastic completeness of manifolds in terms of the
  {G}reen formula.
\newblock {\em J. Math. Pures Appl. (9)}, 100(5):607--632, 2013.

\bibitem{haj}
Piotr Haj{\l}asz.
\newblock Sobolev spaces on an arbitrary metric space.
\newblock {\em Potential Anal.}, 5(4):403--415, 1996.

\bibitem{koskela}
Piotr Haj{\l}asz and Pekka Koskela.
\newblock Sobolev met {P}oincar\'e.
\newblock {\em Mem. Amer. Math. Soc.}, 145(688):x+101, 2000.

\bibitem{ketterer}
Christian Ketterer.
\newblock Ricci curvature bounds for warped products.
\newblock {\em J. Funct. Anal.}, 265(2):266--299, 2013.

\bibitem{koskelazhou}
Pekka Koskela and Yuan Zhou.
\newblock Geometry and analysis of {D}irichlet forms.
\newblock {\em Adv. Math.}, 231(5):2755--2801, 2012.

\bibitem{lottvillani}
John Lott and C{\'e}dric Villani.
\newblock Ricci curvature for metric-measure spaces via optimal transport.
\newblock {\em Ann. of Math. (2)}, 169(3):903--991, 2009.

\bibitem{maroeckner}
Zhi~Ming Ma and Michael R{\"o}ckner.
\newblock {\em Introduction to the theory of (nonsymmetric) {D}irichlet forms}.
\newblock Universitext. Springer-Verlag, Berlin, 1992.

\bibitem{ohtmea}
Shin-ichi Ohta.
\newblock On the measure contraction property of metric measure spaces.
\newblock {\em Comment. Math. Helv.}, 82(4):805--828, 2007.

\bibitem{ohtpro}
Shin-Ichi Ohta.
\newblock Products, cones, and suspensions of spaces with the measure
  contraction property.
\newblock {\em J. Lond. Math. Soc. (2)}, 76(1):225--236, 2007.

\bibitem{okura}
Hiroyuki {\^O}kura.
\newblock A new approach to the skew product of symmetric {M}arkov processes.
\newblock {\em Mem. Fac. Engrg. Design Kyoto Inst. Tech. Ser. Sci. Tech.},
  46:1--12, 1997.

\bibitem{on}
Barrett O'Neill.
\newblock {\em Semi-{R}iemannian geometry (With applications to relativity)},
  volume 103 of {\em Pure and Applied Mathematics}.
\newblock Academic Press Inc. [Harcourt Brace Jovanovich Publishers], New York,
  1983.

\bibitem{rajala2}
Tapio Rajala.
\newblock Interpolated measures with bounded density in metric spaces
  satisfying the curvature-dimension conditions of {S}turm.
\newblock {\em J. Funct. Anal.}, 263(4):896--924, 2012.

\bibitem{reedsimon}
Michael Reed and Barry Simon.
\newblock {\em Methods of modern mathematical physics. {II}. {F}ourier
  analysis, self-adjointness}.
\newblock Academic Press [Harcourt Brace Jovanovich Publishers], New York,
  1975.

\bibitem{shan}
Nageswari Shanmugalingam.
\newblock Newtonian spaces: an extension of {S}obolev spaces to metric measure
  spaces.
\newblock {\em Rev. Mat. Iberoamericana}, 16(2):243--279, 2000.

\bibitem{sturmdirichlet1}
Karl-Theodor Sturm.
\newblock Analysis on local {D}irichlet spaces. {I}. {R}ecurrence,
  conservativeness and {$L^p$}-{L}iouville properties.
\newblock {\em J. Reine Angew. Math.}, 456:173--196, 1994.

\bibitem{sturmdirichlet2}
Karl-Theodor Sturm.
\newblock Analysis on local {D}irichlet spaces. {II}. {U}pper {G}aussian
  estimates for the fundamental solutions of parabolic equations.
\newblock {\em Osaka J. Math.}, 32(2):275--312, 1995.

\bibitem{sturmdirichlet3}
Karl-Theodor Sturm.
\newblock Analysis on local {D}irichlet spaces. {III}. {T}he parabolic
  {H}arnack inequality.
\newblock {\em J. Math. Pures Appl. (9)}, 75(3):273--297, 1996.

\bibitem{stugeo1}
Karl-Theodor Sturm.
\newblock On the geometry of metric measure spaces. {I}.
\newblock {\em Acta Math.}, 196(1):65--131, 2006.

\bibitem{stugeo2}
Karl-Theodor Sturm.
\newblock On the geometry of metric measure spaces. {II}.
\newblock {\em Acta Math.}, 196(1):133--177, 2006.

\bibitem{viltot}
C{\'e}dric Villani.
\newblock {\em Optimal transport, Old and new}, volume 338 of {\em Grundlehren
  der Mathematischen Wissenschaften [Fundamental Principles of Mathematical
  Sciences]}.
\newblock Springer-Verlag, Berlin, 2009.

\bibitem{renessepoincare}
Max-K. von Renesse.
\newblock On local {P}oincar\'e via transportation.
\newblock {\em Math. Z.}, 259(1):21--31, 2008.

\end{thebibliography}

\end{document}